\documentclass{amsart}

\usepackage[alphabetic,initials,nobysame]{amsrefs}
\usepackage{amssymb,mathtools,mathrsfs,comment,todonotes,color}
\usepackage{hyperref} 
\usepackage[shortlabels]{enumitem}
\usepackage{caption}
\usepackage[percent]{overpic}
\usepackage{esint} 

\usepackage{pgfplots}
\pgfplotsset{compat=1.13}
\usetikzlibrary{arrows}

\BibSpec{collection.article}{%
	+{}  {\PrintAuthors}				{author}
	+{,} { \textit}                     {title}
	+{.} { }                            {part}
	+{:} { \textit}                     {subtitle}
	+{,} { \PrintContributions}         {contribution}
	+{,} { \PrintConference}            {conference}
	+{}  {\PrintBook}                   {book}
	+{,} { }                            {booktitle}
	+{,} { }							{series}
	+{,} { }							{publisher}
	+{,} { }							{address}
	+{,} { \PrintDateB}                 {date}
	+{,} { pp.~}                        {pages}
	+{,} { }                            {status}
	+{,} { \PrintDOI}                   {doi}
	+{,} { available at \eprint}        {eprint}
	+{}  { \parenthesize}               {language}
	+{}  { \PrintTranslation}           {translation}
	+{;} { \PrintReprint}               {reprint}
	+{.} { }                            {note}
	+{.} {}                             {transition}
	+{}  {\SentenceSpace \PrintReviews} {review}
}

\theoremstyle{plain}
\newtheorem{theorem}{Theorem}
\newtheorem{lemma}[theorem]{Lemma}

\newtheorem{prop}[theorem]{Proposition}
\newtheorem{corollary}[theorem]{Corollary}

\newtheorem{problem}[theorem]{Problem}

\theoremstyle{definition}
\newtheorem{definition}[theorem]{Definition}

\newtheorem{innercustomthm}{Conclusion}
\newenvironment{conclusion}[1]
  {\renewcommand\theinnercustomthm{#1}\innercustomthm}
  {\endinnercustomthm}

\theoremstyle{remark}
\newtheorem{remark}[theorem]{Remark}
\newtheorem{question}[theorem]{Question}

\numberwithin{equation}{section}
\numberwithin{theorem}{section}
\numberwithin{conjecture}{section}

\newcommand{\Conv}{\mathop{\scalebox{1.5}{\raisebox{0.0ex}{$\ast$\!}}}}%
\newcommand{\x}{\scalebox{1.2}{$\chi$} } 
\newcommand{\br}{\overline}
\newcommand{\R}{\mathbb R}

\newcommand{\C}{\mathbb C}

\newcommand{\Z}{\mathbb Z}
\newcommand{\N}{\mathbb N}
\newcommand{\Q}{\mathbb Q}

\newcommand{\h}{\mathscr H}

\DeclareMathOperator{\dist}{{\mathrm{dist}}}
\DeclareMathOperator{\diam}{{\mathrm{diam}}}

\DeclareMathOperator{\inter}{{\mathrm{int}}}

\DeclareMathOperator{\md}{\mathrm{Mod}}
\DeclareMathOperator{\cp}{\mathrm{Cap}}

\DeclareMathOperator{\loc}{\mathrm{loc}}

\DeclareMathOperator{\NED}{\mathit{NED}}
\DeclareMathOperator{\CNED}{\mathit{CNED}}
\DeclareMathOperator{\QCH}{\mathit{QCH}}
\DeclareMathOperator{\*ned}{\Conv\NED}
\DeclareMathOperator{\SH}{\mathit{SH}}

\begin{document}
\title{CNED sets: countably negligible for extremal distances}

\author{Dimitrios Ntalampekos}
\address{Mathematics Department, Stony Brook University, Stony Brook, NY 11794, USA.}

\thanks{The author is partially supported by NSF Grant DMS-2000096.}
\email{dimitrios.ntalampekos@stonybrook.edu}

\date{\today}

\setcounter{tocdepth}{1}

\begin{abstract}
The author has recently introduced the class of $\CNED$ sets in Euclidean space, generalizing the classical notion of $\NED$ sets, and shown that they are quasiconformally removable. A set $E$ is $\CNED$ if the conformal modulus of a curve family is not affected when one restricts to the subfamily intersecting $E$ at countably many points. We prove that several classes of sets that were known to be removable are also $\CNED$, including sets of $\sigma$-finite Hausdorff $(n-1)$-measure and boundaries of domains with $n$-integrable quasihyperbolic distance. Thus, this work puts in common framework many known results on the problem  of quasiconformal removability and suggests that the $\CNED$ condition should also be necessary for  removability. We give a new necessary and sufficient criterion for closed sets to be (\textit{C})\textit{NED}. Applying this criterion, we show that countable unions of closed (\textit{C})\textit{NED} sets are (\textit{C})\textit{NED}. Therefore we enlarge significantly the known classes of quasiconformally removable sets. 
\end{abstract}

\maketitle

\tableofcontents

\section{Introduction}

\subsection{Definitions}

Before presenting our results, we first discuss some background. We assume throughout that $n\geq 2$. For an open set $U\subset \R^n$ and two continua  {$F_1,F_2\subset  U$} the family of curves joining $F_1$ and $F_2$ inside $U$ is denoted by $\Gamma(F_1,F_2;U)$. For a set $E\subset \R^n$ we denote by $\mathcal F_0(E)$ the family of curves in $\R^n$ that do not intersect $E$, {except possibly at the endpoints}, and by $\mathcal F_{\sigma}(E)$ the family of curves in $\R^n$ that intersect $E$  at countably many points, not counting multiplicity. 

A set $E\subset \R^n$ is \textit{negligible for extremal distances} if for every pair of non-empty, disjoint continua $F_1,F_2\subset \R^n$ we have
\begin{align*}
\md_n \Gamma(F_1,F_2;\R^n) =\md_n (\Gamma(F_1,F_2;\R^n)\cap \mathcal F_0(E)).
\end{align*}
In this case, we write $E\in \NED$; note that we suppress the dimension $n$ in this notation. If, instead, there exists a uniform constant $M\geq 1$ such that 
\begin{align*}
\md_n \Gamma(F_1,F_2;\R^n)\leq M\cdot\md_n (\Gamma(F_1,F_2;\R^n)\cap \mathcal F_0(E)),
\end{align*}
we say that $E$ is \textit{weakly} $\NED$ and we write $E\in \NED^w$. We remark that we do not require $E$ to be closed. 
For closed sets, the classes $\NED$ and $\NED^w$ agree \cite{AseevSycev:removable}. 

The author in \cite{Ntalampekos:metric_definition_qc} introduced the class of $\CNED$ sets, that is, \textit{countably negligible for extremal distances}. We say that a set $E\subset \R^n$ is of class $\CNED$ if
\begin{align*}
\md_n \Gamma(F_1,F_2;\R^n) =\md_n (\Gamma(F_1,F_2;\R^n)\cap \mathcal F_{\sigma}(E))
\end{align*}
for every pair of non-empty, disjoint continua $F_1,F_2\subset \R^n$. In this case we write $E\in \CNED$. As above, we also define the class $\CNED^w$ in the obvious manner. Again, $E$ need not be closed  and the dimension $n$ is suppressed in this notation. For closed sets we show in Theorem \ref{theorem:criterion_compact} that the classes $\CNED$ and $\CNED^w$ agree. The monotonicity of modulus implies that $\NED\subset \CNED \subset \CNED^w.$

\subsection{Properties of negligible sets}

Closed $\NED$ sets have been studied extensively in the plane by Ahlfors and Beurling in \cite{AhlforsBeurling:Nullsets}, where they proved that these sets coincide with the closed sets $E\subset \C$ that are \textit{removable for conformal embeddings} or else \textit{$S$-removable}; that is, every conformal embedding of $\C\setminus E$ into $\C$ is the restriction of a M\"obius transformation. See also Pesin's work \cite{Pesin:removable}. Equivalently, we may replace conformal with quasiconformal maps in this definition. See \cite{Younsi:removablesurvey} for a survey. V\"ais\"al\"a initiated the study of closed $\NED$ sets in higher dimensions \cite{Vaisala:null}, proving that closed sets of Hausdorff $(n-1)$-measure zero are of class $\NED$. The result of Ahlfors--Beurling was partially generalized in higher dimensions by Aseev--Sy\v{c}ev \cite{AseevSycev:removable} and Vodopyanov--Goldshtein \cite{VodopjanovGoldstein:removable}, who proved that if a closed set $E\subset \R^n$, $n\geq 3$, is of class $\NED$, then it is removable for quasiconformal embeddings. The converse is not known in dimensions $n\geq 3$. Finally, a characterization of closed $\NED$ sets in $\R^n$ was provided by Vodopyanov--Goldshtein \cite{VodopjanovGoldstein:removable}, who proved that closed $\NED$ sets coincide with sets that are removable for the Sobolev space $W^{1,n}$. We direct the reader to  the introduction of \cite{Aseev:nedhyperplane} for a survey of the known results. $\NED$ sets are closely related to quasiextremal distance ($\mathit{QED}$) exceptional sets, introduced by Gehring--Martio \cite{GehringMartio:QEDdomains}. As remarked, here we will work with $\NED$ sets that are not necessarily closed. 

The relation between $\CNED$ sets and quasiconformal maps was unveiled in \cite{Ntalampekos:metric_definition_qc}. We state a special case of the main theorem.

\begin{theorem}\label{theorem:removable}
Let $E\subset \R^n$ be a closed $\CNED$ set. Then every homeomorphism of $\R^n$ that is quasiconformal on $\R^n\setminus E$ is quasiconformal on $\R^n$. 
\end{theorem}

In fact, in \cite{Ntalampekos:metric_definition_qc}*{Theorem 1.2} the set $E$ is not assumed to be closed, in which case the quasiconformality of $f$ in the set $\R^n\setminus E$ has to be interpreted appropriately, using the metric definition or a variant.

Closed sets $E$ satisfying the conclusion of Theorem \ref{theorem:removable} are called \textit{removable for quasiconformal homeomorphisms} or else \textit{$\QCH$-removable}. The difference to $S$-removable sets that we discuss above is that here we study global homeomorphisms of $\R^n$, instead of topological embeddings of $\R^n\setminus E$. Moreover, note that if a planar set is $S$-removable, then it is $\QCH$-removable. Although we have satisfactory characterizations of the former sets by Ahlfors--Beurling \cite{AhlforsBeurling:Nullsets} in dimension $2$, it is a notoriously difficult problem to characterize $\QCH$-removable sets even in dimension $2$. The current work suggests that $\CNED$ sets could provide a characterization. 

There are many open problems related to $\QCH$-removable sets, one of which is the problem of \textit{local removability} \cite{Bishop:flexiblecurves}*{Question 4}, \cite{Ntalampekos:gasket}*{Question 2}: if a closed set $E$ is $\QCH$-removable, is it true that every topological embedding of an open set $\Omega\subsetneq \R^n$ into $\R^n$ that is quasiconformal in $\Omega\setminus E$, is quasiconformal in $\Omega$? For $\CNED$ sets an affirmative answer is provided by \cite{Ntalampekos:metric_definition_qc}*{Theorem 1.2}.

Another open problem is whether the union of two $\QCH$-removable closed sets is removable \cite{JonesSmirnov:removability}. While for disjoint sets the answer is affirmative, in general, for intersecting sets this is known only in the cases of totally disconnected sets and quasicircles \cite{Younsi:RemovabilityRigidityKoebe}*{Theorem 4}. We prove here that countable unions of closed $\NED$ and $\CNED$ sets are $\NED$ and $\CNED$, respectively. 

\begin{theorem}\label{theorem:unions}
Let $E_i$, $i\in \N$, be a countable collection of closed subsets of $\R^n$.
\begin{enumerate}[\upshape(i)]\smallskip
	\item $\displaystyle{\textrm{If $E_i$ is $\NED$ for each $i\in \N$, then $\bigcup_{i\in \N} E_i$ is $\NED.$}}$\label{item:union:i}\smallskip
	\item $\displaystyle{\textrm{If $E_i$ is $\CNED$ for each $i\in \N$, then $\bigcup_{i\in \N} E_i$ is $\CNED.$}}$\label{item:union:ii}
\end{enumerate}
\end{theorem}

In the case that a countable union of closed $\NED$ sets is \textit{closed}, this result follows from the Baire category theorem \cite{Younsi:removablesurvey}*{Section 4}. The case of \textit{non-closed} unions is significantly more complicated, since they could even be dense.  The proof is given in Section \ref{section:union} and relies on an intricate characterization of $\NED$ and $\CNED$ sets from Section \ref{section:criteria}. We give a vague formulation of this characterization here. 

\begin{theorem}
A closed set $E\subset \R^n$ is $\NED$ (resp.\ $\CNED$) if and only if for every $n$-integrable metric $\rho\, ds$, almost every path $\gamma$ in $\R^n$ can be perturbed by an arbitrarily small amount of $\rho$-length to avoid the set $E$ (resp.\ to intersect the set $E$ at countably many points). 
\end{theorem}

See Theorem \ref{theorem:criterion_compact} for a precise statement. Another application of this characterization is the removability of $\CNED$ sets for continuous Sobolev functions. The proof is given in Section \ref{section:sobolev}.
\begin{theorem}\label{theorem:sobolev_intro}
Let $E\subset \R^n$ be a closed $\CNED$ set. Then every continuous function $f\colon \R^n\to \R$ with  $f\in W^{1,n}(\R^n\setminus E)$ lies in $W^{1,n}(\R^n)$. 
\end{theorem}

\subsection{Examples of negligible sets}
So far we understand some general classes of $\QCH$-removable sets. First, sets of $\sigma$-finite Hausdorff $(n-1)$-measure are removable as shown by Besicovitch  \cite{Besicovitch:Removable} in dimension $2$ and by Gehring \cite{Gehring:Rings} in higher dimensions. Thus, we can say that such sets are removable for \textit{rectifiability} reasons.

Next, it is known that sets with good geometry, such as quasicircles, are removable in dimension $2$. More generally, in all dimensions, boundaries of John domains, H\"older domains, and domains with $n$-integrable quasihyperbolic distance are removable  \cites{Jones:removability, JonesSmirnov:removability, KoskelaNieminen:quasihyperbolic}. Roughly speaking, all of these sets have either no outward cusps or  some outward cusps,  but not too many on average. Thus, these sets are removable for \textit{geometric} reasons.

Finally, $\NED$ sets are removable as well due to \cite{AhlforsBeurling:Nullsets} in dimension $2$ and   \cites{AseevSycev:removable,VodopjanovGoldstein:removable} in higher dimensions. Thus, one could say that $\NED$ sets are removable because they are small in a \textit{potential theoretic} sense. Note that all $\NED$ sets are necessarily totally disconnected in dimension $2$.

The three classes of sets are mutually singular in a sense. Namely, there are rectifiable sets that have bad geometry and are large from a potential theoretic point of view. For example, consider a rectifiable curve with a dense set of both inward and outward cusps. Likewise, there are sets with good geometry that are not rectifiable and are large for potential theory. As an example, take a quasicircle of Hausdorff dimension larger than $1$. Finally, there are sets that are small in a potential theoretic sense, but are large in terms of rectifiability and have bad geometry. For example, consider a Cantor set $E\subset \R$ of measure zero and Hausdorff dimension $1$, and then take the set $E\times E$; this is an $\NED$ set by \cite{AhlforsBeurling:Nullsets}*{Theorem 10} since its projections to the coordinate directions have measure zero. 

A natural question is whether one can reconcile these three different worlds. In other words, is there a common reason for which all of the above classes of sets are $\QCH$-removable?  We provide an affirmative answer to this question.

\begin{theorem}\label{theorem:cned}
The following sets are of class $\CNED$ in $\R^n$.
\begin{enumerate}[\upshape(i)]
	\item Sets of class $\NED$.\label{item:cned:i}
	\item Sets of $\sigma$-finite Hausdorff $(n-1)$-measure. \label{item:cned:ii}
	\item Boundaries of domains with $n$-integrable quasihyperbolic distance. \label{item:cned:iii}
\end{enumerate}
\end{theorem}

The class of sets in \ref{item:cned:iii} is defined and discussed in Section \ref{section:quasi}, where we also give the proof. This class includes quasicircles, boundaries of John domains, and boundaries of H\"older domains. As discussed earlier, \ref{item:cned:i} is immediate; however, the other conclusions are new. The technique used for the proof of \ref{item:cned:ii} allows us to generalize a result of V\"ais\"al\"a \cite{Vaisala:null}, stating that a closed set $E\subset \R^n$ with Hausdorff $(n-1)$-measure zero is of class $\NED$, to non-closed sets.

\begin{theorem}\label{theorem:zero}
Let $E\subset \R^n$ be a set of Hausdorff $(n-1)$-measure zero. Then $E$ is $\NED$.
\end{theorem}

Theorem \ref{theorem:cned} \ref{item:cned:ii} and Theorem \ref{theorem:zero} are proved in Section \ref{section:perturbation} with the aid of the notion of a \textit{family of curve perturbations}; see Theorem \ref{theorem:perturbation_hausdorff}. Roughly speaking, such curve families contain almost every parallel translate of a curve and almost every radial segment. The main theorem of that section is Theorem \ref{theorem:perturbation_family}, which asserts that the modulus of a curve family remains unaffected, if one restricts to the intersection of that family with a family of curve perturbations.

Combining Theorems \ref{theorem:removable}, \ref{theorem:unions}, and \ref{theorem:cned}, we obtain the next removability result.

\begin{theorem}
Let $E\subset \R^n$ be a closed set that admits a decomposition into countably many sets $E_i$, $i\in \N$, each of which is contained in a closed set that is either $\NED$, or has $\sigma$-finite Hausdorff $(n-1)$-measure, or is the boundary of a domain with $n$-integrable quasihyperbolic distance. Then $E$ is $\CNED$ and $\QCH$-removable.
\end{theorem}

We also present some further examples. We show in Theorem \ref{theorem:projection} that planar sets (not necessarily closed) whose projection to each coordinate axis has measure zero are $\NED^w$, generalizing a result of Ahlfors--Beurling for closed sets \cite{AhlforsBeurling:Nullsets}*{Theorem 10}. Theorem \ref{theorem:projection} is used in the proof of Theorem \ref{example:ned} below. In Section \ref{section:nonmeasurable} we present an example of a non-measurable $\CNED$ set, constructed by Sierpi\'nski.

The results of this paper have already found an application in the problem of \textit{rigidity of circle domains}. A connected open set $\Omega$ in the Riemann sphere $\widehat{\C}$ is a \textit{circle domain} if each boundary component of $\Omega$ is either a circle or a point. A circle domain $\Omega$ is \textit{conformally rigid} if every conformal map from $\Omega$ onto another circle domain is the restriction of a M\"obius transformation of  $\widehat \C$. He--Schramm \cite{HeSchramm:Rigidity} proved that circle domains whose boundary has $\sigma$-finite Hausdorff $1$-measure are rigid. Later, Younsi and the author \cite{NtalampekosYounsi:rigidity} proved the rigidity of circle domains with $2$-integrable quasihyperbolic distance (as in Theorem \ref{theorem:cned} \ref{item:cned:iii}). It is conjectured that rigidity of a circle domain is equivalent to  $\QCH$-removability of the boundary. The next result incorporates all previous results and provides strong evidence towards this conjecture. 

\begin{theorem}[\cite{Ntalampekos:rigidity_cned}]
A circle domain is conformally rigid if every compact subset of its point boundary components is $\CNED$. 
\end{theorem}

The proof features especially Theorem \ref{theorem:unions} and the characterization of $\CNED$ sets given in Theorem \ref{theorem:criterion_compact}.

\subsection{Examples of non-negligible sets}
We remark that in Theorem \ref{theorem:unions} we are not requiring that the union of the closed sets $E_i$ be closed. However, both cases of the theorem fail without assuming that each individual set $E_i$ is closed. 

\begin{theorem}\label{example:ned}
There exist Borel $\NED$ sets $E_1,E_2 \subset \R^2$ such that $E_1\cup E_2$ is a closed set that is neither $\NED$ nor $\CNED$ nor $\QCH$-removable. 
\end{theorem}

The proof is given in Section \ref{section:nonexamples}. One can construct a more basic example with $E_1\cup E_2\notin \NED$ as follows. Tukia \cite{Tukia:hausdorff} gives an example of a set $E_1\subset [0,1]$ of full measure that can be mapped under a quasiconformal map of $\C$ onto a set of $1$-measure zero in the real line. Note that sets of $1$-measure zero are $\NED$ by Theorem \ref{theorem:zero} and such $\NED$ sets are quasiconformally invariant by Corollary \ref{corollary:qc_invariance}. Thus, $E_1$ is $\NED$. Also, $E_2=[0,1]\setminus E_1$ is $\NED$ because it has measure zero. However, $E_1\cup E_2=[0,1]$, which is not totally disconnected, so it is not $\NED$. 

Compared to $\NED$ sets, it is significantly harder to produce sets that are not $\CNED$. For the proof of Theorem \ref{example:ned} we use tools from the existing literature, and in particular from a work of Wu \cite{Wu:cantor}, to construct $\NED$ sets $E_1$ and $E_2$ such that $E_1\cup E_2$ is the product of a Cantor set in $\R$ with $[0,1]$. Such sets are not $\QCH$-removable (see \cite{Carleson:null} or the Introduction of \cite{Ntalampekos:removabilitydetour}) and thus they are not $\CNED$ by Theorem \ref{theorem:removable}; this can be proved directly in this simple situation.

As a corollary of Theorem \ref{example:ned} and \cite{Ntalampekos:metric_definition_qc}*{Theorem 1.2}, we obtain that unions of exceptional sets for the metric definition of quasiconformality are not necessarily exceptional. Here $H_f$ denotes the metric distortion of a map $f$; see \cite{Ntalampekos:metric_definition_qc}. 

\begin{corollary}
There exist Borel sets $E_1,E_2\subset \R^2$ such that $E_1\cup E_2$ is closed and for each $i\in \{1,2\}$, every homeomorphism $f$ of $\R^2$ with 
$$\sup_{x\in \R^2\setminus E_i}H_f(x)<\infty$$
is quasiconformal, but there exists a homeomorphism $g$ of $\R^2$ that is quasiconformal in $\R^2\setminus (E_1\cup E_2)$ and not quasiconformal in $\R^2$.
\end{corollary}

Finally, we end the introduction with some remarks on non-removable sets. In dimension $2$ it is known that all sets of positive area are not $\QCH$-removable \cite{KaufmanWu:removable}. There are Jordan curves of Hausdorff dimension $1$ that are non-removable \cites{Kaufman:graphnonremovable,Bishop:flexiblecurves}. Moreover, if $C\subset \R$ is a Cantor set, then $C\times [0,1]$ is non-removable as we discussed above. More interestingly, Wu \cite{Wu:cantor} proved that if $E,F\subset \R$ are Cantor sets and $E\notin \NED$, then $E\times F$ is non-removable; the converse is not true since $\NED$ sets can have positive length \cite{AhlforsBeurling:Nullsets}. More recently, the current author studied the problem of removability for fractals with infinitely many complementary components and proved that the Sierpi\'nski gasket and all Sierpi\'nski carpets are non-removable \cites{Ntalampekos:gasket,Ntalampekos:CarpetsNonremovable}. The latter result was generalized to higher dimensional carpets, known as Sierpi\'nski spaces, by the author and Wu \cite{NtalampekosWu:Sierpinski}. 

Gaskets and carpets fall into the general class of \textit{residual sets of packings}. A packing $\mathcal P$ in $\R^n$ is a collection of bounded, connected open sets $D_i$, $i\in \N\cup \{0\}$, such that $D_i\subset D_0$ for every $i\in \N$ and $D_i\cap D_j=\emptyset$ for $i,j\in \N$ with $i\neq j$.  The residual set of the packing $\mathcal P$ is the set 
\begin{align*}
\br {D_0} \setminus \bigcup_{i\in \N}D_i.
\end{align*}
We observe below that in many cases such residual sets  are not $\CNED$. 

\begin{prop}\label{proposition:packings}
Let $\mathcal P=\{D_i\}_{i\in \N\cup \{0\}}$ be a packing in $\R^n$ such that $\partial D_i\cap \partial D_j$ is  countable for $i\neq j$, $i,j\in \N\cup\{0\}$. Then the residual set of $\mathcal P$ is not  $\CNED$. 
\end{prop}

It was earlier observed that such residual sets in the plane can have Hausdorff dimension $1$ but not $\sigma$-finite Hausdorff $1$-measure \cite{MaioNtalampekos:packings}. Proposition \ref{proposition:packings} covers the Sierpi\'nski gasket and all Sierpi\'nski carpets. 

\subsection{Open problems}
Based on the results of this work, it is natural to propose the following problem, whose resolution would answer many of the open questions related to removable sets.

\begin{problem}\label{problem}
Do $\QCH$-removable sets coincide with closed $\CNED$ sets?
\end{problem}
We also formulate a series of questions for $\CNED$ sets.

\begin{question}\label{question:positive}
If $E\subset \R^n$ is a closed set that is not $\CNED$, does there exist a homeomorphism of $\R^n$ that is quasiconformal in $\R^n\setminus E$ and maps $E$ to a set of positive $n$-measure?
\end{question}
A positive answer to this question would resolve Problem \ref{problem} and thus it would also resolve among others the problems of local removability and of removability of unions of removable sets mentioned earlier. For closed $\NED$ sets in the plane the answer to the corresponding question is already known to be affirmative by Ahlfors--Beurling \cite{AhlforsBeurling:Nullsets}: if $E\notin \NED$, then there exists a conformal embedding $f\colon \R^2\setminus E \to \R^2$ such that the complement of $f(\R^2\setminus E)$ has positive area.
\begin{question}If $E\subset \R^n$ is a closed set that is not $\CNED$, does there exist a totally disconnected closed subset of $E$ that is not $\CNED$?
\end{question}
If yes, it would suffice to answer Question \ref{question:positive} for totally disconnected sets, which could be more approachable. Note that $\NED$ sets in the plane are already totally disconnected, so this makes the study of these sets more accessible.

\begin{problem}
Do removable sets for continuous $W^{1,n}$ functions  coincide with closed $\CNED$ sets?
\end{problem}
This problem is motivated by Theorem \ref{theorem:sobolev_intro}.  Obviously, the answer would be positive if the answer to Problem \ref{problem} were positive, in which case, $\QCH$-removable sets would coincide with removable sets for continuous $W^{1,n}$ functions. This is another open problem discussed in \cites{Bishop:flexiblecurves,JonesSmirnov:removability}.

\subsection*{Acknowledgements}
The author would like to thank Hrant Hakobyan and Malik Younsi for their comments.

\section{Preliminaries}\label{section:preliminaries}

\subsection{Notation and definitions}
We denote the Euclidean distance between points $x,y\in \R^n$ by $|x-y|$. For $x\in \R^n$ and $0\leq r<R$ we denote by $B(x,R)$ the open ball $\{y\in \R^n: |x-y|<R\}$ and by $A(x;r,R)$ the open ring $\{y\in \R^n: r<|x-y|<R\}$. The corresponding closed ball and ring are denoted by $\br B(x,r)$ and $\br A(x;r,R)$, respectively. If $B$ is an open (resp.\ closed) ball, then for $\lambda>0$ we denote by $\lambda B$ the open (resp.\ closed) ball with the same center as $B$ and radius multiplied by $\lambda$. We also set $S^{n-1}(x,r)=\partial B(x,r)$.  The open $\varepsilon$-neighborhood of a set $E\subset \R^n$ is denoted by $N_{\varepsilon}(E)$.

We use the notation $m_n$ for the $n$-dimensional Lebesgue measure in $\R^n$, $m_n^*$ for the outer $n$-dimensional Lebesgue measure, and $\int_{\R^n}\rho(x) \, dx$ or simply $\int \rho$ for the Lebesgue integral of a Lebesgue measurable extended function $\rho\colon \R^n\to [-\infty,\infty]$, if it exists. For such a function $\rho$, if $B$ is a measurable set with $m_n(B)\in (0,\infty)$, we define
$$\fint_B \rho= \frac{1}{m_n(B)} \int_B \rho.$$
For simplicity, extended functions will be called functions. A non-negative function is assumed to take values in $[0,\infty]$.

The cardinality of a set $E$ is denoted by $\#E$. For quantities $A$ and $B$ we write $A\lesssim B$ if there exists a constant $c>0$ such that $A\leq cB$. If the constant $c$ depends on another quantity $H$ that we wish to emphasize, then we write instead $A\leq c(H)B$ or $A\lesssim_H B$. Moreover, we use the notation $A\simeq B$ if $A\lesssim B$ and $B\lesssim A$. As previously, we write $A\simeq_H B$ to emphasize the dependence of the implicit constants on the quantity $H$. All constants in the statements are assumed to be positive even if this is not stated explicitly and the same letter may be used in different statements to denote a different constant.  

For $s\geq 0$ the \textit{$s$-dimensional Hausdorff measure} $\mathscr H^s(E)$ of a set $E\subset \R^n$ is defined by
$$\mathscr{H}^{s}(E)=\lim_{\delta \to 0} \mathscr{H}_\delta^{s}(E)=\sup_{\delta>0} \mathscr{H}_\delta^{s}(E),$$
where
$$
\mathscr{H}_\delta^{s}(E)=\inf \bigg\{ c(s)\sum_{j=1}^\infty \operatorname{diam}(U_j)^{s}: E \subset \bigcup_j U_j,\, \operatorname{diam}(U_j)<\delta \bigg\}
$$
for a normalizing constant $c(s)>0$ so that the $n$-dimensional Hausdorff measure agrees with Lebesgue measure in $\R^n$. Note that $c(1)=1$. The quantity $\mathscr{H}_\delta^{s}(E)$, $\delta\in (0,\infty]$, is called the \textit{$s$-dimensional Hausdorff $\delta$-content} of $E$. If $\delta=\infty$ we simply call this quantity the \textit{$s$-dimensional Hausdorff  content} of $E$. A standard fact that we will use is that 
\begin{align*}
\textrm{$\h^s(E)=0$ if and only if $\h^s_{\infty}(E)=0$. }
\end{align*}
We note that if $E\subset \R^n$ is a connected set, then (see 	\cite{BuragoBuragoIvanov:metric}*{Lemma 2.6.1, p.~53})
\begin{align*}
\h^1(E) \geq \h^1_{\infty}(E)=\diam(E).
\end{align*}

\begin{lemma}[\cite{Bojarski:inequality}]\label{lemma:bojarski}
Let $p\geq 1$ and $\lambda >0$. Suppose that $\{B_i\}_{i\in \N}$ is a collection of balls in $\R^n$ and $a_i$, $i\in \N$, is a sequence of non-negative numbers. Then
\begin{align*}
\left \|  \sum_{i\in \N} a_i \x_{\lambda B_i} \right \|_{L^p(\R^n)} \leq c({n,p,\lambda})  \left \|  \sum_{i\in \N} a_i \x_{B_i} \right \|_{L^p(\R^n)}.
\end{align*} 
\end{lemma}

\subsection{Rectifiable paths}
A \textit{path} or \textit{curve} is a continuous function $\gamma\colon I \to \R^n$, where $I\subset \R$ is a compact interval. The \textit{trace} of a path $\gamma$ is the image $\gamma(I)$ and will be denoted by $|\gamma|$. The \textit{endpoints} of a path $\gamma\colon [a,b]\to \R^n$ are the points $\gamma(a),\gamma(b)$ and we set $\partial \gamma= \{\gamma(a),\gamma(b)\}$. We say that a path $\widetilde \gamma$ is a \textit{weak subpath} of a path $\gamma$ if ${\#}\widetilde {\gamma}^{-1}(x)\leq \#\gamma^{-1}(x)$ for every $x\in \R^n$. In particular, this implies that $|\widetilde \gamma|\subset |\gamma|$. A path $\widetilde \gamma$ is a \textit{(strong) subpath} of a path $\gamma \colon I\to \R^n$ if $\widetilde \gamma$ is the restriction of $\gamma$ to a closed subinterval of $I$. Note that a strong subpath is always a weak subpath, but not vice versa. A path $\gamma$ is \textit{simple} if it is injective. Equivalently, $\#\gamma^{-1}(x)=1$ for each $x\in |\gamma|$. It is well-known that every path has a simple weak subpath with the same endpoints \cite{Willard:topology}*{Theorem 31.2, p.~219}. 

If $\gamma$ is a path and $E\subset \R^n$ is a set, then we say that $\gamma$ \textit{avoids} the set $E$ if $E\cap |\gamma|=\emptyset$ and intersects $E$ at (e.g.) finitely many points if $E\cap |\gamma|$ is a finite set; note that we are not taking into account the multiplicity in the latter case.

If $\gamma_i\colon [a_i,b_i]\to \R^n$, $i=1,2$, are paths such that $\gamma_1(b_1)=\gamma_2(a_2)$, then we can define the \textit{concatenation} of the two paths to be the path $\gamma\colon [a_1,b_2]\to \R^n$ such that $\gamma|_{[a_1,b_1]}=\gamma_1$ and $\gamma|_{[a_2,b_2]}=\gamma_2$. If $x,y\in \R^n$, then we denote the line segment from $x$ to $y$ by $[x,y]$. 

The \textit{length} of a path $\gamma$ is the total variation of the function $\gamma$ and is denoted by $\ell(\gamma)$. A path is \textit{rectifiable} if it has finite length. Let $\gamma\colon [a,b]\to \R^n$ be a path and $s\colon [c,d]\to [a,b]$ be an increasing or decreasing continuous surjection. Then the path $\gamma\circ s\colon [c,d]\to \R^n$ is called a  \textit{reparametrization} of $\gamma$ (by the function $s$). Every rectifiable path $\gamma$ admits a unique reparametrization $\widetilde \gamma \colon [0,\ell(\gamma)]\to \R^n$ by an increasing function so that $\ell(\widetilde \gamma|_{[0,t]})=t$ for all $t\in [0,\ell(\gamma)]$. The path $\widetilde \gamma$ is called the \textit{arclength parametrization} of $\gamma$.

If $\rho\colon \R^n \to [0,\infty]$ is a Borel function and $\gamma$ is a rectifiable path, then one can define the line integral $\int_{\gamma} \rho\, ds$ using the arclength parametrization of $\gamma$; see \cite{Vaisala:quasiconformal}*{Chapter 1, pp.~8--9}. Namely, if $\gamma\colon [0,\ell(\gamma)]\to \R^n$ is parametrized by arclength, then 
$$\int_{\gamma}\rho\, ds = \int_0^{\ell(\gamma)}\rho(\gamma(t))\, dt. $$
We gather some properties of line integrals below.

\begin{lemma}\label{lemma:paths}
Let $\gamma$ be a rectifiable path, $\widetilde \gamma$ be a weak subpath of $\gamma$, and $\rho\colon \R^n \to [0,\infty]$ be a Borel function. The following statements are true.
\begin{enumerate}[\upshape(i)]\smallskip
	\item\label{lemma:paths:i} $\displaystyle{\int_{\gamma}\rho\, ds}=  \int_{\R^n} \rho(x) \#\gamma^{-1}(x)\, d\h^1(x)$.\smallskip
	\item\label{lemma:paths:iv}$\displaystyle{\int_{|\gamma|} \rho\, d\h^1 \leq \int_{\gamma}\rho\, ds}$ with equality if $\gamma$ is simple.\smallskip
	\item\label{lemma:paths:ii} $\displaystyle{\int_{\widetilde \gamma}\rho\, ds\leq \int_{\gamma}\rho\, ds}$.\smallskip
	\item\label{lemma:paths:iii} If $G\subset \R^n$ is a Borel set such that $\h^1(G\cap |\gamma|)=0$, then 
	\begin{align*}
\int_{\gamma} \rho \, ds =  \int_{\gamma}\rho \x_{\R^n\setminus G}\, ds.
	\end{align*}
\end{enumerate}
In particular, the above statements hold for $\rho=1$, in which case $\displaystyle{\displaystyle{\int_{\gamma}\rho\, ds}=\ell(\gamma)}$.
\end{lemma}

\begin{proof}
Part \ref{lemma:paths:i} follows from \cite{Federer:gmt}*{Theorem 2.10.13, p.~177}. The inequality and equality in \ref{lemma:paths:iv} follow from \ref{lemma:paths:i}. Since $\widetilde \gamma$ is a weak subpath of $\gamma$, we have ${\#}\widetilde {\gamma}^{-1}(x)\leq \#\gamma^{-1}(x)$. Thus \ref{lemma:paths:i} implies \ref{lemma:paths:ii}. Part \ref{lemma:paths:iii} also follows from \ref{lemma:paths:i}  upon observing that $\x_{\R^n\setminus G}(x)\#\gamma^{-1}(x)= \#\gamma^{-1}(x)$ for $x\notin G\cap |\gamma|$ and thus for $\h^1$-a.e.\ $x\in \R^n$.  
\end{proof}

\subsection{Modulus}
Let $\Gamma$ be a family of curves in $\R^n$. A Borel function $\rho\colon \R^n \to [0,\infty]$ is \textit{admissible} for the path family $\Gamma$ if $$\int_{\gamma}\rho\, ds\geq 1$$
for all rectifiable paths $\gamma\in \Gamma$. For $p\geq 1$ we define the \textit{$p$-modulus} of $\Gamma$ as 
$$\md_p \Gamma = \inf_\rho \int \rho^p,$$
where the infimum is taken over all admissible functions $\rho$ for $\Gamma$. By convention, $\md_p \Gamma = \infty$ if there are no admissible functions for $\Gamma$. Note that unrectifiable paths do not affect modulus. Hence, we will assume that families of $p$-modulus zero appearing in the next considerations contain all unrectifiable paths. We will use the following standard facts about modulus:
\begin{enumerate}[label=(M\arabic*)]
\item The modulus $\md_p$ is an outer measure in the space of all curves in $\R^n$. In particular, it obeys the monotonicity and countable subadditivity laws. \label{m:outer_measure}

\item If every path of a family $\Gamma_1$ has a subpath lying in a family $\Gamma_2$, then $\md_p\Gamma_1\leq \md_p\Gamma_2.$\label{m:subordinate}

\item If $\Gamma_0$ is a path family with $\md_p\Gamma_0=0$, then the family of paths $\gamma$ that have a weak subpath in $\Gamma_0$ also has $n$-modulus zero (by Lemma \ref{lemma:paths} \ref{lemma:paths:ii}). Moreover, the family of paths that have a reparametrization contained in $\Gamma_0$ also has $p$-modulus zero. \label{m:subpath}


\item If $\rho\colon \R^n\to [0,\infty]$ is a Borel function with $\rho=0$ a.e., then for the family $\Gamma_0$ of paths $\gamma$ with $\int_{\gamma}\rho\, ds>0$ we have  $\md_n\Gamma_0=0$. \label{m:zero_ae}

\item  The modulus $\md_p$ obeys the \textit{serial} law: if $\Gamma_i$, $i\in \N$, are curve families supported in disjoint Borel sets, then
\begin{align*}
\md_p\left(\bigcup_{i=1}^\infty \Gamma_i \right) \geq \sum_{i=1}^\infty \md_p \Gamma_i.
\end{align*}\label{m:serial}


\item Let $E\subset \R^n$ be a set with $m_n(E)=0$. Let $\Gamma_0$ be the family of paths $\gamma$ such that $\h^1( |\gamma|\cap E)>0$. Then $\md_p\Gamma_0=0$. \label{m:positive_length}

\item For $p=n$ the modulus $\md_n$ is invariant under conformal maps.

\item Let $\Gamma=\Gamma(A(x;r,R))$ be the family of curves in $\R^n$ joining the boundary components of the ring $A(x;r,R)$.  Then 
$$\md_n\Gamma= \omega_{n-1}\left(\log\frac{R}{r} \right)^{1-n}.$$\label{m:ring}
\end{enumerate}
Here $\omega_{n-1}$ is the area of the unit sphere in $\R^n$. See \cite{Vaisala:quasiconformal}*{Chapter 1} and \cite{HeinonenKoskelaShanmugalingamTyson:Sobolev}*{Sections 5.2--5.3} for more details about modulus and proofs of these facts.

For a path $\gamma\colon I\to \R^n$ and $x\in \R^n$ we define $\gamma+x$ to be the path $I\ni t\mapsto \gamma(t)+x$.  

\begin{lemma}\label{modulus:avoid}
Let $\Gamma$ be a family of paths in $\R^n$ with $\md_p\Gamma=0$.
\begin{enumerate}[\upshape(i)]
	\item For each rectifiable path $\gamma$ and for a.e.\ $x\in \R^n$ we have $\gamma+x\notin \Gamma$.
	\item For each line segment $L$ parallel to a direction $v\in \R^n$ and for $\h^{n-1}$-a.e.\ $x\in \{v\}^{\perp}$ we have $L+x\notin \Gamma$.
	\item For each $x\in \R^n$ and $w\in S^{n-1}(0,1)$ define $\gamma_w(t)=x+tw$, $t\geq 0$. Then for $0<r<R$ we have $\gamma_w|_{[r,R]}\notin \Gamma$ for $\h^{n-1}$-a.e.\ $w\in S^{n-1}(0,1)$.  
\end{enumerate}  
\end{lemma}
\begin{proof}
For each $\varepsilon>0$ there exists a function $\rho$ that is admissible  for $\Gamma$ with $\int \rho^p<\varepsilon$. 

For (i), let $\gamma$ be a rectifiable path in $\R^n$, parametrized by arclength. Let $r>0$ and $G_r$ be the set of $x\in B(0,r)$ such that $\gamma+x\in \Gamma$. We also fix $R>0$ such that $|\gamma+x|\subset B(0,R)$ whenever $x\in B(0,r)$. Then by Chebychev's inequality and Fubini's theorem we have
\begin{align*}
m_n^*(G_r) &\leq m_n \left(\left\{x \in B(0,r): \int_{\gamma+x} \rho\, ds \geq 1  \right\} \right) \leq \int_{B(0,r)} \int_0^{\ell(\gamma)} \rho(\gamma(t)+x)\, dtdx\\
&= \int_0^{\ell(\gamma)} \int_{B(0,r)} \rho(\gamma(t)+x)\, dxdt\leq \ell(\gamma)\|\rho\|_{L^1(B(0,R))}.
\end{align*}
As $\varepsilon\to 0$, we have $\|\rho\|_{L^p(B(0,R))} \to 0$, so $\|\rho\|_{L^1(B(0,R))}\to 0$. This shows that $m_n(G_r)=0$ for all $r>0$. Thus, $\gamma+x\notin \Gamma$ for a.e.\ $x\in \R^n$.

For part (ii), let $G_r$ be the set of $x\in B(0,r)\cap \{v\}^{\perp}$ such that $L+x\in \Gamma$. Then
\begin{align*}
\h^{n-1}(G_r) \leq \h^{n-1} \left( \left\{ x\in B(0,r)\cap \{v\}^{\perp} : \int_{L+x} \rho\, ds \geq 1  \right\}\right)\leq \|\rho\|_{L^1(D)},
\end{align*}
where $D$ is the cylinder of radius $r$ with axis $L$. We now let $\varepsilon\to 0$ as above.  

Finally, for (iii), let $G$ be the set of $w\in S^{n-1}(0,1)$ such that $\gamma_w|_{[r,R]}\in \Gamma$. Then, using polar integration we have
\begin{align*}
\h^{n-1}(G) &\leq \h^{n-1} \left( \left\{ w\in S^{n-1}(0,1) : \int_{\gamma_w|_{[r,R]}} \rho\, ds \geq 1  \right\}\right)\\
&\leq \int_{S^{n-1}(0,1)} \int_{r}^R \rho( x+tw)\, dt d\h^{n-1}(w)\lesssim_n r^{-n+1} \|\rho\|_{L^1(B(x,R))}.
\end{align*}
As before, we let $\varepsilon\to 0$ to obtain the conclusion. 
\end{proof}

\subsection{\texorpdfstring{Elementary properties of $\NED$ and $\CNED$ sets}{Elementary properties of NED and CNED sets}}\label{section:elementary}

We first recall the definitions. For an open set $U\subset \R^n$ and for any two closed sets {$F_1,F_2\subset \R^n$} the family of curves joining $F_1$ and $F_2$ inside $U$ is denoted by $\Gamma(F_1,F_2;U)$. In other words, this family contains the curves $\gamma\colon [a,b]\to \br U$ with $\gamma(a)\in F_1$, $\gamma(b)\in F_2$, and $\gamma((a,b))\subset U$.  For a set $E\subset \R^n$ we denote by $\mathcal F_0(E)$ the family of curves in $\R^n$ that do not intersect $E$, {except possibly at the endpoints}; that is, $\mathcal F_0(E)$ contains the curves $\gamma\colon [a,b]\to \R^n$ such that $|\gamma|\setminus \gamma(\{a,b\})$ does not intersect $E$. Moreover, we define $\mathcal F_{\sigma}(E)$ to be the family of curves in $\R^n$ that intersect $E$ in countably many points; that is, the set $E\cap |\gamma|$ is countable (finite or infinite). 

Let $E\subset \R^n$ be a set and $\Omega\subset \R^n$ be an open set. The set $E$ lies in $\NED(\Omega)$ if for every pair of non-empty, disjoint continua $F_1,F_2\subset \Omega$ we have
\begin{align*}
\md_n \Gamma(F_1,F_2;\Omega) =\md_n (\Gamma(F_1,F_2;\Omega)\cap \mathcal F_0(E)).
\end{align*}
If we have instead
\begin{align*}
\md_n \Gamma(F_1,F_2;\Omega) \leq M\cdot \md_n (\Gamma(F_1,F_2;\Omega)\cap \mathcal F_0(E))
\end{align*}
for a uniform constant $M\geq 1$, then $E$ lies in $\NED^w(\Omega)$. A set $E\subset \R^n$ lies in $\CNED(\Omega)$ and $\CNED^w(\Omega)$ if the above equality and inequality, respectively, hold with $\mathcal F_{\sigma}(E)$ in place of $\mathcal F_0(E)$. In the case that $\Omega=\R^n$, we simply use the notation $\NED$, $\NED^w$, $\CNED$, and $\CNED^w$.

We will use the notation $\*ned(\Omega)$ and $\mathcal F_*(E)$ for $\NED(\Omega)$ or $\CNED(\Omega)$ and for $\mathcal F_0(E)$ or $\mathcal F_{\sigma}(E)$, respectively. Similarly, we will use the notation $\*ned^w(\Omega)$ for $\NED^w(\Omega)$ or $\CNED^w(\Omega)$.  By the monotonicity of modulus, if $E\in \*ned(\Omega)$ (resp.\ $\*ned^w(\Omega)$) and $F\subset E$, then $F$ lies in $\*ned(\Omega)$ (resp.\ $\*ned^w(\Omega)$). Moreover, if $E\in \*ned^w(\Omega)$, it is immediate that $E\cap \Omega$ must have empty interior.

A set is \textit{non-degenerate} if it contains more than one points. For two non-degenerate sets $F_1,F_2\subset \R^n$ we define the \textit{relative distance} $\Delta(F_1,F_2)$ by
\begin{align*}
\Delta(F_1,F_2)=\frac{\dist(F_1,F_2)}{\min\{\diam(F_1),\diam(F_2)\}}.
\end{align*}
\begin{lemma}\label{lemma:measure_zero_loewner}
Let $E\subset \R^n$ be a set. Suppose that there exist constants $t,\phi>0$ such that for each $x_0\in \R^n$ there exists $r_0>0$ with the property that for every pair of non-degenerate, disjoint continua $F_1,F_2\subset B(x_0,r_0)$ with $\Delta(F_1,F_2)\leq t$ we have
\begin{align*}
\md_n(\Gamma(F_1,F_2;\R^n) \cap \mathcal F_*(E)) \geq \phi.
\end{align*}
If $E$ is Lebesgue measurable, then $m_n(E)=0$. 
\end{lemma}

\begin{proof}
Since $E$ is measurable, it differs from a Borel subset by a set of measure zero. Note also that the family $\mathcal F_*(E)$ increases when we pass to a subset of $E$.  Thus, we assume that $E$ is Borel itself. Suppose that $m_n(E)>0$ and that $x_0$ is a Lebesgue density point of $E$. Let $r_0$ be as in the assumption and consider $r<r_0$. Define $F_1=\partial B(x_0,r)$ and $F_2=\partial B(x_0,ar)$, where $a\in (0,1)$ is chosen so that 
\begin{align*}
\Delta(F_1,F_2)= \frac{1-a}{2a} \leq t. 
\end{align*} 
Consider the function 
$$\rho(x)=\frac{1}{|x-x_0|\log(a^{-1})} \x_{A(x_0;ar,r)}(x).$$
Then $\rho$ is admissible for $\Gamma(F_1,F_2;\R^n)$; see for example \cite{Heinonen:metric}*{7.14, pp.~52--53}. We set $\rho_1=\rho\x_{\R^n\setminus E}$, which is Borel measurable. If $\gamma$ is a curve in $\Gamma(F_1,F_2;\R^n)\cap \mathcal F_{*}(E)$, then $\gamma$ intersects $E$ at countably many points, so by Lemma \ref{lemma:paths} \ref{lemma:paths:iii} we have
\begin{align*}
	\int_{\gamma} \rho_1\, ds= \int_{\gamma}\rho\, ds\geq 1.
\end{align*}
Hence, $\rho_1$ is admissible for $\Gamma(F_1,F_2;\R^n)\cap \mathcal F_{*}(E)$. We conclude that
\begin{align*}
\phi\leq  \md_n (\Gamma(F_1,F_2;\R^n)\cap \mathcal F_{*}(E))\leq \int_{\R^n\setminus E} \rho^n\leq \frac{m_n( B(x_0,r)\setminus E)}{(\log (a^{-1}))^n a^n r^n} .
\end{align*} 
Letting $r\to 0$ contradicts the assumption that $x_0$ is a density point of $E$.  
\end{proof}

\begin{lemma}\label{lemma:measure_zero}
Let $\Omega\subset \R^n$ be an open set and $E\subset \Omega$ be a set with $E\in \*ned^w(\Omega)$. 
\begin{enumerate}[\upshape(i)]
\item If  $\br E\subset \Omega$, then $E$ satisfies the assumption of Lemma \ref{lemma:measure_zero_loewner}.\label{lemma:measure_zero:i}
\item If $E$ is Lebesgue measurable, then $m_n(E)=0$. \label{lemma:measure_zero:ii}
\end{enumerate}
\end{lemma}

It was proved in \cite{Vaisala:null}*{Theorem 1} that closed $\NED$ sets have measure zero. However, we remark that there exists a non-measurable set, constructed by Sierpi{\'n}\-ski for a different purpose, that is of class $\CNED$; see the discussion in Section \ref{section:nonmeasurable}.  

\begin{proof}
Let $\br E\subset \Omega$ as in \ref{lemma:measure_zero:i}. If $x_0\notin \br E$, there exists a ball $B(x_0,r_0)\subset \R^n\setminus E$. For any pair of non-degenerate, disjoint continua $F_1,F_2\subset B(x_0,r_0)$ we have 
\begin{align*}
\md_n(\Gamma(F_1,F_2;\R^n) \cap \mathcal F_*(E)) \geq \md_n\Gamma(F_1,F_2;B(x_0,r_0))
\end{align*}
by the monotonifity of modulus. If $x_0\in \br E$, consider a ball $B(x_0,r_0)\subset \Omega$. Since $E\in \*ned^w(\Omega)$, there exists a uniform constant $M\geq 1$ such that for $F_1,F_2\subset B(x_0,r_0)$ as above, we have
\begin{align*}
\md_n(\Gamma(F_1,F_2;\R^n) \cap \mathcal F_*(E)) &\geq \md_n(\Gamma(F_1,F_2;\Omega)\cap \mathcal F_*(E))\\
&\geq M^{-1} \cdot \md_n \Gamma(F_1,F_2;\Omega)\\
&\geq M^{-1} \cdot  \md_n \Gamma(F_1,F_2;B(x_0,r_0)). 
\end{align*} 
Each open ball in Euclidean space is a \textit{Loewner space} \cite{Heinonen:metric}*{Example 8.24 (a), p.~65}, so the latter modulus is uniformly bounded from below, provided that $\Delta(F_1,F_2)\leq 1$. Therefore the assumption of Lemma \ref{lemma:measure_zero_loewner} is satisfied.

For  \ref{lemma:measure_zero:ii}, note that each closed set $K\subset E$ is contained in $\Omega$, lies in $\*ned^w(\Omega)$, and by part \ref{lemma:measure_zero:i} satisfies the assumption of Lemma \ref{lemma:measure_zero_loewner}. By Lemma \ref{lemma:measure_zero_loewner}, $m_n(K)=0$. Since $E$ is measurable, $m_n(E)=0$.
\end{proof}

\subsection{\texorpdfstring{Comparison to classical definition of $\NED$ sets}{Comparison to classical definition of NED sets}}\label{section:ned_classical}

According to the classical definition, a \textit{closed} set $E\subset \R^n$ is $\NED$ if for every pair of non-empty, disjoint continua $F_1,F_2\subset \R^n\setminus E$ we have 
\begin{align*}
\md_n \Gamma(F_1,F_2;\R^n) = \md_n \Gamma(F_1,F_2;\R^n\setminus E)=\md_n (\Gamma(F_1,F_2;\R^n)\cap \mathcal F_0(E)).
\end{align*}
In our definition, we required the stronger condition that the above equality holds for all disjoint continua $F_1,F_2\subset \R^n$ regardless of whether they intersect the set $E$. The reason for allowing such a generality in our approach is that we impose no topological assumptions on $E$, which could be even dense in $\R^n$; therefore it would be too restrictive and unnatural to work with continua $F_1,F_2\subset \R^n\setminus E$.

We show that the two definitions agree. The proof relies on some results relating $n$-modulus with \textit{$n$-capacity}. Let $U\subset \R^n$ be an open set and $F_1,F_2\subset \br U$ be disjoint continua. The $n$-capacity of the \textit{condenser} $(F_1,F_2;U)$ is defined as 
\begin{align*}
\cp_n(F_1,F_2;U)= \inf_{u} \int_U |\nabla u|^n,
\end{align*}
where the infimum is taken over all functions $u$ that are continuous in $U\cup F_1\cup F_2$, ACL  in $U$ \cite{Vaisala:quasiconformal}*{Section 26, p.~88}, with $u=0$ in a neighborhood of $F_1$ and $u=1$ in a neighborhood of $F_2$ \cite{Hesse:capacity}*{Theorem 3.3}.  Equivalently, one can replace in this definition continuous ACL functions with the Dirichlet space $L^{1,n}(U)$ of locally integrable functions in $U$ with distributional derivatives of first order lying in $L^n(U)$. It has been shown by Hesse \cite{Hesse:capacity}*{Theorem 5.5} that whenever $F_1,F_2$ are continua in $U$ (specifically, not intersecting $\partial U$), then 
$$\cp_n(F_1,F_2;U) =\md_n\Gamma(F_1,F_2;U).$$
This result was generalized to the case that $F_1,F_2\subset \br U$, provided that $U$ is a $\mathit{QED}$ \textit{domain}, by Herron--Koskela \cite{HerronKoskela:QEDcircledomains}; see also the work of Shlyk \cite{Shlyk:CapacityModulus} for a more general result. By definition, a connected open set $U\subset \R^n$ is a $\mathit{QED}$ domain if there exists a constant $M\geq 1$ such that 
$$ \md_n \Gamma(F_1,F_2;\R^n)\leq M\cdot \md_n\Gamma(F_1,F_2;U)$$ 
for all pairs of non-empty, disjoint continua $F_1,F_2\subset U$.  

Suppose that $E$ is an $\NED$ set according to the classical definition. Then $U=\R^n\setminus E$ is a $\mathit{QED}$ domain, whose closure is $\R^n$; classical $\NED$ sets have empty interior \cite{Vaisala:null}. Thus, the result of Herron--Koskela gives
\begin{align*}
\cp_n(F_1,F_2;\R^n\setminus E) &=\md_n\Gamma(F_1,F_2;\R^n\setminus E)
\end{align*}
for all non-empty, disjoint continua $F_1,F_2\subset \R^n$. Summarizing, in order to show the equivalence of the classical definition to the current one, it suffices to show that 
\begin{align*}
\cp_n(F_1,F_2;\R^n)= \cp_n(F_1,F_2;\R^n\setminus E)
\end{align*}
for all non-empty, disjoint continua $F_1,F_2\subset \R^n$. Observe  that this equality already holds if $F_1,F_2\subset \R^n\setminus E$. Hence, $E$ is \textit{removable for $n$-capacity}, in the sense of Vodopyanov--Goldshtein \cite{VodopjanovGoldstein:removable}. By \cite{VodopjanovGoldstein:removable}*{Theorem 3.1, p.~46}, such sets coincide with the sets that are removable for the Dirichlet space $L^{1,n}$. That is, $m_n(E)=0$ and if $u\in L^{1,n}(\R^n\setminus E)$, then $u\in L^{1,n}(\R^n)$ and the distributional derivatives of $u$ are the same in both spaces.   Now, let $F_1,F_2$ be any non-empty, disjoint continua in $\R^n$; in particular, they might intersect the set $E$, as in the current definition of $\NED$ sets. We trivially have $$\cp_n(F_1,F_2;\R^n\setminus E)\leq \cp_n(F_1,F_2;\R^n).$$
Since $E$ is removable for the Dirichlet space $L^{1,n}$,  we also have the reverse inequality, completing the proof.

\section{Families of curve perturbations}\label{section:perturbation}
Let $\mathcal F$ be a path family in $\R^n$. We define $\partial \mathcal F$ to be the set of points $x\in \R^n$ that are endpoints of some path of $\mathcal F$ and $x\notin |\gamma|\setminus \partial \gamma$ for any path $\gamma\in \mathcal F$. In other words, $\partial \mathcal F$ contains endpoints of paths in $\mathcal F$ that are not interior points of any path of $\mathcal F$. For example, if $\mathcal F$ is the family $\Gamma( F_1,F_2;U)$, where $U$ is a ring $A(0;r,R)$ and $F_1,F_2$ are the boundary components of $U$, then $\partial \mathcal F=F_1\cup F_2$. Another example is the family $\mathcal F_0(E)$ for some set $E\subset \R^n$; recall its definition from Section \ref{section:elementary}. Then, $\partial \mathcal F_0(E)\supset E$. Indeed, every point $x\in E$ can be considered as a constant path in $\mathcal F_0(E)$; recall that paths of $\mathcal F_0(E)$ can have endpoints in $E$. Moreover, if $x\in E$, then $x\notin |\gamma|\setminus \partial \gamma$ for any $\gamma\in \mathcal F_0(E)$; thus $x\in \partial \mathcal F_0(E)$. 

\begin{definition}\label{definition:pfamily}
We say that a path family $\mathcal F$ in $\R^n$ is a \textit{family of curve perturbations}, or else, a \textit{$P$-family}, if
\begin{enumerate}[\upshape(P1)]
 	\item\label{perturbations:1} for all non-constant rectifiable paths $\gamma$ in $\R^n$ we have $\gamma+x\in \mathcal F$ for a.e.\ $x\in \R^n$,
 	\item\label{perturbations:2} for every $x\in \R^n$ and $r>0$, the radial segment $t\mapsto x+tw$, $0\leq t\leq r$, lies in $\mathcal F$ for $\h^{n-1}$-a.e.\ $w\in S^{n-1}(0,1)$,
 	\item\label{perturbations:3} $\mathcal F$ is closed under strong subpaths and reparametrizations, and
 	\item\label{perturbations:4} if two paths $\gamma_1,\gamma_2\in \mathcal F$ have a common endpoint that does not lie in $\partial \mathcal F$, then the concatenation of $\gamma_1$ with $\gamma_2$ on that endpoint lies in $\mathcal F$.
\end{enumerate}
\end{definition}

Property \ref{perturbations:4} holds always for families that are closed under concatenations; for example the family $\mathcal F_{\sigma}(E)$ of curves intersecting a given set $E$ at countably many points is such. The reason for requiring that the common endpoint of $\gamma_1$ and $\gamma_2$ does not lie in $\partial \mathcal F$ is that we wish to accommodate curve families such as $\mathcal F_0(E)$, which contains paths that do not intersect $E$ \textit{except possibly at the endpoints}. Since $\partial \mathcal F_0(E)\supset E$, we remark that $\mathcal F_0(E)$ always satisfies \ref{perturbations:4}. Finally, we note that \ref{perturbations:3} always holds for $\mathcal F_0(E)$ and $\mathcal F_{\sigma}(E)$. We summarize these remarks below.

\begin{remark}\label{remark:p3p4}
Properties \ref{perturbations:3} and \ref{perturbations:4} always hold for the families $\mathcal F_0(E)$ and $\mathcal F_{\sigma}(E)$, for each $E\subset \R^n$. 
\end{remark}

\begin{lemma}\label{lemma:intersection}
The intersection of countably many $P$-families $\mathcal F_i$, $i\in \N$, is again a $P$-family.
\end{lemma}
\begin{proof}
Let $\mathcal F=\bigcap_{i\in \N} \mathcal F_i$. Properties \ref{perturbations:1}, \ref{perturbations:2}, and \ref{perturbations:3} are immediate for $\mathcal F$. For \ref{perturbations:4}, note that every constant path in $\R^n$ lies in $\mathcal F_i$ for each $i\in \N$; this follows by combining \ref{perturbations:2} with \ref{perturbations:3}. Thus, if $x\in \partial \mathcal F_i$ for some $i\in \N$, then $x$ is the endpoint of a (constant) path in $\mathcal F$. Moreover, $x \notin |\gamma|\setminus \partial \gamma$ for any path $\gamma$ of $\mathcal F_i\supset \mathcal F$, so
\begin{align*}
\bigcup_{i\in \N} \partial \mathcal F_i\subset \partial \mathcal F.
\end{align*}
Now, if $\gamma_1,\gamma_2\in \mathcal F$ have a common endpoint that does not lie in $\partial \mathcal F$, then it also does not lie in $\partial \mathcal F_i$ for any $i\in \N$. Hence, by \ref{perturbations:4}, the concatenation of $\gamma_1$ with $\gamma_2$ lies in $\mathcal F_i$ for each $i\in \N$. This proves that \ref{perturbations:4} holds for $\mathcal F$.
\end{proof}

In Section \ref{section:perturbation:examples} we will see important examples of such families. Specifically, if $\mathscr H^{n-1}(E)=0$, then the family $\mathcal F_0(E)$ is a $P$-family and if $E$ has $\sigma$-finite Hausdorff $(n-1)$-measure, then $\mathcal F_{\sigma}(E)$ is a $P$-family.

\subsection{The invariance theorem}

The main result of the section states that $n$-modulus is not affected if we restrict a path family to a $P$-family.

\begin{theorem}\label{theorem:perturbation_family}
Let $\mathcal F$ be a family of curve perturbations in $\R^n$.  Then for every open set $U\subset \R^n$ and all pairs of non-empty, disjoint continua $F_1,F_2 \subset  U$ we have
\begin{align*}
\md_n \Gamma (F_1,F_2;U) = \md_n (\Gamma(F_1,F_2;U) \cap \mathcal F)
\end{align*}
\end{theorem}

We first establish several auxiliary results.

\begin{lemma}\label{lemma:loewner}
Let $\mathcal F$ be a family of curve perturbations in $\R^n$. Let $A=A(0;r,R)$, $0<7r\leq R$, and $F_1,F_2\subset \R^n$ be disjoint continua such that every sphere $S^{n-1}(0,\rho)$, $r<\rho<R$, intersects both $F_1$ and $F_2$. Then 
\begin{align*}
\md_n(\Gamma(F_1,F_2;A)\cap \mathcal F)\geq c(n) \log\left(\frac{R}{r}\right).
\end{align*}
\end{lemma}

The statement is also true for $r<R<7r$ but we do not prove this for the sake of brevity. The statement without restricting to the family $\mathcal F$ is classical and can be found in \cite{Vaisala:quasiconformal}*{Theorem 10.12, p.~31}. Our proof relies on the next lemma.

\begin{lemma}[\cite{KallunkiKoskela:quasiconformal}*{Lemma 2.1}]\label{lemma:kallunki_koskela}
Let $u\colon \R^n\to [0,\infty]$ be a Borel function and $F\subset \R^n$ be a continuum. Suppose that for each $y\in F$ there exists a set $D_y\subset S^{n-1}(y,1)$ with $\h^{n-1}(D_y)\geq a>0$ for some $a>0$ such that 
$$\int_{[y,w]}u\, ds \geq 1$$
for each $w\in D_y$. Then
$$\int_{\R^n} u^n \geq c(n,a) \diam (F).$$
\end{lemma}

\begin{proof}[Proof of Lemma \ref{lemma:loewner}]
We will first prove the statement for $R=7r$. We will perform several normalizations and reductions. By the conformal invariance of $n$-modulus, we may assume that $r=1$. There exist closed balls ${B_i}=\br B(z_i,1/2)$ with $z_i\in F_i\cap S^{n-1}(0,9/2)$, and  disjoint continua $F_i'\subset F_i\cap B_i$ with $\diam(F_i')\geq 1/2$ for $i=1,2$. We will find a uniform lower bound for $\md_n\Gamma(F_1',F_2';A)$, which will give a uniform lower bound for $\md_n\Gamma(F_1,F_2;A)$ by the monotonicity of modulus. From now on, we denote $F_i'$ by $F_i$, $i=1,2$, for simplicity. 

By applying a conformal map, we assume that $B_1$ and $B_2$ are symmetric with respect to the hyperplane $ \R^{n-1}\times\{0\}$ and their centers have non-negative first coordinate. The choice of the balls and of the normalization is such that for all points $w$ in the $(n-1)$-dimensional disk $S=B((6,0,\dots,0),1)\cap (\{6\}\times \R^{n-1})$ and for all $y\in B_i$, $i=1,2$, the segment $[y,w]$ lies in the original ring $A$; see Figure \ref{figure:rings}. 

\begin{figure}
	\begin{tikzpicture}[line cap=round,line join=round,>=triangle 45,x=1.0cm,y=1.0cm, scale=0.5]
\clip(-11.372839727182667,-7.1) rectangle (11.610591314631261,7.1);
\draw [line width=.4pt] (0.,0.) circle (1.cm); 
\draw [line width=.4pt] (0.,0.) circle (7.cm);
\draw [line width=.4pt] (6.,1.)-- (6.,0.);
\draw [line width=.4pt] (6.,-1.)-- (6.,0.);
\fill (6,0)  circle[radius=2pt] node[anchor=west] {${(6,0)}$};
\node[anchor=north] (s) at (6,-1) {$S$};
\draw [line width=.4pt] (0.,4.5) circle (0.5cm);
\node[anchor=west] (b1) at (0.5,5) {$B_1$};
\draw [line width=.4pt] (0.,-4.5) circle (0.5cm);
\node[anchor=west] (b2) at (0.5,-5) {$B_2$};
\draw [line width=.4pt] (0.3378467503693706,4.546400712676679)-- (6.,0.5482064741847412);
\fill (0.33,4.54)  circle[radius=2pt] node[anchor=east] {$y$};
\draw [line width=.4pt] (6.,0.5482064741847412)-- (-0.3311118187168322,-4.430089853898605);
\fill (-0.33,-4.43)  circle[radius=2pt] node[anchor=west] {$z$};
\fill (6,0.54)  circle[radius=2pt] node[anchor=south west] {$w$};
\end{tikzpicture}
	\caption{}\label{figure:rings}
\end{figure}

We remark that $1/2\leq \diam(F_i)\leq 1$, $\diam(S)=2$, and $1\leq \dist(F_i,S)\leq 14=2R$ for $i=1,2$. Thus, $\diam(S)\simeq \diam(F_i)\simeq \dist(F_i,S)\simeq 1$. For $y\in F_1$, $w\in S$, and $z\in F_2$ consider the concatenation $\gamma(y,w,z)$ of the line segments $[y,w]$ and $[w,z]$.  Note that $\gamma(y,w,z)\subset A$ and $\gamma(y,w,z)\in \Gamma(F_1,F_2;A)\cap \mathcal F$ for each $y\in F_1$, $z\in F_2$, and a.e.\ $w\in S$, by the properties of a $P$-family. Indeed, \ref{perturbations:2} and \ref{perturbations:3} imply that for a.e.\ $w\in S$ the segments $[y,w]$ and $[w,z]$ lie in $\mathcal F$. Moreover, the same properties imply that a.e.\ $w\in S$ does not lie in $\partial \mathcal F$. Hence, by \ref{perturbations:4}, the concatenation of the segments $[y,w]$ and $[w,z]$ lies in $\mathcal F$. 

For each fixed $y\in F_i$ consider the map $S\ni w\mapsto \Phi_y(w)= \frac{w-y}{|w-y|}$. By the relative position of $y$ and $S$, this map is uniformly bi-Lipschitz. Thus, if $S'\subset S$ and $\h^{n-1}(S') \geq a$ for some $a>0$, then  $\h^{n-1}( \Phi_y(S'))\gtrsim_n a$.  We note that the implicit constants are independent of the point $y\in F_i$.

Now, let $\rho$ be an admissible function for $\Gamma(F_1,F_2;A)\cap \mathcal F$. We have 
\begin{align*}
\int_{\gamma(y,w,z)} \rho\, ds\geq 1
\end{align*}
for all $y\in F_1$, $z\in F_2$ and a.e.\ $w\in S$. Suppose that for each $y\in F_1$ there exists $S_y\subset S$ with $\h^{n-1}(S_y)\geq \frac{1}{2}\h^{n-1}(S)$ such that we have $$\int_{[y,w]}\rho\, ds \geq 1/2$$
for all $w\in S_y$. Then for the set $D_y= \Phi_y(S_y)$ we have $\h^{n-1}(D_y)\gtrsim_n 1$.  Lemma \ref{lemma:kallunki_koskela} now implies that $$\int \rho^n \gtrsim_n 1. $$

The other case is that there exists $y\in F_1$ such that there exists a subset $S'$ of $S$ with $\h^{n-1}(S')\geq \frac{1}{2}\h^{n-1}(S)$, and 
$$\int_{[y,w]}\rho\, ds < 1/2$$
for each $w\in S'$. This implies that for each $z\in F_2$ and for a.e.\ $w\in S'$ we have $$\int_{[z,w]} \rho\, ds\geq 1/2.$$
As before, Lemma \ref{lemma:kallunki_koskela} gives that 
$$\int \rho^n \gtrsim_n 1.$$ 
Therefore, we have shown that 
$$\md_n(\Gamma(F_1,F_2;A)\cap \mathcal F)  \geq c(n)>0$$
for a uniform constant $c(n)$ depending only on $n$, whenever $R=7r$.

In the general case, let $k\in \N$ be the largest integer such that $R\geq 7^kr$. Consider the rings $A_i=A(0;7^{i-1}r, 7^ir)$, $i\in \{1,\dots,k\}$. By the serial law \ref{m:serial}, we have
\[
\md_n(\Gamma(F_1,F_2;A)\cap \mathcal F)\geq \sum_{i=1}^k \md_n(\Gamma(F_1,F_2;A_i)\cap \mathcal F)\gtrsim_n k \gtrsim_n \log(R/r). \qedhere\]
\end{proof}

\begin{lemma}\label{lemma:path_bound_modulus}
Let $x\in \R^n$, $R>0$, and $\rho\colon \R^n\to [0,\infty]$  be a Borel function with $\rho \in L^n(\R^n)$.
\begin{enumerate}[\upshape(i)]
	\item\label{lemma:path_bound_modulus:i} For $M>0$, let $\Gamma_M$ be the family of paths $\gamma$ that intersect the ball $B(x,R)$ and satisfy $\ell(\gamma)>MR$. Then $\md_n \Gamma_M\to 0$ as $M\to\infty$.\smallskip
	\item\label{lemma:path_bound_modulus:ii}  Let  $\Gamma$ be a path family with $\md_n\Gamma>a$ for some $a>0$ such that each path $\gamma\in \Gamma$ intersects the ball $B(x,R)$. Then there exists a path $\gamma\in \Gamma$ with
\begin{align*}
\int_{\gamma}\, \rho \, ds \leq c(n,a) \left(\int_{B(x,c(n,a)R)} \rho^n \right)^{1/n}\quad \textrm{and}\quad \ell(\gamma)\leq c(n,a)R.
\end{align*}
\end{enumerate}
\end{lemma}

\begin{proof}
Both statements are conformally invariant. Hence, using a conformal transformation, we may assume that $x=0$ and $R=1$. For $M>1$, the family $\Gamma_M$ is contained in the union of the families
\begin{align*}
\Gamma_1&= \{ \gamma: \ell(\gamma)> M \,\,\, \textrm{and}\,\,\, |\gamma|\subset B(0,\sqrt{M})\} \\
\Gamma_2&=\{\gamma: |\gamma|\cap \partial B(0,1)\neq \emptyset \,\,\, \textrm{and}\,\,\, |\gamma|\cap \partial B(0,\sqrt{M})\neq \emptyset \}
\end{align*}
By the subadditivity of modulus, it suffices to show that $\md_n\Gamma_i$ converges to $0$ as $M\to\infty$ for $i=1,2$. The function $M^{-1}\x_{B(0,\sqrt{M})}$ is admissible for $\Gamma_1$, so $\md_n\Gamma_1\leq c(n)M^{-n/2}$. The modulus of $\Gamma_2$ is given by the explicit formula $\md_n\Gamma_2=c(n)( \log \sqrt{M})^{1-n}$;see property \ref{m:ring}. This proves part \ref{lemma:path_bound_modulus:i}.

Now we prove \ref{lemma:path_bound_modulus:ii}. Let $M=M(n,a)$ be sufficiently large, so that $\md_n\Gamma_M<a/2$. Define $\rho_1=\rho \x_{ B(0,M+1)}$ and let $\Gamma_{1}$ be the family of paths $\gamma\in \Gamma$ such that 
\begin{align*}
\int_{\gamma}\, \rho_1 \, ds >  2^{1/n}a^{-1/n} \|\rho_1\|_{L^n(\R^n)}.
\end{align*}
Then the function $$2^{-1/n}a^{1/n}\|\rho_1\|_{L^n(\R^n)}^{-1}\rho_1$$
is admissible for $\Gamma_1$, provided that $\|\rho_1\|_{L^n(\R^n)}\neq 0$, in which case we have $\md_n\Gamma_1 \leq a/2$. If $\|\rho_1\|_{L^n(\R^n)}=0$, then $\md_n\Gamma_1=0$ by property \ref{m:zero_ae}.  Also, let $\Gamma_2$ be the family of paths $\gamma\in \Gamma$ such that $\ell(\gamma)>M$, so $\md_n\Gamma_2\leq \md_n\Gamma_M <a/2$. By the subadditivity of modulus we have 
$\md_n(\Gamma_1\cup \Gamma_2)<a <\md_n\Gamma$. It follows that   $\Gamma\setminus (\Gamma_1\cup \Gamma_2)$ has positive modulus, and in particular it is non-empty. Thus, there exists a path $\gamma\in \Gamma$ with $\ell(\gamma)\leq M$ and 
$$\int_{\gamma}\rho_1\, ds \leq 2^{1/n}a^{-1/n} \|\rho_1\|_{L^n(\R^n)}.$$
Finally, note that $|\gamma|\subset  B(0,M+1)$ since $|\gamma|\cap B(0,1)\neq \emptyset$ and $\ell(\gamma)\leq M$. Thus,
\begin{align*}
\int_{\gamma}\rho\, ds =\int_{\gamma}\rho_1\, ds,
\end{align*}
which completes the proof, with $c(n,a)=\max\{ M(n,a)+1, 2^{1/n}a^{-1/n}\}$. 
\end{proof}

For a continuum $F\subset \R^n$ and $r>0$ we define $F^r$ to be the continuum $F+\br B(0,r)=\{x+y: x\in F, y\in \br B(0,r)\}$. 

\begin{lemma}\label{lemma:fattening}
Let $\mathcal F$ be a family of curve perturbations in $\R^n$. Then for every open set $U\subset \R^n$ and all pairs of non-empty, disjoint continua $F_1,F_2 \subset  U$ we have
\begin{align*}
\md_n (\Gamma(F_1,F_2;U) \cap \mathcal F)=\lim_{r\to 0} \md_n (\Gamma(F_1^r,F_2^r;U) \cap \mathcal F).
\end{align*}
\end{lemma}

Our proof relies on the properties of $P$-families, which is a new concept, but the main ideas originate in the proof of \cite{Vaisala:null}*{Lemma 2}. 

\begin{proof}Note that $\Gamma(F_1,F_2;U)\cap \mathcal F\subset \Gamma(F_1^r,F_2^r;U)\cap \mathcal F$ for every $r>0$, so one inequality is immediate. Also, if $F_i$ is a point $x_0$ for some $i=1,2$, then there exists $R>0$ such that by properties \ref{m:subordinate} and \ref{m:ring} we have
$$\md_n (\Gamma(F_1^r,F_2^r;U) \cap \mathcal F) \leq \md_n \Gamma(A(x_0;r,R))= c(n) \left(\log\frac{R}{r}\right)^{1-n}$$
for all sufficiently small $r>0$. Taking $r\to 0$, we obtain the desired conclusion.

We suppose that $\diam(F_i)>0$ for $i=1,2$. Let $\rho \in L^n(\R^n)$ be admissible for $\Gamma(F_1,F_2;U)\cap \mathcal F$. We will show that for each $q<1$ we have 
$$\int_{\gamma}\rho\, ds \geq q$$
for all sufficiently small $r>0$ and $\gamma\in \Gamma(F_1^r,F_2^r;U)\cap \mathcal F$. This will imply that $$\limsup_{r\to 0} \md_n (\Gamma(F_1^r,F_2^r;U) \cap \mathcal F) \leq q^{-n}\int \rho^n.$$
Letting $q\to 1$ and then infimizing over $\rho$ gives the desired
$$\limsup_{r\to 0} \md_n (\Gamma(F_1^r,F_2^r;U) \cap \mathcal F) \leq \md_n (\Gamma(F_1,F_2;U)\cap \mathcal F).$$

Arguing by contradiction, assume that there exists $0<q<1$ and  $r_k\to 0^+$ such that for each $k\in \N$ there exists a path $\gamma_{r_k}\in \Gamma(F_1^{r_k},F_2^{r_k};U)\cap \mathcal F$ with
$$\int_{\gamma_{r_k}}\rho\, ds < q<1.$$
By passing to a subpath, we assume in addition that $|\gamma_{r_k}|$ is disjoint from $F_1\cup F_2$; here we use property \ref{perturbations:3}. We fix $R_0>0$ such that $F_i^{R_0}\subset U$, $\diam(F_i)>2R_0$ for $i=1,2$, and $F_1^{R_0}\cap F_2^{R_0}=\emptyset$. For each  $r_k<R_0$ and $i=1,2$, there exists $x_{i,k}\in F_i$ such that $|\gamma_{r_k}|$ connects the boundary components of the ring $A_{i,k}=A(x_{i,k};r_k,R_0)$. Note that $F_i$ also connects the boundary components of the ring $A_{i,k}$, since $\diam(F_i)>\diam(A_{i,k})$. By passing to a further subpath, we assume in addition that the endpoints of $\gamma_{r_k}$ lie in the inner boundary components of $A_{i,k}$.

We fix $\varepsilon=(1-q)/2>0$. By Lemma \ref{lemma:loewner} we have that if $r_k<R_0/8$, then
\begin{align*}
\md_n (\Gamma( |\gamma_{r_k}|,F_i; A_{i,k}) \cap \mathcal F)\geq c(n) \log(R_0/r_k) \gtrsim_n 1.
\end{align*}
Lemma \ref{lemma:path_bound_modulus} \ref{lemma:path_bound_modulus:ii} implies that if $r_k$ is sufficiently small, depending on $\varepsilon$, then there exists a path $\gamma_i\in \Gamma( |\gamma_{r_k}|,F_i; A_{i,k}) \cap \mathcal F$ such that 
\begin{align*}
\int_{\gamma_i} \rho\, ds\leq c(n) \left(\int_{B(x_{i,k},c(n)r_k)} \rho^n \right)^{1/n} <\varepsilon.
\end{align*}
We concatenate $\gamma_i$, $i=1,2$, with a suitable subpath of $\gamma_{r_k}$; note that the endpoint of $\gamma_i$ that lies in $|\gamma_{r_k}|$ is not in $\partial \mathcal F$ because it is an interior point of a path of $\mathcal F$. By property \ref{perturbations:4}, the concatenation lies in $\mathcal F$.   In this way, we obtain a path $\gamma \in \Gamma( F_1,F_2;U)\cap \mathcal F$ such that 
$$\int_{\gamma} \rho\, ds \leq \int_{\gamma_{r_k}}\rho\, ds + \int_{\gamma_1}\rho\, ds+\int_{\gamma_2}\rho\, ds<q+2\varepsilon=1.$$  
This contradicts the admissibility of $\rho$. 
\end{proof}
\begin{remark}\label{remark:fattening}
From the proof we see that Lemma \ref{lemma:fattening} is valid more generally for families $\mathcal F$ satisfying \ref{perturbations:3}, \ref{perturbations:4}, and the conclusion of Lemma \ref{lemma:loewner} (or a variant of it, such as Lemma \ref{lemma:loewner_proj}, which uses rectangular instead of spherical rings). Recall that $\mathcal F_*(E)$ always satisfies \ref{perturbations:3} and \ref{perturbations:4}; see Remark \ref{remark:p3p4}.
\end{remark}

Our ultimate preliminary result before the proof of Theorem \ref{theorem:perturbation_family} is the following theorem, which is a version of the Lebesgue differentiation theorem for line integrals.

\begin{theorem}\label{theorem:lebesgue_differentiation_lp}
Let $\rho\colon \R^n\to [-\infty,\infty]$ be a Borel function with $\rho\in L^p_{\loc}(\R^n)$ for some $p>1$. Then there exists a path family $\Gamma_0$ with $\md_p\Gamma_0=0$ such that for every rectifiable path $\gamma\notin \Gamma_0$ we have $\int_{\gamma}|\rho|\, ds<\infty$ and 
\begin{align*}
\lim_{r\to 0} \fint_{B(0,r)} \int_{\gamma+x} \rho\, ds= \int_{\gamma}\rho\, ds.
\end{align*} 
\end{theorem}
\begin{proof}
By the subadditivity of modulus, it suffices to prove the statement for paths $\gamma$ contained in a compact set. Thus, we may assume that $\rho\in L^p(\R^n)$. For continuous functions $\rho$ with compact support the statement is trivially true for all rectifiable paths by uniform continuity. For the general case, for fixed $\lambda>0$ consider the  family  $\Gamma_\lambda$ of rectifiable paths $\gamma$ with $\int_{\gamma} |\rho|\, ds<\infty$ and
$$\limsup_{r\to 0} \left|  \fint_{B(0,r)}\int_{\gamma+x}\rho\, ds -\int_{\gamma} \rho\, ds \right| >\lambda.$$
It suffices to show that $\md_p\Gamma_\lambda=0$ for each $\lambda>0$. Let $\phi$ be a continuous function with compact support. Then,  $\Gamma_{\lambda}\subset \Gamma_1\cup \Gamma_2$, where $\Gamma_1$ is the family of rectifiable paths $\gamma$ with  
\begin{align*}
\limsup_{r\to 0}   \fint_{B(0,r)}\int_{\gamma+x}|\rho-\phi|\, ds >\lambda/2
\end{align*}
and $\Gamma_2$ is the family of rectifiable paths $\gamma$ with
\begin{align*}
\int_{\gamma} |\rho-\phi|\, ds>\lambda/2.
\end{align*}
We note that 
$$\md_p\Gamma_2\leq 2^p\lambda^{-p} \|\rho-\phi\|_{L^p(\R^n)}^p.$$
Moreover, if $\gamma$ is parametrized by arclength, we have
\begin{align*}
 \fint_{B(0,r)}\left(\int_{\gamma+x}|\rho-\phi|\, ds\right)dx = \int_{0}^{\ell(\gamma)} \left(\fint_{B(\gamma(t),r)}|\rho-\phi|\right)\, dt \leq \int_{\gamma}M(\rho-\phi)\, ds,
\end{align*}
where $M(\cdot )$ denotes the centered Hardy--Littlewood maximal function. Hence,
\begin{align*}
\int_{\gamma} M(\rho-\phi)\, ds >\lambda/2
\end{align*}
for $\gamma\in \Gamma_1$. The Hardy--Littlewood maximal $L^p$-inequality \cite{Ziemer:Sobolev}*{Theorem 2.8.2, p.~84} implies that 
\begin{align*}
\md_p\Gamma_1 \leq 2^p \lambda^{-p} \|M(\rho-\phi)\|^p_{L^p(\R^n)} \leq c(n,p) 2^p\lambda^{-p} \|\rho-\phi\|_{L^p(\R^n)}^p.
\end{align*}
Altogether,
\begin{align*}
\md_p\Gamma_{\lambda}\leq \md_p\Gamma_1+\md_p\Gamma_2\lesssim_{n,p,\lambda} \|\rho-\phi\|_{L^p(\R^n)}^p
\end{align*}
Since $\phi$ was arbitrary, we conclude that $\md_p\Gamma_{\lambda}=0$, as desired.
\end{proof}

\begin{proof}[Proof of Theorem \ref{theorem:perturbation_family}]
By the monotonicity of modulus, it suffices to prove that 
\begin{align*}
\md_n \Gamma (F_1,F_2;U) \leq  \md_n (\Gamma(F_1,F_2;U) \cap \mathcal F).
\end{align*}
By Lemma \ref{lemma:fattening}, it suffices to prove that 
\begin{align*}
\md_n \Gamma (F_1,F_2;U) \leq  \md_n (\Gamma(F_1^r,F_2^r;U) \cap \mathcal F)
\end{align*}
for all sufficiently small $r>0$. We fix $r>0$ so that $F_1^r,F_2^r\subset U$. Let $\rho\colon \R^n \to [0,\infty]$ be an admissible function for $\Gamma(F_1^r,F_2^r;U) \cap \mathcal F$ with $\rho\in L^n(\R^n)$. Consider the curve family $\Gamma_0$ with $\md_n\Gamma_0=0$, given by Theorem \ref{theorem:lebesgue_differentiation_lp} and corresponding to $\rho$. Let $\gamma\in \Gamma(F_1,F_2;U)\setminus \Gamma_0$ be a rectifiable path. Since $\mathcal F$ is a family of curve perturbations, by \ref{perturbations:1}, for a.e.\ $x\in B(0,r)$ we have $\gamma+x\in  \Gamma(F_1^r,F_2^r;U) \cap \mathcal F$. Now, the admissibility of $\rho$ for $\Gamma(F_1^r,F_2^r;U) \cap \mathcal F$ and Theorem \ref{theorem:lebesgue_differentiation_lp} imply that $\int_{\gamma}\rho\, ds\geq 1$, so $\rho$ is admissible for $\Gamma(F_1,F_2;U)\setminus \Gamma_0$. Therefore,
\[
\md_n \Gamma (F_1,F_2;U)=\md_n (\Gamma (F_1,F_2;U)\setminus \Gamma_0)\leq \md_n (\Gamma(F_1^r,F_2^r;U) \cap \mathcal F).
\qedhere\]
\end{proof}

\subsection{Examples of families of curve perturbations}\label{section:perturbation:examples}
The next theorem, combined with Theorem \ref{theorem:perturbation_family}, gives Theorem \ref{theorem:cned} \ref{item:cned:ii} and Theorem \ref{theorem:zero}.

\begin{theorem}\label{theorem:perturbation_hausdorff}
Let $E\subset \R^n$ be a set.
\begin{enumerate}[\upshape(i)]
	\item\label{theorem:perturbation_hausdorff:i} If $\h^{n-1}(E)=0$, then $\mathcal F_0(E)$ is a $P$-family.\smallskip
	\item\label{theorem:perturbation_hausdorff:ii} If $E$ has $\sigma$-finite Hausdorff $(n-1)$-measure, then $\mathcal F_{\sigma}(E)$ is a $P$-family. 
\end{enumerate} 
\end{theorem}

The case of Hausdorff $(n-1)$-measure zero, as in \ref{theorem:perturbation_hausdorff:i}, has already been considered by V\"ais\"al\"a \cite{Vaisala:null}*{Lemma 5}, where it is proved that for a.e.\ $x\in \R^n$ we have $\gamma+x\in \mathcal F_0(E)$; that is, \ref{perturbations:1} is satisfied. Recall also that \ref{perturbations:3} and \ref{perturbations:4} are always satisfied for $\mathcal F_0(E)$ and $\mathcal F_{\sigma}(E)$; see Remark \ref{remark:p3p4}. We first prove two preliminary lemmas that will be used in the proof of both cases  \ref{theorem:perturbation_hausdorff:i} and \ref{theorem:perturbation_hausdorff:ii} of the theorem. 

\begin{lemma}\label{lemma:p2}
Let $E\subset \R^n$ and $\gamma$ be a non-constant rectifiable path. For $N\in \N$, let $F_N$ be the set of $x\in \R^n$ such that $E\cap |\gamma+x|$ contains at least $N$ points. Then 
\begin{align*}
m_n^*(F_1)&\leq c(n) \max\{\ell(\gamma),\diam(E)\} \diam(E)^{n-1}\quad \textrm{and}\\
m_n^*(F_N) &\leq c(n) \ell(\gamma)N^{-1} \h^{n-1}(E).
\end{align*}
\end{lemma}
\begin{proof}
First we show the second inequality. For $k\in \N$ we define $F_{N,k}$ to be the set of $x\in F_N$ such that $E\cap |\gamma+x|$ contains $N$ points with mutual distances bounded below by $1/k$.  We have $F_{N,k+1} \supset F_{N,k}$, $F_N=\bigcup_{k=1}^\infty F_{N,k}$, and $m_n^*(F_N)=\lim_{k\to\infty}m_{n}^* (F_{N,k})$; see \cite{Bogachev:measure}*{Proposition 1.5.12}. We estimate $m_n^*(F_{N,k})$. We fix a large $k\in \N$ so that $\frac{1}{2k}<{\diam(|\gamma|)}/{4}$. Consider an arbitrary cover of $E$ by sets $U_i$, $i\in I$, with $\diam(U_i)<\frac{1}{2k}$. For each $i\in I$ there exists a closed ball $B_i\supset U_i$ of radius $r_i=\diam(U_i)<\frac{1}{2k}<\diam(|\gamma|)/4$. Define the function 
$$\rho= \frac{1}{N} \sum_{i\in I} \frac{1}{r_i} \x_{2B_i}.$$
If $x\in F_{N,k}$, then $|\gamma+x|$ intersects at least $N$ balls $B_i$ and is not contained in any ball $2B_i$. Therefore,
\begin{align*}
\int_{\gamma+x}\rho\, ds  \geq 1.
\end{align*} 
By Chebychev's inequality and Fubini's theorem, we have
\begin{align*}
m_n^*(F_{N,k}) &\leq\ell(\gamma) \| \rho\|_{L^1(\R^n)}\simeq_n \ell(\gamma) N^{-1} \sum_{i\in I}\diam(U_i)^{n-1}.
\end{align*}
The cover of $E$ by  sets $U_i$, $i\in I$, of diameter less than $\frac{1}{2k}$ was arbitrary, so
$$m_n^*(F_{N,k}) \lesssim_n \ell(\gamma)N^{-1} \mathscr H^{n-1}_{(2k)^{-1}}(E).$$
Letting $k\to \infty$ gives
\[m_n^*(F_N) \lesssim_n \ell(\gamma)N^{-1} \mathscr H^{n-1}(E).\]

For the first inequality, consider two cases. If $\diam(|\gamma|)>4\diam(E)$, then we cover $E$ by a closed ball $B$ of radius $r$ with $0\leq \diam(E)<r<\diam(|\gamma|)/4$. If $x\in F_1$, then $|\gamma+x|$ intersects $B$ and is not contained in $2B$. The above argument for $N=1$ gives $m_n^*(F_1)\lesssim_n \ell(\gamma) r^{n-1}$. Now, we let $r\to \diam(E)$ to obtain $m_n^*(F_1)\lesssim_n \ell(\gamma)\diam(E)^{n-1}$. Next, assume that $\diam(|\gamma|)\leq 4\diam(E)$. In this case, if $x\in F_1$, then $|\gamma+x|\subset \br B(x_0,5\diam(E))$ for a fixed $x_0\in E$. Thus, $m_n^*(F_1)\leq m_n (\br B(x_0,5\diam(E))) =c(n) \diam(E)^n$.
\end{proof}

\begin{lemma}\label{lemma:p3}
Let $E\subset \R^n$, $x\in \R^n$, $r>0$, and for $w\in S^{n-1}(0,1)$ define $\gamma_w(t)=x+tw$, $r/2\leq t\leq r$. For $N\in \N$, let $F_N$ be the set of $w\in S^{n-1}(0,1)$ such that $E\cap |\gamma_w|$ contains at least $N$ points. Then 
\begin{align*}
\h^{n-1}(F_1)&\leq c(n)  r^{-n+1} \min\{r,\diam(E)\}^{n-1}\quad \textrm{and}\\
\h^{n-1}(F_N) &\leq c(n) r^{-n+1}N^{-1} \h^{n-1}(E).
\end{align*}
\end{lemma}
\begin{proof}
For the first inequality, note that $F_1+x$ is equal to the radial projection of $E\cap \{y\in \R^n: r/2\leq |x-y|\leq r\}$ to the sphere $S^{n-1}(x,1)$. This projection is the composition of a uniformly Lipschitz map (projection of $\{y\in \R^n: r/2\leq |x-y|\leq r\}$ to $S^{n-1}(x,r)$) with a scaling by $1/r$. Thus, 
\begin{align*}
\diam(F_1)\lesssim r^{-1} \diam(E\cap \{y\in \R^n: r/2\leq |x-y|\leq r\}) \lesssim r^{-1}\min\{r,\diam(E)\}.
\end{align*}
Moreover, $F_1$ is contained in the intersection of a ball $B_0=\br B(x_0,\diam(F_1))$, where $x_0\in F_1$, with $S^{n-1}(0,1)$. Thus,
\begin{align*}
\h^{n-1}(F_1)\leq \h^{n-1}(B_0\cap S^{n-1}(0,1)) \simeq_n \diam(F_1)^{n-1}.
\end{align*}
This completes the proof of the first inequality.

For the second inequality, we proceed as in the proof of Lemma \ref{lemma:p2}, by defining $F_{N,k}$ to be the set of $w\in S^{n-1}(0,1)$ such that $E\cap |\gamma_w|$ contains $N$ points with mutual distances bounded below by $1/k$. We define the function $\rho$ exactly as in Lemma \ref{lemma:p2}, using an arbitrary cover of $E$ by sets $U_i$ and corresponding balls $B_i\supset U_i$ with $r_i=\diam(U_i)<\frac{1}{2k}<\frac{r}{8}$. Note that if $w\in F_{N,k}$, then 
$$\int_{\gamma_w}\rho\, ds =\int_{r/2}^r \rho( x+tw)\, ds \geq 1.$$
By Chebychev's inequality and polar integration, it follows that 
\begin{align*}
\h^{n-1}( F_{N,k}) &\leq \int_{S^{n-1}(0,1)} \int_{r/2}^r \rho( x+tw)\, dt d\h^{n-1}(w)\\
&\lesssim_n r^{-n+1} \|\rho\|_{L^1(\R^n)} \simeq_n r^{-n+1}  N^{-1} \sum_{i\in I}\diam(U_i)^{n-1}.
\end{align*}
We now proceed as before, infimizing over the covers of $E$ and letting $k\to \infty$. 
\end{proof}

\begin{proof}[Proof of Theorem \ref{theorem:perturbation_hausdorff}]
By Remark \ref{remark:p3p4}, \ref{perturbations:3} and \ref{perturbations:4} are automatically satisfied for $\mathcal F_0(E)$ and $\mathcal F_{\sigma}(E)$. We will establish below properties \ref{perturbations:1} and \ref{perturbations:2}.

Suppose first that $\h^{n-1}(E)=0$ as in \ref{theorem:perturbation_hausdorff:i}. If $\gamma$ is a non-constant rectifiable path, by the second inequality of Lemma \ref{lemma:p2} (for $N=1$) we have that $E\cap |\gamma+x|=\emptyset$, i.e., $\gamma+x\in \mathcal F_0(E)$, for a.e.\ $x\in \R^n$. Hence, \ref{perturbations:1} holds. 

Next, if $x\in \R^n$, $r>0$, and $w\in S^{n-1}(0,1)$, define $\gamma_w(t)=x+tw$, $0\leq t\leq r$. By applying Lemma \ref{lemma:p3} countably many times (for $N=1$) to the paths $\gamma_w|_{[2^{-k}r,2^{-k+1}r ]}$, we have $E\cap \gamma_w([2^{-k}r,2^{-k+1}r ])=\emptyset$ for all $k\in \N$ and for a.e.\ $w\in S^{n-1}(0,1)$. Hence, $E\cap \gamma_w( (0,r])=\emptyset$ for a.e.\ $w\in S^{n-1}(0,1)$. Recall that a path in $\mathcal F_0(E)$, by definition, is allowed to intersect $E$ only at the endpoints. Hence, $\gamma_w\in \mathcal F_0(E)$ for a.e.\ $w\in S^{n-1}(0,1)$. This completes the proof of \ref{perturbations:2} and of  \ref{theorem:perturbation_hausdorff:i}.
    
Next, we suppose that $E$ has $\sigma$-finite Hausdorff $(n-1)$ measure, as in \ref{theorem:perturbation_hausdorff:ii}. We write $E=\bigcup_{k=1}^\infty E_k$, where $\h^{n-1}(E_k)<\infty$ for each $k\in \N$. If we show that $\mathcal F_{\sigma}(E_k)$ is a $P$-family for each $k\in \N$, then $\mathcal F_{\sigma}(E)$ will also be a $P$-family, since the intersection of countably many $P$-families is a $P$-family by Lemma \ref{lemma:intersection}. Hence, for \ref{theorem:perturbation_hausdorff:ii} it suffices to assume that $\h^{n-1}(E)<\infty$. 

Let $\gamma$ be a non-constant rectifiable path. We define $F$ to be the set of $x\in \R^n$ such that $E\cap |\gamma+x|$ is infinite and consider the set $F_N$ as in Lemma \ref{lemma:p2}. Observe that $F=\bigcap_{N=1}^\infty F_N$.  Since 
$$m_n^*(F_N) \lesssim_n \ell(\gamma)N^{-1} \h^{n-1}(E),$$
by letting $N\to \infty$ we obtain $m_n(F)=0$. This proves  \ref{perturbations:1}. 

For \ref{perturbations:2} we fix $x\in \R^n$, $r>0$, and for $w\in S^{n-1}(0,1)$ consider the segment $\gamma_w(t)=x+tw$, $0\leq t\leq r$, as above. For fixed $k\in \N$ we apply Lemma \ref{lemma:p3} to the paths $\gamma_w|_{[2^{-k}r,2^{-k+1}r ]}$ and conclude (by letting $N\to \infty)$ that the set $E\cap \gamma_w([2^{-k}r,2^{-k+1}r ])$ is finite for a.e.\ $w\in S^{n-1}(0,1)$. Hence, for a.e.\ $w\in S^{n-1}(0,1)$ the set $E\cap |\gamma_w|$ is  countable, i.e., $\gamma_w\in \mathcal F_{\sigma}(E)$. 
\end{proof}

\section{Criteria for negligibility}\label{section:criteria}

In this section we prove criteria for the membership of a set $E$ in $\NED$ or $\CNED$. The first of these criteria is crucially used in the proof of Theorem \ref{theorem:unions} regarding the unions of $\NED$ and $\CNED$ sets. Recall that $\*ned$ denotes either $\NED$ or $\CNED$. Also, recall from Section \ref{section:elementary} the definitions of $\*ned(\Omega)$ and $\*ned^w(\Omega)$, and the definition of the relative distance $\Delta(F_1,F_2)$ of two sets $F_1,F_2\subset \R^n$. Let $\gamma\colon [a,b]\to \R^n$ be a non-constant path. If $[c,d]\subset (a,b)$, then the strong subpath $\gamma|_{[c,d]}$ of $\gamma$ is called \textit{strict}.

\begin{theorem}[Main criterion]\label{theorem:criterion_compact}
Let $E\subset \R^n$ be a set such that either $E$ is closed or $m_n(\br E)=0$. The following are equivalent.
\begin{enumerate}[\upshape(I)]
	\item\label{cc:i} $E\in \*ned(\Omega)$ for all open sets $\Omega\subset \R^n$.
	\item\label{cc:ii} $E\in \*ned$.
	\item\label{cc:iii} $E\in \*ned^w(\Omega)$ for some open set $\Omega\subset \R^n$ with $\Omega\supset \br E$.
	\item\label{cc:loewner}There exist constants $t,\phi>0$ such that for each $x_0\in \R^n$ there exists $r_0>0$ with the property that for every pair of non-degenerate, disjoint continua $F_1,F_2\subset B(x_0,r_0)$ with $\Delta(F_1,F_2)\leq t$ we have
\begin{align*}
\md_n(\Gamma(F_1,F_2;\R^n) \cap \mathcal F_*(E)) \geq \phi.
\end{align*}
	\item\label{cc:iv} For each Borel function $\rho\colon \R^n\to [0,\infty]$ with $\rho\in L^n_{\loc}(\R^n)$ there exists a path family $\Gamma_0$ with $\md_n\Gamma_0=0$ such that Conclusion \ref{conclusion:a} below holds for each rectifiable path $\gamma\notin \Gamma_0$ with distinct endpoints. Moreover, $\Gamma_0$ has the property that if $\{\eta_j\}_{j\in J}$ is a finite collection of paths outside $\Gamma_0$ and $\gamma$ is a path with $|\gamma|\subset \bigcup_{j\in J} |\eta_j|$, then $\gamma\notin \Gamma_0$.
	\item\label{cc:v} For each Borel function $\rho\colon \R^n\to [0,\infty]$ with $\rho\in L^n_{\loc}(\R^n)$ there exists a path family $\Gamma_0$ with $\md_n\Gamma_0=0$ such that Conclusion \ref{conclusion:b} below holds for each rectifiable path $\gamma\notin \Gamma_0$ with distinct endpoints.
\end{enumerate}	
Moreover, the following implications 
are true for all sets $E\subset \R^n$.
\begin{align*}
\textrm{\ref{cc:iv} $\Rightarrow$ \ref{cc:v}  $\Rightarrow$ \ref{cc:i} $\Rightarrow$ \ref{cc:ii} $\Rightarrow$ \ref{cc:iii} $\Rightarrow$ \ref{cc:loewner} } 
\end{align*}
	
\begin{conclusion}{A}[A$(E,\rho,\gamma)$]\label{conclusion:a}
For each open neighborhood $U$ of $|\gamma|\setminus \partial\gamma$  and for each $\varepsilon>0$ there exists a collection of paths $\{\gamma_i\}_{i\in I}$ and a simple path $\widetilde\gamma$ such that
\begin{enumerate}[\upshape({A}-i)]\smallskip
	\item\label{criterion:i} $\widetilde \gamma \in  \mathcal F_*(E)$,\smallskip
	\item\label{criterion:ii} $\partial \widetilde \gamma=\partial \gamma$, $\displaystyle{|\widetilde \gamma|\setminus \partial \gamma\subset {(|\gamma|\setminus  E)}\cup \bigcup_{i\in I} |\gamma_i| }$, and $\bigcup_{i\in I} |\gamma_i|\subset U$,  
	\item\label{criterion:iii} $\displaystyle{\sum_{i\in I}\ell(\gamma_i)<\varepsilon}$, and \smallskip
	\item\label{criterion:iv} $\displaystyle{\sum_{i\in I} \int_{\gamma_i}\rho\, ds <\varepsilon.}$
\end{enumerate}
The paths $\gamma_i$, $i\in I$, may be taken to lie outside a given path family $\Gamma'$ with $\md_n\Gamma'=0$.  If $\br E\cap \partial \gamma=\emptyset$, then $I$ may be taken to be finite. In general, the trace of each strict subpath of $\widetilde \gamma$ intersects finitely many traces $|\gamma_i|$, $i\in I$.
\end{conclusion}	
	
\begin{conclusion}{B}[B$(E,\rho,\gamma)$]\label{conclusion:b}
For each open neighborhood $U$ of $|\gamma|$ and for each $\varepsilon>0$ there exists a simple path $\widetilde \gamma$ such that
\begin{enumerate}[\upshape({B}-i)]
	\item\label{criterion:bi} $\widetilde \gamma\in \mathcal F_*(E)$,\smallskip
	\item\label{criterion:bii} $\partial \widetilde \gamma=\partial \gamma$ and $|\widetilde \gamma|\subset U$,\smallskip
	\item\label{criterion:biii} $\ell(\widetilde \gamma)\leq \ell(\gamma)+\varepsilon$, and  \smallskip
	\item\label{criterion:biv} $\displaystyle{\int_{\widetilde \gamma} \rho\, ds \leq \int_{\gamma} \rho\, ds +\varepsilon.}$
\end{enumerate}
\end{conclusion}

\end{theorem}
Note that the implications \ref{cc:i} $\Rightarrow$ \ref{cc:ii} $\Rightarrow$ \ref{cc:iii} are trivial. Moreover,  \ref{cc:iii} $\Rightarrow$ \ref{cc:loewner} follows immediately from Lemma \ref{lemma:measure_zero}. Conclusion \ref{conclusion:b} in \ref{cc:v} is only a less technical statement  that follows easily from Conclusion \ref{conclusion:a} in \ref{cc:iv}. Indeed, \ref{criterion:biii} and \ref{criterion:biv} follow from \ref{criterion:ii}, \ref{criterion:iii}, and \ref{criterion:iv}, using Lemma \ref{lemma:paths} \ref{lemma:paths:iv}.  Hence, we will show implications \ref{cc:loewner} $\Rightarrow$ \ref{cc:iv}, which is the most technical one,  and \ref{cc:v} $\Rightarrow$ \ref{cc:i}.

Roughly speaking, Conclusions \ref{conclusion:a} and \ref{conclusion:b} say that with small cost we can alter the path $\gamma$ to bring it inside the curve family $\mathcal F_*(E)$. The assumption that $E$ is closed or $m_n(\br E)=0$ will be crucially used in the proof of \ref{cc:loewner} $\Rightarrow$ \ref{cc:iv} (see Lemma \ref{lemma:exceptional_family}) and it is doubtful whether this implication holds without that assumption.  

An immediate corollary of Theorem \ref{theorem:criterion_compact} is the quasiconformal and bi-Lipschitz invariance of compact $\*ned$ sets.
\begin{corollary}\label{corollary:qc_invariance}
Let $E\subset \R^n$ be a bounded set such that either $E$ is closed or $m_n(\br E)=0$. Let $\Omega\subset \R^n$ be an open set with $\br E\subset \Omega$, and $f\colon \Omega\to \R^n$ be a quasiconformal embedding. If $E\in \*ned$, then $f(E)\in \*ned$. 
\end{corollary}
\begin{proof}
Under either assumption, we have $m_n(\br E)=0$ by Lemma \ref{lemma:measure_zero}. Observe that $f(\br E)=\br{f(E)}$ and that this is a compact subset of $f(\Omega)$ having $n$-measure zero by quasiconformality.  Since $f$ distorts the $n$-modulus of each curve family in $\Omega$ by a bounded multiplicative factor, we see that $f(E)\in \*ned^w (f(\Omega))$.  By Theorem \ref{theorem:criterion_compact}, we conclude that $f(E)\in \*ned$. 
\end{proof}

We also prove an alternative criterion in terms of $P$-families; recall the definition of a $P$-family from Section \ref{section:perturbation}. The result is also true for the case of $\NED$ sets but we do not prove it here to avoid some technicalities.

\begin{theorem}[$P$-family criterion]\label{theorem:criterion_pfamily}
Let $E\subset \R^n$ be a set such that either $E$ is closed or $m_n(\br E)=0$. The following are equivalent.
\begin{enumerate}[\upshape(I)]
	\item\label{cp:i} $E\in \CNED$.
	\item\label{cp:ii} For each Borel function $\rho\colon \R^n\to [0,\infty]$ with $\rho\in L^n(\R^n)$ the following statements are true.
		\begin{enumerate}[\upshape({II}-1)]
			\item\label{cp:ii:i} For each rectifiable path $\gamma$, a.e.\ $x\in \R^n$, and every strong subpath $\eta$ of $\gamma+x$ with distinct endpoints, Conclusion \ref{conclusion:b}$(E,\rho,\eta)$ is true.
			\item\label{cp:ii:ii} For $x\in \R^n$, $0<r<R$, and $w\in S^{n-1}(0,1)$ define $\gamma_{w}(t)=x+tw$, $t\in [r,R]$. Then for $\h^{n-1}$-a.e.\ $w\in S^{n-1}(0,1)$,  and for every strong subpath $\eta$ of $\gamma_w$, Conclusion \ref{conclusion:b}$(E,\rho,\eta)$ is true.
		\end{enumerate}
	\item\label{cp:iii} For each Borel function $\rho\colon \R^n\to [0,\infty]$ with $\rho\in L^n(\R^n)$ there exists a $P$-family $\mathcal F$ such that if $\gamma\in \mathcal F$ is a rectifiable path, then  Conclusion \ref{conclusion:b}$(E,\rho,\eta)$ is true for each strong subpath $\eta$ of $\gamma$ with distinct endpoints.
\end{enumerate}
\end{theorem}

Ahlfors--Beurling  \cite{AhlforsBeurling:Nullsets}*{Theorem 10} proved that if a closed set $E\subset \R^n$ is $\NED$ then any two points in $\R^n\setminus E$ can by joined by a curve in $\R^n\setminus E$ of short length. The analogous statement is true for closed $\CNED$ sets.

\begin{corollary}
Let $E\subset \R^n$ be a closed set with $E\in \CNED$. Then for every $\varepsilon>0$ and points $x,y\in \R^n$ there exists a path $\gamma \in \mathcal F_\sigma(E)$ connecting $x$ and $y$ with $\ell(\gamma)\leq  |x-y|+\varepsilon$.
\end{corollary}

\begin{proof}
Let $x,y\in \R^n$ be distinct points, $\gamma$ be the line segment $[x,y]$, and $\varepsilon>0$. Let $\rho \equiv 0$ and consider the $P$-family $\mathcal F$ given by Theorem \ref{theorem:criterion_pfamily} \ref{cp:iii}. By property \ref{perturbations:1} of the $P$-family $\mathcal F$, for a.e.\ $z\in \R^n$ the path $\gamma+z$ lies in $\mathcal F$. Using spherical coordiantes we see that for a.e.\ $r>0$ and for $\h^{n-1}$-a.e.\ $w\in S^{n-1}(0,1)$ the above is true for $z=rw$. We fix  $r<\varepsilon/5$ such that this is true. Using property \ref{perturbations:2} of a $P$-family, for $\h^{n-1}$-a.e.\ $w\in S^{n-1}(0,1)$ the radial segments $\gamma_w^x(t)= x+tw$ and $\gamma_w^y(t)= y+tw$, $0\leq t\leq r$, lie in $\mathcal F$. Thus, there exists $w\in S^{n-1}(0,1)$ such that both radial segments lie in $\mathcal F$ and $\gamma+rw\in \mathcal F$. We now apply Conclusion \ref{conclusion:b} \ref{criterion:biii} (with $\varepsilon=r$) to each of $\gamma_w^x$, $\gamma_w^y$, and $\gamma+rw$, to conclude that there exist paths $\eta_x,\eta_y,\eta\in\mathcal F_{\sigma}(E)$ with the same endpoints as $\gamma_w^x,\gamma_w^y,\gamma+rw$, respectively, such that $\ell(\eta_x) \leq 2r$, $\ell(\eta_y)\leq 2r$, and $\ell(\gamma+rw) \leq |x-y|+r$. Concatenating these paths gives a path in $\mathcal F_{\sigma}(E)$ connecting $x$ and $y$ with length bounded by $|x-y|+ 5r<|x-y|+\varepsilon$.
\end{proof}

\subsection{Auxiliary results}

We will need some auxiliary results before we prove Theorems \ref{theorem:criterion_compact} and \ref{theorem:criterion_pfamily}. The following lemma is elementary.

\begin{lemma}\label{claim:lambda}
Let $E\subset \R^n$ be a compact set with $m_n(E)=0$, $g\colon \R^n\to [0,\infty]$ be a function in $L^1(\R^n)$, and $\lambda>0$. For each $\varepsilon>0$ there exists $\delta>0$ such that if $0<r<\delta$ and $B_i=B(x_i,r)$,   $i\in \{1,\dots,N\}$, is a family of pairwise disjoint balls centered at $E$, then 
\begin{align*}
\sum_{i=1}^N\int_{\lambda B_i} g<\varepsilon.
\end{align*}
\end{lemma}

\begin{proof}
Let $\lambda,\varepsilon>0$. Since $E$ is compact with $m_n(E)=0$ and $g\in L^1(\R^n)$, there exists $\delta>0$ such that
\begin{align}\label{claim:ineq}
\int_{N_{\lambda \delta}(E)}g<C(n,\lambda)^{-1}\varepsilon 
\end{align}
for a constant $C(n,\lambda)>0$ to be determined. Let $0<r<\delta$ and consider a family of finitely many disjoint balls $B_i=B(x_i,r)$, $i\in \{1,\dots,N\}$, centered at $E$. Suppose that $\lambda>1$. If $x\in \lambda B_i$, then $B_i\subset \lambda B_i\subset B(x,2\lambda r)$. Since the balls $B_i$, $i\in \{1,\dots,N\}$, are disjoint, by volume comparison we see that 
\begin{align*}
\sum_{i=1}^N \x_{\lambda B_i }\leq 2^n\lambda^n \x_{\bigcup_{i=1}^N \lambda B_i} \leq 2^n\lambda^n \x_{N_{\lambda r }(E)}.
\end{align*}
The same inequality is trivially true when $0<\lambda\leq 1$ with constant $1$ in place of $2^n\lambda^n$. We now set $C(n,\lambda)=\max\{1,2^n\lambda^n\}$ and by \eqref{claim:ineq} we  obtain
\[
\sum_{i=1}^N \int_{\lambda B_i} g= \int g \sum_{i=1}^N \x_{\lambda  B_i}\leq C(n,\lambda) \int_{N_{\lambda r}(E)} g<\varepsilon.\qedhere
\]
\end{proof}

\begin{lemma}\label{lemma:exceptional_family}
Let $E\subset \R^n$ be a closed set with $m_n(E)=0$. Then for each non-negative function $\rho\in L^n_{\loc}(\R^n)$  {and for each $\lambda>0$}
\begin{enumerate}[\upshape(i)]
	\item there exists a path family $\Gamma_0$ with $\md_n\Gamma_0=0$ and
	\item\label{lemma:exceptional_family_radii} for each $m\in \N$ there exists a family of balls $\{B_{i,m}=B(x_{i,m}, r_{i,m})\}_{i\in I_m}$ covering $E$ with $r_{i,m}<1/m$, $i\in I_m$,
\end{enumerate}
such that for every non-constant curve $\gamma\notin \Gamma_0$ we have
\begin{align}
\begin{aligned}\label{lemma:exceptional_family_limits}
&\lim_{m\to\infty} \sum_{i: B_{i,m}\cap |\gamma|\neq \emptyset} r_{i,m} \left(\fint_{{\lambda B_{i,m}}} \rho ^n \right)^{1/n}=0\quad \textrm{and}\\
&\lim_{m\to\infty} \sum_{i: B_{i,m}\cap |\gamma|\neq \emptyset} r_{i,m} =0.
\end{aligned}
\end{align}
\end{lemma}

\begin{proof}
First, we reduce to the case that $E$ is compact. Suppose that  the statement is true for compact sets. For $k\in \N$, let $A_k= \{x\in \R^n: k-1\leq |x|\leq k\}$. Then for each $k\in \N$, there exists a curve family $\Gamma_k$ with $\md_n\Gamma_k=0$ and for each $m\in \N$ there exists a family of balls $\{B_{i,m}\}_{i\in I_{m,k}}$ covering $E\cap A_k$, with radii less than $1/m$, so that \eqref{lemma:exceptional_family_limits} is true for non-constant paths $\gamma\notin \Gamma_k$. We discard the balls not intersecting $E\cap A_k$, if any. Let $I_m=\bigcup_{k\in \N}I_{m,k}$ and  $\Gamma_0=\bigcup_{k\in \N}\Gamma_k$. Note that $\md_n\Gamma_0=0$ by the subadditivity of modulus. We claim that \eqref{lemma:exceptional_family_limits} is true for the balls $\{B_{i,m}\}_{i\in I_m}$. 

If $\gamma$ is a non-constant path outside $\Gamma_0$, then $\gamma$ is contained in the union of finitely many sets $A_k$, $k\in \N$. Moreover, there exists $k_0\in \N$ such that $B_{i,m}\cap |\gamma|= \emptyset$ for all $i\in I_{m,k}$, $m\in \N$, and $k>k_0$.  Thus, in \eqref{lemma:exceptional_family_limits}, the sums over the indices $i\in I_{m,k}$ such that $B_{i,m}\cap |\gamma|\neq \emptyset$ are zero for all $k>k_0$. For $k\leq k_0$, the limits of the corresponding sums vanish, since $\gamma\notin \Gamma_k$. Since limits commute with finite sums, we obtain \eqref{lemma:exceptional_family_limits} for the family of balls $\{B_{i,m}\}_{i\in I_m}$.

Suppose now that $E$ is compact. Let $g\in L^n(\R^n)$ be a non-negative function, to be specified later. By the $5B$-covering lemma (\cite{Heinonen:metric}*{Theorem 1.2}) for each $r>0$ we can find a cover of $E$ by finitely many balls of radius $r$ centered at $E$ so that the balls with the same centers and radius $r/5$ are disjoint. Combining this with Lemma \ref{claim:lambda} (for $5\lambda$ in place of $\lambda$), for each $m\in \N$ we may find a cover of $E$ by balls $B_{i,m}=B(x_{i,m},r_m)$, $i\in I_m$, centered at $E$, such that $r_m<1/m$, the balls $\frac{1}{5}B_{i,m}$ are disjoint, and
\begin{align*}
\sum_{i\in I_m} \int_{\lambda B_{i,m}}g^n <\frac{1}{2^m}.
\end{align*}
For $m\in \N$, we define the Borel function 
$$h_m=  \sum_{i\in I_m }\left(\fint_{\lambda B_{i,m}} g ^n \right)^{1/n} \x_{2 B_{i,m}}.$$
By Lemma \ref{lemma:bojarski}, we have
\begin{align*}
\sum_{m\in \N}\|h_m\|_{L^n(\R^n)} &\lesssim_n \sum_{m\in \N} \left\|\sum_{i\in I_m }\left(\fint_{\lambda B_{i,m}} g^n \right)^{1/n} \x_{\frac{1}{5}B_{i,m}} \right\|_{L^n(\R^n)} \\
&\lesssim_{n,\lambda}  \sum_{m\in \N} \left(\sum_{i\in I_m }\int_{\lambda B_{i,m}} g^n \right)^{1/n}\lesssim_{n,\lambda}  \sum_{m\in \N}\frac{1}{2^{m/n}}<\infty.
\end{align*}
By a variant of Fuglede's lemma  \cite{Vaisala:quasiconformal}*{Theorem 28.1}, there exists a curve family $\Gamma_0$ with $\md_n\Gamma_0=0$ such that for each path $\gamma\notin \Gamma_0$ we have
\begin{align*}
\lim_{m\to\infty}\int_{\gamma} h_m\, ds=0.
\end{align*}
Implicitly we assume that $\Gamma_0$ contains all curves that are not rectifiable.

Note that if $\gamma$ is a non-constant rectifiable curve, $B_{i,m}\cap |\gamma|\neq \emptyset$, and $m$ is sufficiently large so that $\diam(|\gamma|)> 4m^{-1}>4r_{m}$, then $|\gamma|$ is not contained in $2B_{i,m}$, so
$$\int_{\gamma}\x_{2B_{i,m}} \, ds \geq r_{m}.$$
Thus, 
$$\int_{\gamma} h_m\, ds  \geq \sum_{i: B_{i,m}\cap |\gamma|\neq \emptyset} r_{m} \left(\fint_{\lambda B_{i,m}} g ^n \right)^{1/n}.$$
We conclude that if $\gamma$ is a non-constant  curve outside $\Gamma_0$, then 
\begin{align}\label{lemma:exceptional:family:g}
\lim_{m\to\infty}\sum_{i: B_{i,m}\cap |\gamma|\neq \emptyset} r_{m} \left(\fint_{\lambda B_{i,m}} g ^n \right)^{1/n}=0.
\end{align}

We finally set $g=(\rho+1)\x_{B(0,R)}$ in the above manipulations, where $\rho$ is the given function from the statement  and $B(0,R)$ is a large ball containing the compact set $E$. Note that for all large $m\in \N$ we have $\lambda B_{i,m}\subset B(0,R)$ for all $i\in I_m$.  Then for every non-constant curve $\gamma\notin \Gamma_0$ we obtain \eqref{lemma:exceptional_family_limits} from \eqref{lemma:exceptional:family:g}. 
\end{proof}

\begin{lemma}\label{lemma:iteration}
Let $E\subset \R^n$ be a set, $\rho\colon \R^n \to [0,\infty]$ be a Borel function, and  $\gamma$ be a rectifiable path with distinct endpoints such that $|\gamma|\cap \br E$ is totally disconnected.  If Conclusion \ref{conclusion:a}$(E,\rho,\eta)$   is true for each {strict} subpath $\eta$ of $\gamma$ with distinct endpoints that do not lie in the set $\br E$, then Conclusion \ref{conclusion:a}$(E,\rho,\gamma)$  is also true.  The above statement remains true for Conclusion \ref{conclusion:b} in place of Conclusion \ref{conclusion:a}.
\end{lemma}

\begin{proof}
We only present the argument for Conclusion \ref{conclusion:a} since the argument for Conclusion \ref{conclusion:b} is similar and less technical. Let $U$ be an open neighborhood of $|\gamma|\setminus \partial \gamma$ and $\varepsilon>0$. It suffices to prove that a strong subpath of $\gamma$ with the same endpoints satisfies Conclusion \ref{conclusion:a}. Consider a parametrization $\gamma\colon [a,b]\to \R^n$. Then there exists an open subinterval $J$ of $[a,b]$ such that $\gamma(J)\subset U$ and $\gamma_{\br J}$ has the same endpoints as $\gamma$. Without loss of generality, we assume that $J=(0,1)$ and we denote the path $\gamma|_{[0,1]}$ by $\gamma$.  

Suppose first that $\br E\cap \partial \gamma=\emptyset$. Since $\br E$ is closed, there exist paths $\gamma|_{[0,t_1]}$, $\gamma|_{[t_2,1]}$ that do not intersect $\br E$, and a strict subpath $\eta=\gamma|_{[t_1,t_2]}$ with $\br E\cap \partial \eta=\emptyset$. By assumption, there exists a simple path $\widetilde \eta$ with the same endpoints as $\eta$ and \textit{finitely many} paths $\eta_i$, $i\in I$, inside $U$ as in  Conclusion \ref{conclusion:a}$(E,\rho,\eta)$. Concatenating $\widetilde \eta$ with $\gamma|_{[0,t_1]}$ and $\gamma|_{[t_2,1]}$, and then passing to a simple weak subpath, gives the desired path $\widetilde \gamma$ that verifies Conclusion \ref{conclusion:a}$(E,\rho,\gamma)$. 

Next, suppose that $\br E\cap \partial\gamma\neq \emptyset$. Consider a sequence $a_j\in (0,1)$, $j\in \Z$, such that $a_{j-1}<a_{j}$ for each $j\in \Z$ and $\bigcup_{j\in \Z}[a_{j-1},a_j]=(0,1)$. Let $\gamma^j=\gamma|_{[a_{j-1},a_j]}$, which is a strict subpath of $\gamma$. Since $\gamma((0,1))$ is disjoint from $\{\gamma(0),\gamma(1)\}$,  the points $a_j$ can be chosen so that $\gamma^j$ has distinct endpoints for $j\in \Z$. Since $|\gamma|\cap \br E$ is totally disconnected, we may further choose the points $a_j$ so that $\gamma(a_j)\notin \br E$ for each $j\in \Z$. 

By assumption, the strict subpath $\gamma^j$ of $\gamma$ satisfies  Conclusion \ref{conclusion:a}$(E,\rho,\gamma^j)$ for each $j\in \Z$. Thus, we obtain paths $\widetilde \gamma^j$ and $\gamma_i^j$, $i\in \{1,\dots,N_j\}$, as in the conclusion, for $\varepsilon 2^{-|j|}$ in place of $\varepsilon$,  and such that $\bigcup_{i=1}^{N_j}|\gamma^j_i|\subset U$. In particular, by \ref{criterion:i} we have $\widetilde \gamma^j\in \mathcal F_*(E)$. If $\mathcal F_*(E)=\mathcal F_0(E)$, since the endpoints of $\widetilde \gamma^j$ do not lie in $E$, we have $|\widetilde\gamma^j|\cap E=\emptyset$. By discarding some paths $\gamma_i^j$, we assume that $|\gamma_i^j|$ intersects $|\widetilde \gamma^j|$ for all $i\in \{1,\dots,N_j\}$. 

We consider parametrizations $\widetilde\gamma^j \colon [a_{j-1},a_j]\to U$ with $\widetilde \gamma^j|_{\{a_{j-1},a_j\}}=\gamma^j|_{\{a_{j-1},a_j\}}$, $j\in \Z$. Then we create a curve $\widetilde \gamma\colon [0,1]\to \br U$ such that $\widetilde \gamma((0,1))\subset U$, by concatenating these paths. Namely, we define $\widetilde \gamma(0)=\gamma(0)$, $\widetilde \gamma(1)=\gamma(1)$, and $\widetilde \gamma|_{[a_{j-1},a_j]}= \widetilde \gamma^j$ for $j\in \Z$. Note that $\ell(\widetilde \gamma^j)\leq \ell(\gamma^j)+\varepsilon2^{-|j|}$, which follows from Conclusion \ref{conclusion:a}. We conclude that $\diam( |\widetilde \gamma^j|)\to 0$, as $|j|\to \infty$, so $\widetilde \gamma$ is continuous. By passing to a weak subpath that has the same endpoints, we assume that $\widetilde \gamma$ is simple.

Property \ref{criterion:i} is immediate for $\widetilde \gamma$. Property \ref{criterion:ii} also holds if $\{\gamma_i\}_{i\in I}$ is the family $\{\gamma_i^j\}_{j\in \Z, i\in \{1,\dots,N_j\}}$. Indeed, all these paths are contained in $U$, and by the properties of the paths $\widetilde \gamma^j$ we have
\begin{align*}
|\widetilde \gamma|\setminus \partial \widetilde \gamma &\subset \bigcup_{j\in \Z} |\widetilde \gamma^j| \subset \bigcup_{j\in \Z} \left( (| \gamma^j|\setminus E) \cup \bigcup_{i=1}^{N_j}|\gamma_i^j| \right) \subset (|\gamma|\setminus E)\cup \bigcup_{j\in \Z} \bigcup_{i=1}^{N_j}|\gamma_i^j|.
\end{align*}
Finally, we have
\begin{align*}
\sum_{j\in \Z} \sum_{i=1}^{N_j} \ell(\gamma_i^j)\leq \sum_{j\in \Z} \varepsilon2^{-|j|}  =3\varepsilon \quad \textrm{and}\quad \sum_{j\in \Z} \sum_{i=1}^{N_j}  \int_{\gamma_i^j}\rho\, ds \leq 3\varepsilon.
\end{align*}
Thus \ref{criterion:iii} and \ref{criterion:iv} hold as well. 

Now, we verify the last part of Conclusion \ref{conclusion:a}. The paths $\gamma_{i}^j$ may be taken to lie outside a given curve family $\Gamma'$ with $\md_n\Gamma'=0$, since they were obtained via Conclusion \ref{conclusion:a}. If $\eta$ is a strict subpath of $\widetilde \gamma$, then $|\eta|$ intersects $\bigcup_{i=1}^{N_j} |\gamma_i^j|$ for finitely many $j\in \Z$. Indeed, as $|j|\to \infty$, the paths $\{\gamma_i^j\}_{i\in \{1,\dots,N_j\}}$ have arbitrarily small lengths by the above, and their traces intersect $|\widetilde \gamma^j|$, which is contained in arbitrarily small neighborhoods of the endpoints of $\widetilde \gamma$. Since $|\eta|$ has positive distance from the endpoints of $\widetilde \gamma$, we conclude that  $|\gamma_{i}^j|$ cannot intersect $|\eta|$ if $|j|$ is sufficiently large. This completes the proof.
\end{proof}

\subsection{Proof of main criterion}
As we have discussed, it suffices to show implication   \ref{cc:v} $\Rightarrow$ \ref{cc:i} for all sets $E$, and implication \ref{cc:loewner} $\Rightarrow$ \ref{cc:iv}, which is the most technical one, for sets $E$ that are closed or whose closure has measure zero.

\begin{proof}[Proof of Theorem \ref{theorem:criterion_compact}: \ref{cc:v} $\Rightarrow$ \ref{cc:i}] 
Let $\Omega\subset \R^n$ be an open set and $F_1,F_2\subset \Omega$ be a pair of non-empty, disjoint continua. By the monotonicity of modulus, it suffices to show that
$$\md_n\Gamma(F_1,F_2;\Omega)\leq \md_n(\Gamma(F_1,F_2;\Omega)\cap \mathcal F_*(E)).$$
Let $\rho\in L^n(\R^n)$ be an admissible function for $\Gamma( F_1,F_2;\Omega)\cap \mathcal F_*(E)$. We consider an exceptional curve family $\Gamma_0$ with $\md_n\Gamma_0=0$ as in \ref{cc:v}. Let $\gamma\in \Gamma(F_1,F_2;\Omega)\setminus \Gamma_0$, so Conclusion \ref{conclusion:b}$(E,\rho,\gamma)$ is true. Thus, for any $\varepsilon>0$ there exists a rectifiable path $\widetilde \gamma\in \Gamma(F_1,F_2;\Omega)\cap \mathcal F_*(E)$ with 
$$1\leq \int_{\widetilde \gamma} \rho\, ds \leq \int_{\gamma} \rho\, ds +\varepsilon.$$
Letting $\varepsilon\to 0$ shows that $\rho$ is admissible for $\Gamma(F_1,F_2;\Omega)\setminus \Gamma_0$. Since $\md_n\Gamma_0=0$, the proof is completed.
\end{proof}

\begin{proof}[Proof of Theorem \ref{theorem:criterion_compact}: \ref{cc:loewner} $\Rightarrow$ \ref{cc:iv}]
By the assumption \ref{cc:loewner}, which coincides with the assumption of Lemma \ref{lemma:measure_zero_loewner}, we have $m_n(E)=0$. Therefore, under either of the initial assumptions of Theorem \ref{theorem:criterion_compact}, we have $m_n(E)=0$. Let $\rho\colon \R^n\to [0,\infty]$ be a Borel function in $L^n_{\loc}(\R^n)$.  By Lemma \ref{lemma:exceptional_family}, for each $m\in \N$ there exists a family of balls $\{B_{i,m}\}_{i\in I_m}$ covering the set $\br E$, with radii converging uniformly in $I_m$ to $0$ as $m\to\infty$, and such that \eqref{lemma:exceptional_family_limits} is true for all paths outside an exceptional family with $n$-modulus zero and for a value of $\lambda>0$ to be specified. Let $\Gamma_0$ be the exceptional family of paths $\gamma$ that either do not  satisfy \eqref{lemma:exceptional_family_limits}, or $\h^1(|\gamma|\cap \br E)>0$. By Lemma \ref{lemma:exceptional_family} and property \ref{m:positive_length}, $\md_n\Gamma_0=0$. Note that the path family $\Gamma_0$ has the required property for \ref{cc:iv}, that if $\{\eta_j\}_{j\in J}$ is a finite collection of paths outside $\Gamma_0$ and $\gamma$ is a path with $|\gamma|\subset \bigcup_{j\in J} |\eta_j|$, then $\gamma\notin \Gamma_0$.

We will show that Conclusion \ref{conclusion:a}$(E,\rho,\gamma)$ holds for all paths $\gamma\notin \Gamma_0$ with distinct endpoints. For $\gamma\notin \Gamma_0$, the set $|\gamma|\cap \br E$ is totally disconnected. In view of Lemma \ref{lemma:iteration}, it suffices to show that Conclusion \ref{conclusion:a}$(E,\rho,\eta)$ is true for each subpath $\eta$ of $\gamma$ with distinct endpoints not lying in the closed set $\br E$. By the defining properties of the curve family $\Gamma_0$, all subpaths of $\gamma$ also lie outside $\Gamma_0$. Therefore, it suffices to show that Conclusion \ref{conclusion:a}$(E,\rho,\gamma)$ is true for each path $\gamma\notin \Gamma_0$ with distinct endpoints \textit{not lying in $\br E$}. 

Let $\gamma\notin \Gamma_0$ be a path with distinct endpoints not lying in $\br E$. As a final reduction, we consider a simple weak subpath $\eta$ of $\gamma$ that has the same endpoints. Note that $\eta\notin \Gamma_0$ by the defining properties of $\Gamma_0$ and that is  suffices to show Conclusion \ref{conclusion:a}$(E,\rho,\eta)$, which implies Conclusion \ref{conclusion:a}$(E,\rho,\gamma)$. Hence, we may impose the additional restriction that $\gamma$ is \textit{simple}.

We fix a simple path $\gamma\notin \Gamma_0$ with distinct endpoints not lying in $\br E$, $\varepsilon>0$, and an open neighborhood $U$ of $|\gamma|\setminus \partial \gamma$. We fix an injective parametrization $\gamma\colon [0,1]\to \R^n$ and note that $\gamma(0),\gamma(1)\notin \br E$ and $\gamma((0,1))\subset U$. 

By the compactness of $|\gamma|\cap \br E$ and the assumption \ref{cc:loewner}, there exists a finite cover of $|\gamma|\cap \br E$ by open balls $V_1,\dots,V_L$ such that for every $i\in \{1,\dots,L\}$ and for any non-degenerate, disjoint continua $F_1,F_2\subset 2V_i$, 
\begin{align}\label{theorem:criterion:nedw}
\textrm{if} \,\,\, \Delta(F_1,F_2)\leq t\,\,\, \textrm{then}\,\,\, \md_n (\Gamma(F_1,F_2;\R^n) \cap \mathcal F_*(E)) \geq \phi.
\end{align}
Observe that if a set $D$ intersects $|\gamma|\cap \br E$ and has sufficiently small diameter, namely $\diam(D)\leq \min\{2^{-1}\diam(V_i): i\in \{1,\dots,L\}\}$, then $D\subset 2V_i$ for some $i\in \{1,\dots,L\}$. Hence,  \eqref{theorem:criterion:nedw} holds for any non-degenerate, disjoint continua $F_1,F_2\subset D$. We also fix $a>1$, depending on $t$, so that if $A=A(x;r,ar)$ is any ring and $F_1,F_2\subset \br A$ are disjoint continua connecting the boundary components of $A$, then 
\begin{align*}
\Delta(F_1,F_2) \leq \frac{2}{a-1} \leq t.  
\end{align*}

If $m\in \N$ is fixed and sufficiently large, by \eqref{lemma:exceptional_family_limits} we have
\begin{align}\label{theorem:criterion:sums}
\sum_{i: B_{i,m}\cap |\gamma|\neq \emptyset} r_{i,m} \left(\fint_{\lambda B_{i,m}} \rho^n \right)^{1/n}< \varepsilon \quad \textrm{and}\quad \sum_{i: B_{i,m}\cap |\gamma|\neq \emptyset} r_{i,m}<\varepsilon.
\end{align}
By the compactness of $|\gamma|\cap \br E$, there exists a finite subcollection $D_1,\dots,D_N$ of  $\{B_{i,m}\}_{i\in I_m}$ covering the set $|\gamma|\cap \br E$. We also assume that $D_i$ intersects $|\gamma|\cap \br E$ for each $i\in \{1,\dots,N\}$ and we denote the radius of $D_i$ by $r_i$. Since the endpoints of $\gamma$ do not lie in $\br E$, we have $|\gamma|\cap \br E \subset U$. Thus, 
$$\delta=  \dist(|\gamma|\cap \br E,\partial \gamma\cup \partial U)>0.$$ 
If $m$ is sufficiently large so that $2(a+\lambda)r_i<\delta$ for each $i\in \{1,\dots,N\}$, we have 
\begin{align}\label{theorem:criterion:inclusion}
\bigcup_{i=1}^N  (a+\lambda)D_i\subset U  \quad \textrm{and}\quad \bigcup_{i=1}^N aD_i\cap \partial\gamma=\emptyset.
\end{align}
Finally, we choose an even larger $m$, so that \eqref{theorem:criterion:nedw} holds for any non-degenerate, disjoint continua $F_1,F_2\subset (a+1)D_i$,  $i\in \{1,\dots,N\}$.

We set $\widetilde \gamma_0=\gamma$. We will define inductively simple paths $\widetilde \gamma_k$, $k\in \{0,\dots,N\}$, with the same endpoints as $\gamma$. Once $\widetilde \gamma_{k-1}$ has been defined for some $k\in \{1,\dots,N\}$ and has the same endpoints as $\gamma$, we define $\widetilde \gamma_k$ as follows. If $D_k\cap |\widetilde \gamma_{k-1}|=\emptyset$, then we set $\gamma_k=\emptyset$ (i.e., the empty path) and $\widetilde \gamma_k=\widetilde \gamma_{k-1}$. Suppose $D_k\cap |\widetilde \gamma_{k-1}|\neq \emptyset$. Consider an injective parametrization $\widetilde \gamma_{k-1}\colon [0,1]\to \R^n$. By \eqref{theorem:criterion:inclusion} the endpoints of $\widetilde \gamma_{k-1}$ do not lie in $aD_k$, thus there exist two moments $0<s_1<s_2<1$ such that $\widetilde \gamma_{k-1}(s_1),\widetilde\gamma_{k-1}(s_2)\in \partial D_k$ and $\widetilde\gamma_{k-1}([0,1]\setminus (s_1,s_2))\cap D_k=\emptyset$.  Moreover, there exist moments $s_1'<s_1$ and $s_2'>s_2$ such that $\widetilde \gamma_{k-1}(s_1'),\widetilde \gamma_{k-1}(s_2')\in \partial (aD_k)$, $G_1=\widetilde\gamma_{k-1}([s_1',s_1])\subset \br{aD_k}\setminus {D_k}$ and $G_2=\widetilde\gamma_{k-1}([s_2,s_2']) \subset \br{aD_k}\setminus {D_k}$; see Figure \ref{figure:gamma}. Note that $G_1$ and $G_2$ are disjoint since the path $\widetilde \gamma_{k-1}$ is simple, and that they connect the boundary components of the ring $aD_k\setminus \br{D_k}$.  Since $\Delta(G_1,G_2)\leq t$ (by the choice of $a$) and $G_1,G_2\subset (a+1)D_k$, by \eqref{theorem:criterion:nedw} we have
\begin{align*}
\md_n (\Gamma(G_1,G_2;\R^n)\cap \mathcal F_*(E))\geq \phi.
\end{align*} 
Note that if $\Gamma'$ is a given path family with $\md_n\Gamma'=0$ as in the end of Conclusion \ref{conclusion:a}, then we also have $$\md_n (\Gamma(G_1,G_2;\R^n)\cap \mathcal F_*(E)\setminus \Gamma')\geq \phi.$$
Each path of $\Gamma(G_1,G_2;\R^n)\cap \mathcal F_*(E)\setminus \Gamma'$ intersects the ball $(a+1)D_k$. By Lemma \ref{lemma:path_bound_modulus} there exists a path $\gamma_k \in \Gamma(G_1,G_2;\R^n)\cap \mathcal F_*(E)\setminus \Gamma'$ such that 
\begin{align}\label{theorem:criterion:integral_length}
\int_{\gamma_k}\rho \, ds \lesssim_{n,\phi,a} r_k \left( \fint_{c(n,\phi,a)D_k} \rho^n \right)^{1/n} \quad \textrm{and}\quad \ell(\gamma_k) \leq  c(n,\phi,a)r_k.
\end{align}
We now set $\lambda=c(n,\phi,a)$ and note that $|\gamma_k|\subset (a+\lambda)D_k\subset U$ by \eqref{theorem:criterion:inclusion}. Also, the endpoints of $\gamma_k$ lie in $|\widetilde \gamma_{k-1}|$. We concatenate $\gamma_k$ with suitable subpaths of $\widetilde \gamma_{k-1}$ that do not intersect $D_k$ to obtain a path $\widetilde \gamma_k$ that has the same endpoints as $\gamma$. If necessary, we replace $\widetilde \gamma_k$ with a simple weak subpath that has the same endpoints. By construction we have 
$$|\widetilde \gamma_k|\subset (|\widetilde \gamma_{k-1}|\setminus D_k)\cup (|\gamma_k|\setminus \partial \gamma_k).$$
Inductively, we see that 
\begin{align}\label{theorem:criterion:induction}
|\widetilde\gamma_k|\subset \left(|\gamma| \setminus \bigcup_{i=1}^k D_i \right)\cup \bigcup_{i=1}^{k} (|\gamma_i|\setminus \partial \gamma_i).
\end{align}

\begin{figure}
\centering
	\begin{overpic}[width=.9\linewidth ]{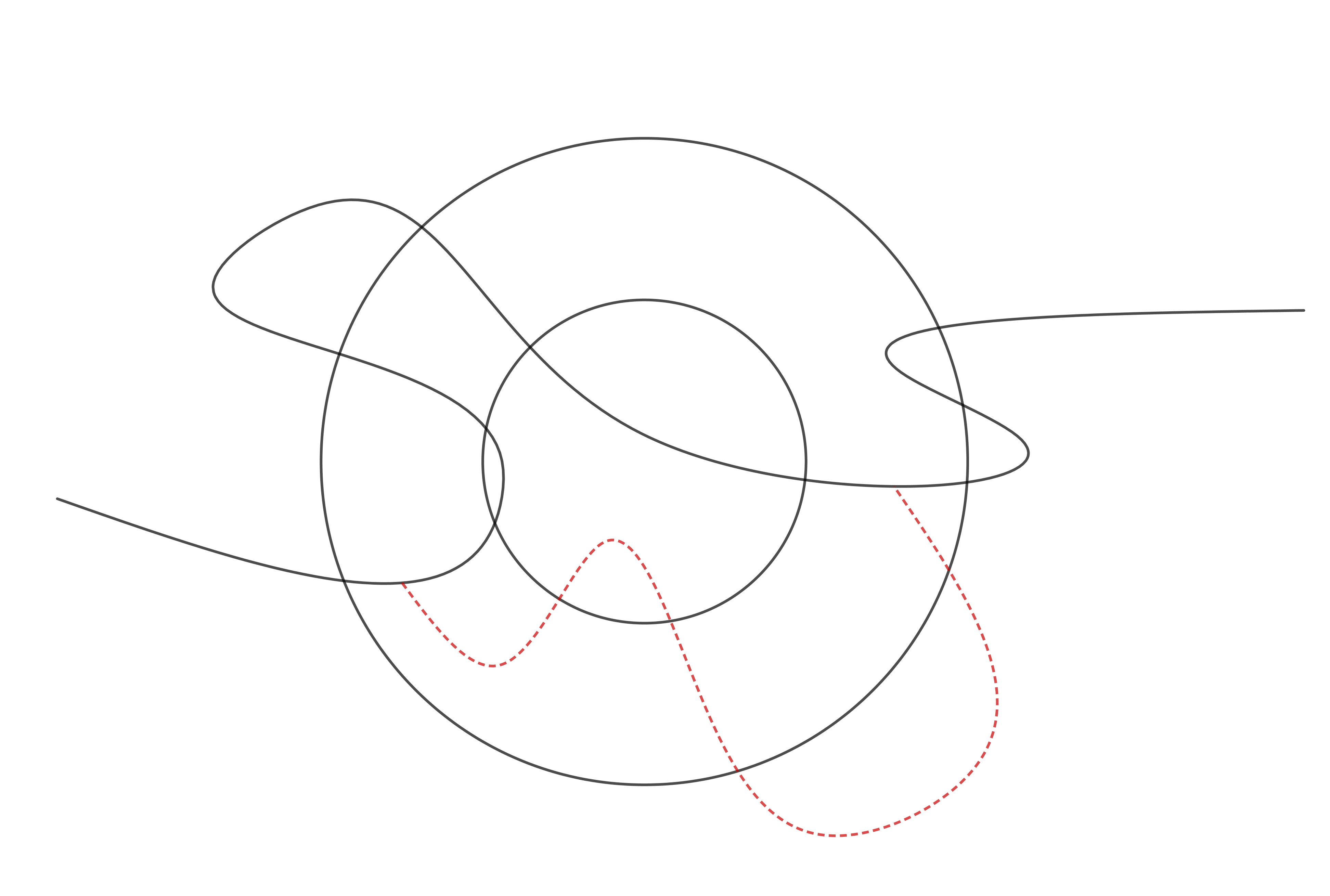}
		\put (3, 29.3) {$\bullet$}
		\put (0,32.3) {$\gamma(0)=\widetilde \gamma_{k-1}(0)$}
		\put (97,43.2) {$\bullet$}
		\put (85,39.3) {$\gamma(1)=\widetilde \gamma_{k-1}(1)$}
		\put (25, 22.8) {$\bullet$}
		\put (36.5, 27.3) {$\bullet$}
		\put (60,30.3) {$\bullet$}
		\put (71.8, 30.3) {$\bullet$}
		\put (55,43.3) {$D_k$}
		\put (65, 51.3 ) {$aD_k$}
		\put (30,25.3) {$G_1$}
		\put (65,31.8) {$G_2$}
		\put (72.5,6) {$\gamma_k$}
		\put (20,53.5) {$\widetilde \gamma_{k-1}$}
	\end{overpic}
	\caption{The construction of $\widetilde \gamma_k$ from $\widetilde \gamma_{k-1}$. }\label{figure:gamma}
\end{figure}
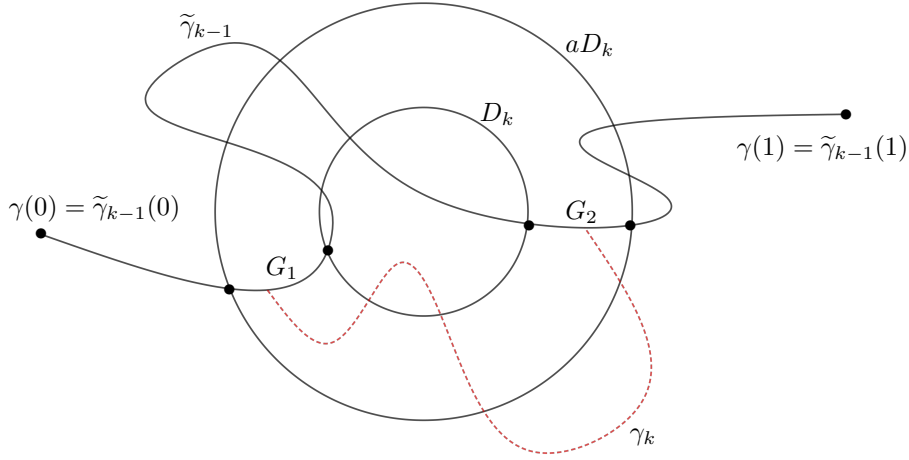

For $k=N$ we obtain a simple path $\widetilde \gamma=\widetilde \gamma_N$ with the same endpoints as $\gamma$. Since $|\gamma|\cap \br E\subset \bigcup_{i=1}^N D_i$, by \eqref{theorem:criterion:induction} we have 
\begin{align}\label{theorem:criterion:induction2}
|\widetilde \gamma| \subset  (|\gamma|\setminus \br E) \cup \bigcup_{i=1}^{N} (|\gamma_i|\setminus \partial \gamma_i) \quad \textrm{and}\quad |\widetilde \gamma|\cap E \subset \bigcup_{i=1}^{N} (|\gamma_i|\setminus \partial \gamma_i)\cap E.
\end{align}
If $\mathcal F_*(E)=\mathcal F_0(E)$, then   $(|\gamma_i|\setminus \partial \gamma_i)\cap E=\emptyset$ since $\gamma_i\in \mathcal F_0(E)$ for each $i\in \{1,\dots,N\}$. Thus, $|\widetilde \gamma|\cap E=\emptyset$ and $\widetilde \gamma\in \mathcal F_0(E)$. If $\mathcal F_*(E)=\mathcal F_{\sigma}(E)$, then $|\widetilde \gamma_i|\cap E$ is countable for each $i\in \{1,\dots,N\}$, so $|\widetilde \gamma|\cap E$ is countable and $\widetilde \gamma\in \mathcal F_{\sigma}(E)$. Thus,  \ref{criterion:i} is satisfied. From \eqref{theorem:criterion:induction2} we obtain immediately
\begin{align*}
|\widetilde \gamma| \subset (|\gamma|\setminus  E) \cup \bigcup_{i=1}^N |\gamma_i|.
\end{align*}
By construction, we have $\bigcup_{i=1}^N|\gamma_i|\subset \bigcup_{i=1}^N (a+\lambda)D_i \subset U$.   Thus, we have established property \ref{criterion:ii}.  Finally, by \eqref{theorem:criterion:sums} and \eqref{theorem:criterion:integral_length} we have 
$$\sum_{i=1}^N \ell(\gamma_i) \lesssim \sum_{i=1}^N r_i \lesssim \varepsilon \,\,\, \textrm{and}\,\,\,
\sum_{i=1}^N \int_{\gamma_i}\rho\, ds \lesssim  \sum_{i=1}^Nr_{i} \left(\fint_{\lambda D_{i}} \rho^n \right)^{1/n}\lesssim \varepsilon.$$
These inequalities address \ref{criterion:iii} and \ref{criterion:iv}. The last part of Conclusion \ref{conclusion:a}, asserting that the collection  $\{\gamma_i\}_i$ is finite, whenever $\br E\cap \partial\gamma=\emptyset$, is also true. 
\end{proof}

\subsection{\texorpdfstring{Proof of $P$-family criterion}{Proof of P-family criterion}}

\begin{proof}[Proof of Theorem \ref{theorem:criterion_pfamily}: \ref{cp:i} $\Rightarrow$ \ref{cp:ii}]
Let $\rho$ be a non-negative Borel function with  $\rho\in L^n(\R^n)$. By Theorem \ref{theorem:criterion_compact} \ref{cc:v}, there exists a path family $\Gamma_0$ with $\md_n\Gamma_0=0$ such that Conclusion \ref{conclusion:b}$(E,\rho,\gamma)$ is true for all paths $\gamma\notin \Gamma_0$ with distinct endpoints. Moreover, by enlarging $\Gamma_0$ while still requiring that $\md_n\Gamma_0=0$, we may assume that if $\gamma\notin \Gamma_0$, then all subpaths of $\gamma$ also have this property; see \ref{m:subpath}.  By Lemma \ref{modulus:avoid}, for each rectifiable path $\gamma$ and for a.e.\ $x\in \R^n$ we have $\gamma+x\notin \Gamma_0$; in particular, the same holds for all strong subpaths of $\gamma+x$, as required in part \ref{cp:ii:i}. Moreover, if $\gamma_w$ is as in \ref{cp:ii:ii}, by the same lemma, for $\h^{n-1}$-a.e.\ $w\in S^{n-1}(0,1)$ the path $\gamma_w$ and all of its strong subpaths lie outside $\Gamma_0$. This completes the proof. 
\end{proof}

\begin{proof}[Proof of Theorem \ref{theorem:criterion_pfamily}: \ref{cp:iii} $\Rightarrow$ \ref{cp:i}]
Let $\rho\in L^n(\R^n)$ be an admissible function for $\Gamma(F_1,F_2;\R^n)\cap \mathcal F_{\sigma}(E)$, where $F_1,F_2$ are disjoint continua. Consider the $P$-family $\mathcal F$ with the given properties; note that $\mathcal F$ depends on $\rho$. Let $\gamma\in \Gamma(F_1,F_2;\R^n)\cap \mathcal F$ be a rectifiable path. By the properties of $\mathcal F$, each strong subpath of $\gamma$ with distinct endpoints, and in particular $\gamma$, satisfies Conclusion \ref{conclusion:b}. Hence, for each $\varepsilon>0$ there exists a rectifiable path $\widetilde \gamma\in \Gamma(F_1,F_2;\R^n) \cap\mathcal F_{\sigma}(E)$ such that 
$$1\leq \int_{\widetilde \gamma}\rho\, ds \leq \int_{\gamma}\rho \, ds+\varepsilon.$$
As $\varepsilon \to 0$, this shows that $\rho$ is admissible for $\Gamma(F_1,F_2;\R^n)\cap \mathcal F$. Hence,
$$\md_n(\Gamma(F_1,F_2;\R^n)\cap \mathcal F) \leq \int \rho^n.$$
Since $\mathcal F$ is a $P$-family, by Theorem \ref{theorem:perturbation_family}, we have
$$\md_n \Gamma(F_1,F_2;\R^n)= \md_n(\Gamma(F_1,F_2;\R^n)\cap \mathcal F)\leq \int \rho^n.$$
Infimizing over $\rho$, gives $\md_n \Gamma(F_1,F_2;\R^n) \leq \md_n (\Gamma(F_1,F_2;\R^n)\cap \mathcal F_{\sigma}(E))$, so $E$ is $\CNED$. 
\end{proof}

For the proof of the last implication \ref{cp:ii} $\Rightarrow$ \ref{cp:iii}, it is crucial that $\mathcal F_{\sigma}(E)$ is closed under concatenations, a fact that is not true for $\mathcal F_0(E)$. Although the result is true for $\mathcal F_0(E)$, the proof is more involved and we only present the argument for $\mathcal F_{\sigma}(E)$.

\begin{proof}[Proof of Theorem \ref{theorem:criterion_pfamily}: \ref{cp:ii} $\Rightarrow$ \ref{cp:iii}]
First, we show that $m_n(E)=0$ in the case that $E$ is closed. Fix an open ball $B$ and consider the Borel function $\rho=(1-\x_E)\x_B$, which lies in $L^n(\R^n)$. By \ref{cp:ii:i}, Conclusion \ref{conclusion:b}$(E,\rho,\gamma)$ is true for a.e.\ line segment $\gamma$ parallel to a given coordinate direction. Let $\gamma$ be such a line segment that is contained in the ball $B$. For each $\varepsilon>0$ there exists a path $\widetilde \gamma \in \mathcal F_{\sigma}(E)$ contained in $B$, with the same endpoints as $\gamma$, and 
\begin{align*}
\int_{\widetilde \gamma}\rho\, ds\leq \int_{\gamma}\rho\,ds+\varepsilon.
\end{align*}
By Lemma \ref{lemma:paths} \ref{lemma:paths:iii} we have $\int_{\widetilde \gamma}\x_E\, ds=0$ since $|\gamma|\cap E$ is countable. Thus,
\begin{align*}
\ell(\gamma)\leq \ell(\widetilde \gamma) =\int_{\widetilde \gamma}1\, ds =\int_{\widetilde \gamma}\rho\, ds\leq \int_{\gamma}\rho\, ds+\varepsilon= \ell(\gamma) - \int_{\gamma}\x_E\, ds+\varepsilon.
\end{align*}
By letting $\varepsilon\to 0$, we obtain $\int_{\gamma}\x_E\, ds=0$, so $\h^1(|\gamma|\cap E)=0$ by Lemma \ref{lemma:paths} \ref{lemma:paths:iv}. This is true for a.e.\ line segment $\gamma$ in $B$ parallel to a coordinate direction, so $m_n(E\cap B)=0$ by Fubini's theorem. The ball $B$ was arbitrary, so $m_n(E)=0$. 

Let $\rho$ be a non-negative Borel function in $L^n(\R^n)$ and let $\mathcal F$ be the family of rectifiable paths $\gamma$ such that Conclusion \ref{conclusion:b}$(E,\rho,\eta)$ holds for each strong subpath $\eta$ of $\gamma$ with distinct endpoints. We show that $\mathcal F$ is a $P$-family. 

First, we see that \ref{perturbations:1} is satisfied. That is, if $\gamma\colon [a,b]\to \R^n$ is a non-constant rectifiable path, then $\gamma+x\in \mathcal F$ for a.e.\ $x\in \R^n$. This is true by the assumption \ref{cp:ii:i}. For \ref{perturbations:2}, we fix $x\in \R^n$ and $R>0$, and we have to show that for a.e.\ $w\in S^{n-1}(0,1)$ the radial segment $\gamma_w(t)=x+tw$, $t\in [0,R]$, lies in $\mathcal F$. By the assumption \ref{cp:ii:ii}, for each $0<r<R$ and for a.e.\ $w\in S^{n-1}(0,1)$ the radial segment $\gamma_w|_{[r,R]}$ lies in $\mathcal F$. For $r_k=2^{-k}R$, $k\in \N$, we see that $\gamma_{w}|_{[r_k,R]}\in \mathcal F$ for a.e.\ $w\in S^{n-1}(0,1)$ and for all $k\in \N$.  We fix $w\in S^{n-1}(0,1)$ such that $\h^1(|\gamma_w|\cap \br E)=0$ and $\gamma_{w}|_{[r_k,R]}\in \mathcal F$ for all $k\in \N$. Since $m_n(\br E)=0$, these statements hold for a.e.\ $w\in S^{n-1}(0,1)$. Our goal is to show that Conclusion \ref{conclusion:b}$(E,\rho,\eta)$ is true for every strong subpath $\eta$ of $\gamma_w$; this will imply that $\gamma_w\in \mathcal F$, as desired. Every \textit{strict} subpath of $\eta$ is a strong subpath of $\gamma_{w}|_{[r_k,R]}$ for some $k\in \N$. Since $\gamma_{w}|_{[r_k,R]}\in \mathcal F$, every strict subpath of $\eta$ satisfies Conclusion \ref{conclusion:b}. Since $|\eta|\cap \br E$ is totally disconnected, we conclude from Lemma \ref{lemma:iteration} that $\eta$ satisfies Conclusion \ref{conclusion:b}.
 
From the definition of $\mathcal F$ it is clear that \ref{perturbations:3} is always satisfied. We finally have to prove \ref{perturbations:4}. Note that  $\mathcal F_*(E)=\mathcal F_{\sigma}(E)$ in \ref{criterion:bi}. Let $\gamma_1,\gamma_2$ be two paths in $\mathcal F$ that have a common endpoint and let $\gamma$ be their concatenation. Consider a strong subpath $\eta$ of $\gamma$ that has distinct endpoints.  Then $\eta$ is either a strong subpath of $\gamma_1$ or $\gamma_2$, or $\eta$ is the concatenation of strong subpaths $\eta_1$ of $\gamma_1$ and $\eta_2$ of $\gamma_2$. In the latter case, which is the nontrivial one, since $\eta_1$ and $\eta_2$ satisfy Conclusion \ref{conclusion:b}, and in particular \ref{criterion:bi}, there exists paths $\widetilde \eta_1, \widetilde \eta_2 \in \mathcal F_{\sigma}(E)$ as in Conclusion \ref{conclusion:b} with the same endpoints as $\eta_1,\eta_2$, by \ref{criterion:bii}. Concatenating $\widetilde \eta_1$ with $\widetilde \eta_2$ gives a path $\widetilde \eta \in \mathcal F_{\sigma}(E)$, which shows that $\eta$ also satisfies Conclusion \ref{conclusion:b}. Thus, $\gamma\in \mathcal F$, as desired. 
\end{proof}

\subsection{Application to Sobolev removability}\label{section:sobolev}
We prove Theorem \ref{theorem:sobolev_intro} as an application  of Theorem \ref{theorem:criterion_compact}. Specifically, we show that closed $\CNED$ sets $E\subset \R^n$ are {removable for continuous $W^{1,n}$ functions}; that is, every continuous function $f\colon \R^n\to \R$ with $f\in W^{1,n}(\R^n\setminus E)$ lies in $W^{1,n}(\R^n)$. 

\begin{proof}[Proof of Theorem \ref{theorem:sobolev_intro}]
Let $f\colon \R^n \to \R$ be continuous with $f\in W^{1,n}(\R^n\setminus E)$. Then the classical gradient $\nabla f$ exists almost everywhere in $\R^n\setminus E$. Since $m_n(E)=0$, there exists a Borel representative of $|\nabla f|$ on $\R^n$. By Fuglede's theorem \cite{Vaisala:quasiconformal}*{Theorem 28.2, p.~95}, there exists a curve family $\Gamma_1$ with $\md_n\Gamma_1=0$ such that for each $\gamma\notin \Gamma_1$ we have $\int_{\gamma} |\nabla f|\, ds<\infty$ and the function $f$ is absolutely continuous on every subpath $\gamma|_{[a,b]}$ of $\gamma$ with $\gamma((a,b))\subset \R^n\setminus E$. Moreover,
$$|f(\gamma(b))-f(\gamma(a))| \leq \int_{\gamma|_{(a,b)}}|\nabla f|\, ds.$$
Let $\gamma\colon [a,b]\to \R^n$ be a path outside  $\Gamma_1$. The set $(a,b)\setminus \gamma^{-1}(E)$ is a countable union of disjoint open intervals $(a_i,b_i)$, $i\in \N$. Using the continuity of $f$ and the above inequality, we have
\begin{align}\label{sobolev:ineq1}
\begin{aligned}
m_1( f( |\gamma|\setminus E))&\leq  \sum_{i\in \N}m_1( f( \gamma([a_i,b_i]))) \leq \sum_{i\in \N}( \max_{[a_i,b_i]} f\circ \gamma- \min_{[a_i,b_i]} f\circ \gamma)\\
&\leq \sum_{i\in \N} \int_{\gamma|_{(a_i,b_i)}} |\nabla f|\, ds \leq  \int_{\gamma} |\nabla f|\, ds.
\end{aligned} 
\end{align} 
In particular, if $\gamma\in \mathcal F_{\sigma}(E)\setminus \Gamma_1$, then
\begin{align}\label{sobolev:ineq2}
m_1(f(|\gamma|)) =m_1(f(|\gamma|\setminus E))\leq \int_{\gamma} |\nabla f|\, ds. 
\end{align}

Since $m_n(E)=0$, in order to show that $f\in W^{1,n}(\R^n)$, it suffices to show that there exists a path family $\Gamma_0$ with $\md_n\Gamma_0=0$, such that $f$ is absolutely continuous along each path $\gamma\notin \Gamma_0$. Let $\Gamma_2$ be the path family given by Theorem \ref{theorem:criterion_compact} \ref{cc:iv}, with $\md_n\Gamma_2=0$ and let $\Gamma_0=\Gamma_1\cup \Gamma_2$. We fix a path $\gamma\notin \Gamma_0$  a subpath $\beta= \gamma|_{[a,b]}$. Note that $\beta\notin \Gamma_2$ by the properties of $\Gamma_2$ in Theorem \ref{theorem:criterion_compact} \ref{cc:iv}. Hence, Conclusion \ref{conclusion:a}$(E,|\nabla f|, \beta)$ is true. For each $\varepsilon>0$ there exists a path $\widetilde \beta \in \mathcal F_{\sigma}(E)$ with the same endpoints as $\beta$ and a collection of paths $\beta_i\notin \Gamma_1$, $i\in I$, such that 
$$|\widetilde \beta|\setminus \partial \widetilde \beta \subset (|\beta|\setminus E) \cup \bigcup_{i\in I}|\beta_i|\,\,\, \textrm{and}\,\,\, \sum_{i\in I}\int_{\beta_i}|\nabla f| \, ds< \varepsilon.$$
Thus, using \eqref{sobolev:ineq1} and \eqref{sobolev:ineq2}, we have
\begin{align*}
|f(\beta(b))-f(\beta(a))| &\leq m_1(f(|\widetilde \beta|)) =m_1( f(|\widetilde \beta|\setminus \partial \beta))\\
&\leq m_1( f(|\beta|\setminus E)) + \sum_{i\in I} m_1(f(|\beta_i|))\\
&\leq \int_{\beta} |\nabla f|  \, ds + \sum_{i\in I} \int_{\beta_i} |\nabla f| \, ds\\
&\leq  \int_{\beta} |\nabla f|  \, ds +\varepsilon.
\end{align*}
Letting $\varepsilon\to 0$ shows that $f$ is absolutely continuous along $\gamma$. 
\end{proof}

\section{Unions of negligible sets}\label{section:union}

In this section we prove Theorem \ref{theorem:unions}, which asserts that the union of countably many closed $\NED$ (resp.\ $\CNED$) sets is $\NED$ (resp.\ $\CNED$). The proof is based on Theorem \ref{theorem:criterion_compact} from the preceding section. We first prove an auxiliary lemma.

\begin{lemma}\label{lemma:hausdorf_continua}
Let $G_i$, $i\in \N$, be a sequence of continua in $\R^n$ with $G_i\subset G_{i+1}$, $i\in \N$, and $\sup_{i\in \N} \h^1(G_i)<\infty$. If $G= \bigcup_{i\in \N} G_i$, then $\h^1(\br G\setminus G)=0$. In particular if $\gamma\colon [0,1]\to \br G$ is a rectifiable path and $\rho\colon \R^n\to [0,\infty]$ is a Borel function, then
\begin{align*}
\int_{\gamma}\rho \, ds =\int_{\gamma} \rho \x_G\, ds. 
\end{align*}
\end{lemma}
\begin{proof}
Without loss of generality, $\diam(G_1)>0$. Let $d=\diam(G_1)$ and we fix $k\in \N$. Since $G_k$ is closed, for each $x\in \br G\setminus G \subset \br G\setminus G_k$ there exists $r_x>0$ such that $B(x,r_x)\cap G_k=\emptyset$. By the $5B$-covering lemma (\cite{Heinonen:metric}*{Theorem 1.2}), there exists family of balls $B_i=B(x_i,r_i)$ with $x_i\in \br G\setminus G$ and $r_i<d$, $i\in \N$, such that  $\br G\setminus G\subset \bigcup_{i\in \N} B_i$, the balls $\frac{1}{5}B_i$ are disjoint, and $\frac{1}{5}B_i\cap G_k=\emptyset$ for each $i\in \N$. 

Fix $i\in \N$. Since $x_i\in \br G$, there exists a point $y\in \frac{1}{10}B_i\cap G$. The sequence $\{G_j\}_{j\in \N}$ is increasing, so there exists $j>k$ such that $y\in \frac{1}{10}B_i\cap G_j$. Since $\diam(G_j)\geq d> \diam(\frac{1}{5}B_i)$ and $G_j$ is a continuum,  there exists a connected set in $   \frac{1}{5}B_i\cap G_j$ that connects $\partial B(x_i,r_i/10)$ to $\partial B(x_i,r_i/5)$. In combination with the fact that $\frac{1}{5}B_i\cap G_k=\emptyset$, we obtain
\begin{align*}
\h^1\left( (G\setminus G_k) \cap \frac{1}{5}B_i\right) \geq \h^1\left( G_j\cap \frac{1}{5}B_i\right) \geq r_i/10.
\end{align*}

We now have
\begin{align*}
\h^1_{\infty}(\br  G\setminus G)\leq \sum_{i\in \N} 2r_i \leq 20 \sum_{i\in \N} \h^1\left( (G\setminus G_k) \cap \frac{1}{5}B_i\right)\leq 20 \h^1(G\setminus G_k). 
\end{align*} 
Since $\h^1(G)=\sup_{i\in \N} \h^1(G_i)<\infty$ and $G\setminus G_k$ decreases to $\emptyset$ as $k\to\infty$, we have $\h^1(\br G\setminus G)=\h^1_{\infty}(\br G\setminus G)=0$. This completes the proof of the first statement. The last statement follows from Lemma \ref{lemma:paths} \ref{lemma:paths:iii}.
\end{proof}

\begin{proof}[Proof of Theorem \ref{theorem:unions}]
We split the proof into several parts.

\smallskip
\noindent
\textbf{Initial setup.} 
Define $E=\bigcup_{i\in \N}E_i$ and let $\rho\in L^n_{\loc}(\R^n)$ be a non-negative Borel function. For each $i\in \N$ there exists an exceptional family $\Gamma_i$ with $\md_n\Gamma_i=0$, given by Theorem \ref{theorem:criterion_compact} \ref{cc:iv}. Also by  property \ref{m:positive_length} the family $\Gamma'$ of paths $\gamma$ with $\h^1(|\gamma|\cap E)>0$ has $n$-modulus zero. Define $\Gamma_0= \Gamma'\cup \bigcup_{i\in \N}\Gamma_i$. By the properties of $\Gamma_i$, the path family $\Gamma_0$ has the property that if $\{\eta_j\}_{j\in J}$ is a finite collection of paths outside $\Gamma_0$ and $\gamma$ is a path with $|\gamma|\subset \bigcup_{j\in J} |\eta_j|$, then $\gamma\notin \Gamma_0$. Moreover $|\gamma|\cap E$ is totally disconnected for all paths $\gamma\notin \Gamma_0$. 

Our goal is to show that Conclusion \ref{conclusion:b}$(E,\rho,\gamma)$ is true for each $\gamma\notin \Gamma_0$. By the last part of Theorem \ref{theorem:criterion_compact} and \ref{cc:v}, this suffices for $E$ to be $\*ned$.  We fix a path $\gamma\notin \Gamma_0$ with distinct endpoints, $\varepsilon>0$, and a neighborhood $U$ of $|\gamma|$. It suffices to show Conclusion \ref{conclusion:b} for a weak subpath of $\gamma$ with the same endpoints; also all weak subpaths of $\gamma$ lie outside $\Gamma_0$ by the properties of $\Gamma_0$. Thus, by replacing $\gamma$ with a simple weak subpath, we assume that $\gamma$ is simple and $\gamma\notin \Gamma_0$.

We introduce some notation. Define $\mathcal W_{\sigma}=\bigcup_{m=1}^\infty \N^m$ and $\mathcal W_0=\N$. We use the notation $\mathcal W_*$ for $\mathcal W_{\sigma}$ or $\mathcal W_0$, depending on whether we are working with $\CNED$ or $\NED$ sets, respectively.  Each element $w\in \N^m\subset \mathcal W_{\sigma}$ is called a \textit{word} of \textit{length} $l(w)=m$. For $w\in \mathcal W_0$ we also define $l(w)=w$. The \textit{empty word} $\emptyset$ has length $0$.   

\smallskip
\noindent
\textbf{Induction assumption.} 
Set $\widetilde \gamma_0= \alpha_{\emptyset}=\gamma$ and note that $\alpha_{\emptyset}\notin \Gamma_0$. Let $U_{\emptyset}$ be a neighborhood of $|\alpha_{\emptyset}|\setminus \partial \alpha_{\emptyset}$ with $\diam(U_{\emptyset})\leq 2\diam(|\alpha_{\emptyset}|)$ and $\br{U_{\emptyset}}\subset U$.  Suppose that for $m\geq 0$  we have defined simple paths $\widetilde \gamma_k$, $k\in \{0,\dots,m\}$, and collections of paths $\{\gamma_i\}_{i\in I_w}$ for $l(w)\leq m-1$ such that  
\begin{enumerate}[\upshape({A}-1)]\smallskip
	\item\label{ucc:i} $\widetilde \gamma_k\in \mathcal F_{*}( \bigcup_{i=1}^k E_i)$ for $k\in \{0,\dots,m\}$,\smallskip
	\item\label{ucc:ii}$\partial \widetilde \gamma_k=\partial \gamma$ and  $|\widetilde \gamma_k|\subset |\gamma|\cup \bigcup_{l(w)\leq k-1} \bigcup_{i\in I_w} |\gamma_i|$ for $k\in \{0,\dots,m\}$,\smallskip
	\item\label{ucc:iii} $\sum_{l(w)\leq m-1}\sum_{i\in I_w}\ell(\gamma_i)<\varepsilon$, and\smallskip
	\item\label{ucc:iv} $\sum_{l(w)\leq m-1}\sum_{i\in I_w}\int_{\gamma_i}\rho\, ds<\varepsilon$.	  
\end{enumerate}
Moreover, suppose we have defined the following objects: 
\begin{enumerate}[\upshape({A}-1)]\setcounter{enumi}{4}
	\item\label{ucc:v} Simple paths $\{\alpha_w\}_{l(w)\leq m}$ such that $\{\alpha_{w}\}_{l(w)=k}$ is a collection of parametri\-zations of the closures of the components of $|\widetilde \gamma_k|\setminus (E_1\cup \dots \cup E_{k})$, for each $k\in \{0,\dots,m\}$. In addition, for $l(w)\leq m$, each strict subpath of $\alpha_w$ lies outside $\Gamma_0$. We remark that in the case of $\mathcal W_*=\mathcal W_0$, there is only one component of $|\widetilde \gamma_k|\setminus (E_1\cup \dots \cup E_{k})$ and $\alpha_k=\widetilde \gamma_k$. 
	\item\label{ucc:vi}  Open sets $\{U_w\}_{l(w)\leq m}$ such that $U_w$ is a neighborhood of $|\alpha_w|\setminus \partial \alpha_w$ with $\diam(U_w)\leq 2\diam(|\alpha_w|)$ for $l(w)\leq m$ and the collection $\{U_w\}_{l(w)=k}$ is disjointed for each $k\in \{0,\dots,m\}$. Moreover,  if $m\geq 1$ and $l(w)\leq m-1$, then $U_{(w,j)}\subset U_w$ for $j\in \N$ in the case $\mathcal W_{*}=\mathcal W_{\sigma}$ and $U_{w+1}\subset U_{w}$ in the case $\mathcal W_*=\mathcal W_0$. Finally,   $\br{U_w}$ does not intersect $E_1\cup \dots\cup E_{l(w)}$, except possibly at the endpoints of $\alpha_w$, for $l(w)\leq m$.
\end{enumerate}
Finally, we require the compatibility property
\begin{enumerate}[\upshape({A}-1)]\setcounter{enumi}{6}
	\item\label{ucc:vii}$|\widetilde \gamma_m|\setminus \bigcup_{l(w)=k} U_w= |\widetilde \gamma_k| \setminus \bigcup_{l(w)=k} U_w$ for $k\in \{0,\dots,m\}$.
\end{enumerate}

\smallskip
\noindent
\textbf{Inductive step.} 
We now define $\widetilde \gamma_{m+1}$ as follows.
Fix $w\in \mathcal W_*$ with $l(w)=m$.  By \ref{ucc:v}   each strict subpath $\eta$ of $\alpha_w$ avoids the path family $\Gamma_0$; thus Conclusion \ref{conclusion:a}$(E_{m+1},\rho, \eta )$ is true and $\h^1(|\eta|\cap E_{m+1})=0$. The latter implies that $|\alpha_w|\cap E_{m+1}$ is totally disconnected. By Lemma \ref{lemma:iteration}, Conclusion \ref{conclusion:a}$(E_{m+1},\rho, \alpha_w )$ is true. Thus, for each $\delta_w>0$ there exists a simple path $\widetilde \alpha_w$ and paths $\gamma_i\notin \Gamma_0$, $i\in I_w$, such that 
	\begin{enumerate}[\upshape({A'}-1)]\smallskip
	\item\label{uc':i} $\widetilde \alpha_{w} \in \mathcal F_*(E_{m+1})$,\smallskip
	\item\label{uc':ii} $\partial \widetilde \alpha_{w}=\partial \alpha_{w}$ and $|\widetilde \alpha_{w}|\subset |\alpha_{w}|\cup \bigcup_{i\in I_{w}} |\gamma_{i}|$ and $\bigcup_{i\in I_{w}} |\gamma_{i}| \subset U_{w}$,\smallskip
	\item\label{uc':iii} $\sum_{i\in I_{w}}\ell(\gamma_{i})<\delta_w$, and \smallskip
	\item\label{uc':iv} $\sum_{i\in I_{w}} \int_{\gamma_{i}}\rho\, ds <\delta_w.$\smallskip
\end{enumerate}
If we choose a sufficiently small $\delta_w$, we can ensure that \ref{ucc:iii} and \ref{ucc:iv} are true for the index $m+1$ in place of $m$. By \ref{uc':ii}, the path $\widetilde \alpha_w$ is obtained by modifying $\alpha_w$ within $U_w$.  By \ref{ucc:vi}, $\br{U_w}$ does not intersect $E_1\cup \dots\cup E_m$, except possibly at the endpoints of $\alpha_w$. Therefore, $\widetilde \alpha_w\in \mathcal F_{*}(E_1\cup\dots\cup E_{m})$. Combining this with \ref{uc':i}, we obtain $\widetilde \alpha_w\in \mathcal F_{*}(E_1\cup\dots\cup E_{m+1})$. 

\begin{figure}
\centering
	\begin{overpic}[width=.7\linewidth]{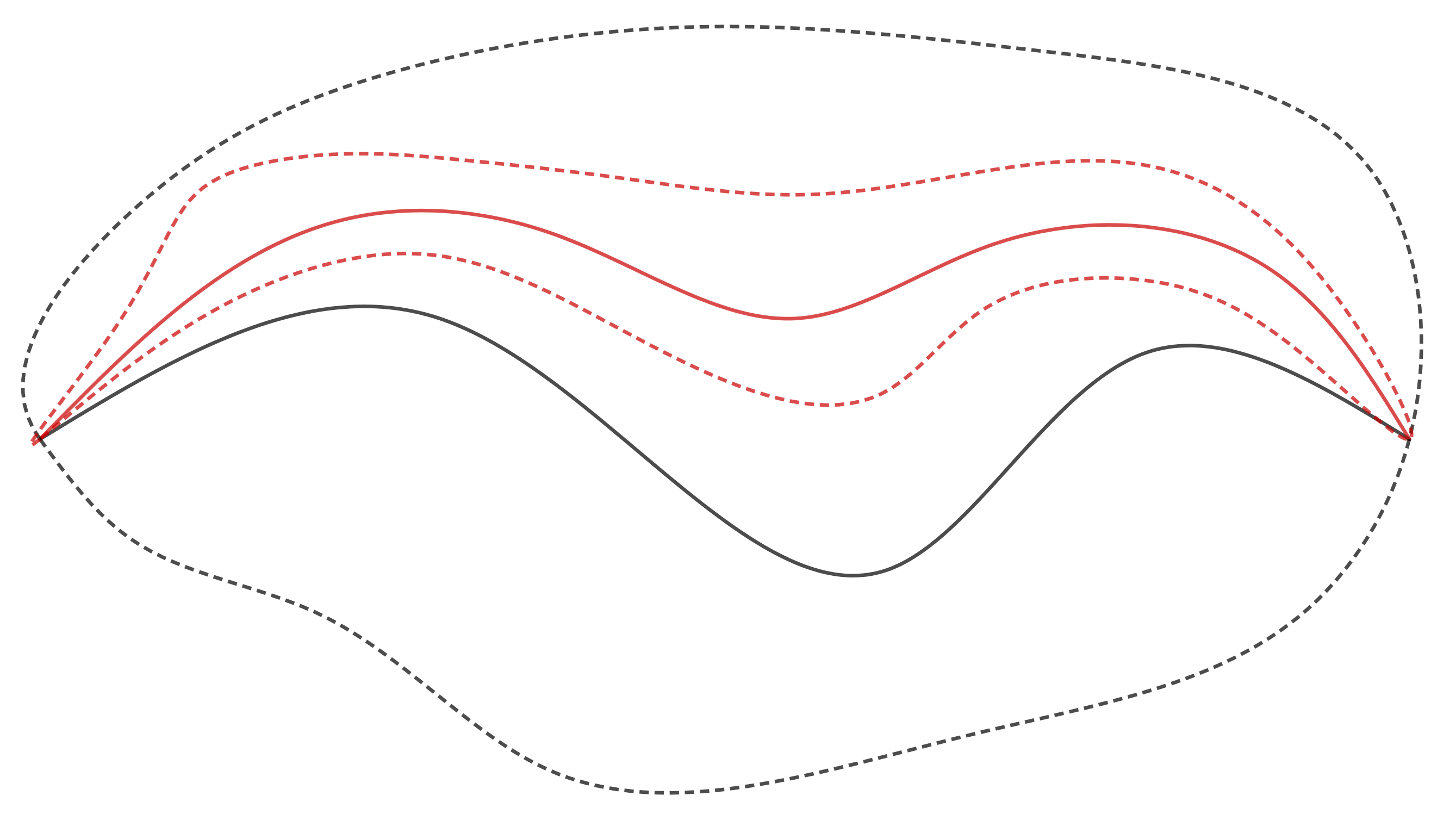}
		\put (2, 26.5) {$\bullet$}
		\put (96,26.5) {$\bullet$}
		\put (90,51) {$U_m$}
		\put (84, 45) {$U_{m+1}$}
		\put (51,38) {$\widetilde \gamma_{m+1}$}
		\put (55,20) {$\widetilde \gamma_m$}
	\end{overpic}
	\begin{overpic}[width=.99\linewidth]{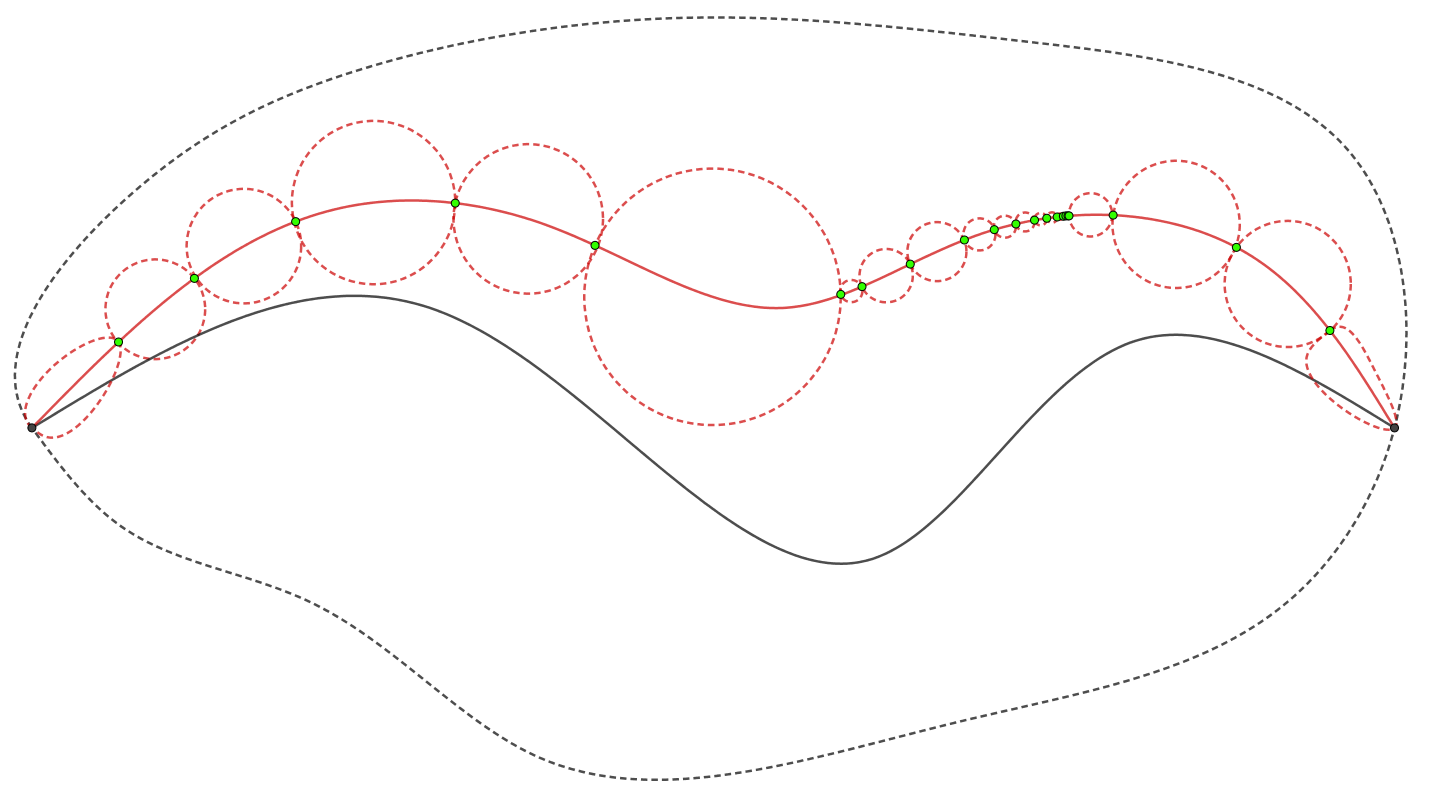}
		\put (90,51) {$U_w$}
		\put (50, 45) {$U_{(w,j)}$}
		\put (50,36) {$ \alpha_{(w,j)}$}
		\put (55,18) {$ \alpha_w$}
	\end{overpic}
	\caption{Construction of $\widetilde \gamma_{m+1}$ from $\widetilde \gamma_m$. Top figure: the case of $\mathcal W_0$. Bottom figure: the case of $\mathcal W_{\sigma}$. The red curve is $\widetilde \alpha_w$ and the green points denote the set $(|\widetilde \gamma_{m+1}|\cap E_{m+1})\cap U_w$, which is countable. In fact, a large part of $\widetilde \alpha_w$ should be shared with $\alpha_w$ by \ref{uc':ii} but we do not indicate this to simplify the figure.}\label{figure:union}
\end{figure}
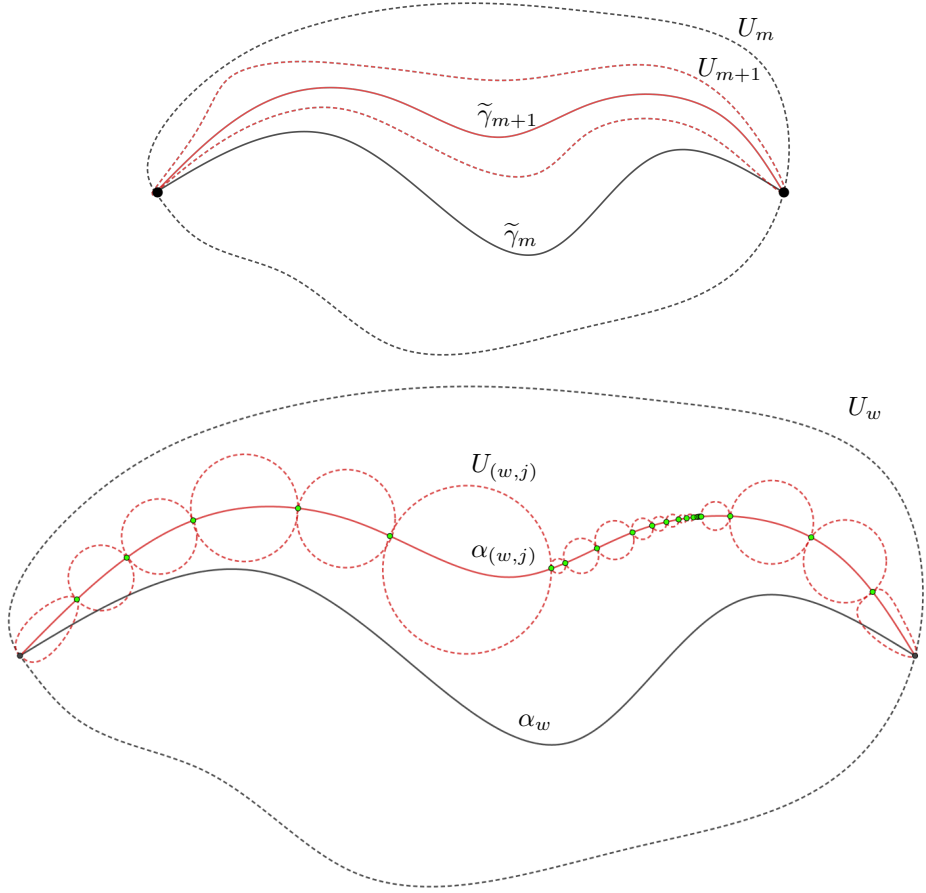

Using \ref{ucc:v}, we define $\widetilde \gamma_{m+1}$ by replacing each $\alpha_w$ in $\widetilde \gamma_m$, where $l(w)=m$, with $\widetilde \alpha_w$; see Figure \ref{figure:union}. In the case of $\mathcal W_{\sigma}$, we  need to ensure that this procedure gives a path. Indeed, by \ref{uc':ii} and \ref{uc':iii} we have $\ell(\widetilde \alpha_w)\leq \ell(\alpha_w)+\delta_w$, so if $\delta_w$ is sufficiently small, then we obtain a path $\widetilde \gamma_{m+1}$. By \ref{ucc:v}, the endpoints of $\widetilde \gamma_m$ do not lie in $|\alpha_w|\setminus \partial \alpha_w$ for any $w$ with $l(w)=m$, so they are not modified. Thus, $\widetilde \gamma_{m+1}$ has the same endpoints as $\gamma$. Moreover, by \ref{ucc:vi}, the regions $\{U_w\}_{l(w)=m}$, where the paths $\widetilde \alpha_w$ differ from $\alpha_w$ are pairwise disjoint. Hence $\widetilde \gamma_{m+1}$ is a simple path, since $\widetilde \gamma_m$ is simple. 

We now verify \ref{ucc:i} and \ref{ucc:ii}. By construction and \ref{ucc:v}, we have
$$|\widetilde \gamma_{m+1}|\setminus \bigcup_{l(w)=m}(|\widetilde \alpha_w|\setminus \partial \widetilde \alpha_w)= |\widetilde \gamma_m| \setminus \bigcup_{l(w)=m} (|\alpha_w|\setminus \partial \alpha_w) = \partial \gamma \cup  \left(|\widetilde \gamma_m|\cap \bigcup_{i=1}^m E_i\right).$$
In the case of $\mathcal W_{\sigma}$, since $\widetilde \gamma_m\in \mathcal F_{\sigma}(E_1\cup\dots\cup E_m)$ and $\widetilde \alpha_w\in \mathcal F_{\sigma}(E_1\cup\dots\cup E_{m+1})$ for $l(w)=m$,  
we conclude that 
$\widetilde \gamma_{m+1}\in \mathcal F_{\sigma}(E_1\cup\dots\cup E_{m+1})$, as required in \ref{ucc:i}. In the case $\mathcal W_*=\mathcal W_0$ we simply have $\widetilde \gamma_m=\alpha_m$ (see \ref{ucc:v}) and $\widetilde \gamma_{m+1}=\widetilde \alpha_m$, so $\widetilde \gamma_{m+1}\in \mathcal F_0(E_1\cup \dots\cup E_{m+1})$, as desired.  By construction and \ref{uc':ii} we have
\begin{align*}
|\widetilde \gamma_{m+1}| \subset |\widetilde \gamma_m|\cup \bigcup_{l(w)=m}\bigcup_{i\in I_w}|\gamma_i|.
\end{align*}
Thus, by the induction assumption we obtain \ref{ucc:ii} for the index $m+1$. 

Since $\widetilde \gamma_{m+1}$ is obtained by modifying $\widetilde \gamma_m$ in the open sets $U_w$, $l(w)=m$, we have $$|\widetilde \gamma_{m+1}|\setminus \bigcup_{l(w)=m}U_w =|\widetilde \gamma_{m}|\setminus \bigcup_{l(w)=m}U_w.$$
By \ref{ucc:vi}, for $k\leq m$ we have $\bigcup_{l(w)=k}U_w\supset \bigcup_{l(w)=m}U_w$, so 
$$|\widetilde \gamma_{m+1}|\setminus \bigcup_{l(w)=k}U_w= |\widetilde \gamma_m| \setminus \bigcup_{l(w)=k}U_w=  |\widetilde \gamma_k| \setminus \bigcup_{l(w)=k}U_w,$$
where the last equality follows from the induction assumption \ref{ucc:vii}.  This proves the equality in \ref{ucc:vii} for the index $m+1$. 

Next, we verify \ref{ucc:v}. Note that the paths $\{\widetilde \alpha_w\}_{l(w)=m}$, parametrize the closures of the components of $|\widetilde \gamma_{m+1}|\setminus (E_1\cup \dots \cup E_m)$; this follows from the construction and \ref{ucc:v}. For $l(w)=m$, let $\{\alpha_{(w,j)}\}_{j\in \N}$ be a collection of simple paths parametrizing the closures of the components of $|\widetilde \alpha_w|\setminus E_{m+1}$; see Figure \ref{figure:union}. Then $\{\alpha_w\}_{l(w)=m+1}$ gives a collection that parametrizes the closures of the components of $|\widetilde\gamma_{m+1}|\setminus (E_1\cup \dots \cup E_{m+1})$. This verifies the first part of \ref{ucc:v}. Moreover, if $\mathcal W_*=\mathcal W_{\sigma}$, for $l(w)=m$ and $j\in \N$, each strict subpath $\widetilde \eta$ of  $\alpha_{(w,j)}$ is a strict subpath of $\widetilde \alpha_w$. If $\mathcal W_*=\mathcal W_0$, each strict subpath $\widetilde \eta$ of  $\alpha_{w+1}$  is a strict subpath of $\widetilde \alpha_w$. In either case, by \ref{uc':ii}, there exists a strict subpath $\eta$ of $\alpha_w$ such that $|\widetilde \eta|\subset |\eta|\cup \bigcup_{i\in I_w}|\gamma_i|$. By the induction assumption \ref{ucc:v}, $\eta\notin \Gamma_0$. By the last part of Conclusion \ref{conclusion:a} (see the statement in Theorem \ref{theorem:criterion_compact}), $|\widetilde \eta|$ intersects finitely many of the traces $|\gamma_i|$, $i\in I_w$. Recall that $\gamma_i\notin \Gamma_0$, $i\in I_w$. The properties of the family $\Gamma_0$ imply that $\widetilde \eta\notin \Gamma_0$. This completes the proof of the second part of \ref{ucc:v}.

We now discuss \ref{ucc:vi}. If $\mathcal W_*=\mathcal W_0$, we define $U_{m+1}$ to be a neighborhood of $|\alpha_{m+1}|\setminus \alpha_{m+1}$ such that $\diam(U_{m+1})\leq 2\diam(|\alpha_{m+1}|)$, $U_{m+1}\subset U_m$, and $\br{U_{m+1}}$ does not intersect $E_1\cup \dots\cup E_{m+1}$, except possibly at the endpoints of $\alpha_{m+1}$; this uses that $E_1\cup \dots \cup E_{m+1}$ is closed and does not intersect $|\alpha_{m+1}|$, except possibly at the endpoints. If $\mathcal W_*=\mathcal W_{\sigma}$, for $l(w)=m$ we define $U_{(w,1)}$ to be a neighborhood of $|\alpha_{(w,1)}|\setminus  \partial \alpha_{(w,1)}$ such that $\diam(U_{(w,1)})\leq 2\diam(|\alpha_{(w,1)}|)$, $U_{(w,1)}\subset U_w$, and $\br {U_{(w,1)}}$ does not intersect $E_1\cup \dots \cup E_{m+1}$, except possibly at the endpoints of $\alpha_{(w,1)}$. Moreover, since $\alpha_{w}$ is simple, we may require that $\br{U_{(w,1)}}$ is disjoint from $|\alpha_{(w,j)}|\setminus \partial \alpha_{(w,j)}$ for $j>1$. Next, we define $U_{(w,2)}$ to  be a neighborhood of $|\alpha_{(w,2)}|\setminus \partial \alpha_{(w,2)}$ with the same properties as $U_{(w,1)}$ and with $U_{(w,1)}\cap U_{(w,2)}=\emptyset$.  Inductively, we define $U_{(w,j)}$ for all $j\in \N$ with the desired properties; see Figure \ref{figure:union}. We have completed the proof of the inductive step.

\smallskip
\noindent
\textbf{Completion of the proof.} 
Now we will show that Conclusion \ref{conclusion:b}$(E,\rho,\gamma)$ is true. For $m\in \N$ define $G_m=\bigcup_{k=0}^m|\widetilde \gamma_k|$. This is a continuum, since $|\widetilde \gamma_k|$ contains the endpoints of $\gamma$ for each $k\in \N$ by \ref{ucc:ii}. We have $G_{m}\subset G_{m+1}$ and 
\begin{align*}
G_m \subset |\gamma|\cup \bigcup_{l(w)\leq m-1} \bigcup_{i\in I_w} |\gamma_i|
\end{align*}
for $m\in \N$. 
By Lemma \ref{lemma:paths} \ref{lemma:paths:iv} and property \ref{ucc:iii}   we have
\begin{align*}
\h^1(G_m) &\leq \ell(\gamma) + \sum_{l(w)\leq m-1} \sum_{i\in I_w}\ell(\gamma_{i})\leq \ell(\gamma)+\varepsilon.
\end{align*} 
Thus $\sup_{m\in \N} \h^1(G_m)<\infty$, as required in Lemma \ref{lemma:hausdorf_continua}. We set $G=\bigcup_{m\in \N}G_m$.

Since $\widetilde \gamma_m$ is a simple path for each $m\in \N$, by Lemma \ref{lemma:paths} \ref{lemma:paths:iv}  we have
\begin{align*}
\ell(\widetilde \gamma_m)&= \h^1(|\widetilde \gamma_m|) \leq \h^1(G_m)\leq \ell(\gamma)+\varepsilon.
\end{align*} 
By the Arzel\`a--Ascoli theorem, there exists a subsequence of $\widetilde \gamma_m$, parametrized by arclength, that converges uniformly to a path $\widetilde \gamma$ with the same endpoints as $\gamma$ and 
$$\ell(\widetilde \gamma) \leq \liminf_{m\to\infty}\ell(\widetilde \gamma_m)\leq \ell(\gamma)+\varepsilon.$$
Hence, \ref{criterion:biii} holds.   Since $|\widetilde\gamma_m| \subset \br{U_{\emptyset}}$ for each $m\in \N$ by \ref{ucc:vi}, we have $|\widetilde \gamma|\subset  \br{U_{\emptyset}}\subset U$, as required in \ref{criterion:bii}. We assume that $\widetilde \gamma$ is simple by considering a weak subpath if necessary.  Since $|\widetilde \gamma|\subset \br G$, by Lemma \ref{lemma:hausdorf_continua} and Lemma \ref{lemma:paths} \ref{lemma:paths:iv} we have 
\begin{align*}
\int_{\widetilde \gamma}\rho\, ds &= \int_{\widetilde \gamma}\rho \x_G\, ds = \int_{|\widetilde \gamma|\cap G}\rho\, d\h^1 \leq \int_G \rho\, d\h^1\\
&\leq \int _{\gamma}\rho \, ds + \sum_{w\in \mathcal W_*}\sum_{i\in I_w} \int_{\gamma_{i}}\rho\, ds \leq \int_{\gamma}\rho\, ds + \varepsilon,
\end{align*}
where the last inequality follows from \ref{ucc:iv}. This shows \ref{criterion:biv}. 

Finally, we argue for \ref{criterion:bi}. By \ref{ucc:vii}  for $m\geq k\geq 0$ we have
\begin{align}\label{union:inclusion_mk}
|\widetilde \gamma_m| \subset |\widetilde \gamma_k| \cup  \br{\bigcup_{l(w)=k} U_{w}}.
\end{align}
By \ref{ucc:vi},  $\diam(U_w)\leq 2\diam(|\alpha_w|)$. If there are infinitely many non-empty sets $U_w$ with $l(w)=k$, only finitely many curves $\alpha_w$ can have diameter larger than a given number; indeed by \ref{ucc:v} these are subpaths of $\widetilde \gamma_k$ that have pairwise disjoint traces, except possibly at the endpoints. It follows that $U_w$ is contained in a small neighborhood of $|\widetilde \gamma_k|$ for all but finitely many $w$ with $l(w)=k$. Hence, we see that 
$$\br{\bigcup_{l(w)=k} U_{w}} \subset |\widetilde \gamma_k| \cup {\bigcup_{l(w)=k} \br{U_{w}}}.$$
By letting $m\to\infty$ in \eqref{union:inclusion_mk}, we conclude that 
$$|\widetilde \gamma|\subset |\widetilde \gamma_k| \cup  \br{\bigcup_{l(w)=k} U_{w}}\subset  |\widetilde \gamma_k| \cup \bigcup_{l(w)=k} \br{U_w}.$$
In the case of $\mathcal W_{\sigma}$, the set in the right-hand side intersects $E_1\cup \dots\cup E_k$ at countably many points by \ref{ucc:i} and \ref{ucc:vi}. Thus, $\widetilde \gamma\in \mathcal F_{\sigma}(E_1\cup\dots\cup E_k)$ for each $k\in \N$, so $\widetilde \gamma\in \mathcal F_\sigma(E)$. In the case of $\mathcal W_0$, we have
$|\widetilde \gamma|\subset |\widetilde \gamma_k|\cup \br {U_k} \subset \br{U_k}$ 
and the set $\br{U_k}$ does not intersect $E_1\cup \dots \cup E_k$, except possibly at the endpoints of $\alpha_k=\widetilde \gamma_k$. Thus, $|\widetilde \gamma|$ does not intersect $E_1\cup \dots \cup E_k$, except possibly at the endpoints, for each $k\in \N$. We conclude that $\widetilde \gamma \in \mathcal F_0(E)$.  We have completed the verification of Conclusion \ref{conclusion:b}$(E,\rho,\gamma)$, and thus, the proof of Theorem \ref{theorem:unions}. 
\end{proof}

\section{Examples of negligible sets}\label{section:examples}

Recall that all sets of $\sigma$-finite (resp.\ zero) Hausdorff $(n-1)$ measure are $\CNED$ (resp.\ $\NED$), by Theorem \ref{theorem:perturbation_hausdorff}. In this section we will discuss further examples.

\subsection{\texorpdfstring{A quasihyperbolic condition for $\CNED$ sets}{A quasihyperbolic condition for CNED sets}}\label{section:quasi}

Let $\Omega \subsetneq \R^n$ be a domain, i.e., a connected open set. For a point $x\in \Omega$, define $\delta_{\Omega}(x)= \dist(x,\partial \Omega)$. We define the \textit{quasihyperbolic distance} of two points $x_1,x_2\in \Omega$ by
\begin{align*}
k_{\Omega}(x_1,x_2)=\inf_{\gamma} \int_\gamma \frac{1}{\delta_{\Omega}}\,ds,
\end{align*}
where the infimum is taken over all rectifiable paths $\gamma$ in $\Omega$ that connect $x_1$ and $x_2$.  

\begin{theorem}\label{theorem:quasihyperbolic}
Let $\Omega\subset \R^n$ be a domain such that $k_{\Omega}(\cdot, x_0)\in L^n(\Omega)$ for some $x_0\in \Omega$. Then $\partial \Omega\in \CNED$. In particular, boundaries of John and H\"older domains are $\CNED$.
\end{theorem}
See \cite{SmithStegenga:HolderPoincare} for the definitions of the latter two classes of domains. The condition $k_{\Omega}(\cdot, x_0)\in L^n(\Omega)$ appeared in the work of Jones--Smirnov \cite{JonesSmirnov:removability}, who showed its sufficiency for $\partial \Omega$ to be $\QCH$-removable. The same condition has also been used in recent work of the current author \cite{Ntalampekos:removabilitydetour} to establish the removability of certain fractals with infinitely many complementary components for Sobolev spaces; in addition, it has appeared in work of the current author and Younsi \cite{NtalampekosYounsi:rigidity} in establishing the rigidity of circle domains under this condition. We will use some auxiliary results from \cite{NtalampekosYounsi:rigidity}, which have been proved there in dimension $2$, but the proofs apply to all dimensions.

\begin{remark}\label{remark:quasihyperbolic}
Domains satisfying the condition of the theorem are bounded \cite{NtalampekosYounsi:rigidity}*{Lemma 2.6} and thus have finite $n$-measure. Using this, one can show that an equivalent condition is $k_{\Omega}\in L^n(\Omega\times \Omega)$, so the base point $x_0$ is not of importance.
\end{remark}

We will prepare the necessary background before proving the theorem. 
For a domain $\Omega\subsetneq  \R^n$ we consider the \textit{Whitney cube decomposition} $\mathcal W(\Omega)$, which is a collection of closed dyadic cubes $Q\subset \Omega$, called \textit{Whitney cubes}, such that
\begin{enumerate}
\smallskip
\item the cubes of $\mathcal W(\Omega)$ have disjoint interiors and $\bigcup_{Q\in \mathcal W(\Omega)}Q=\Omega$,
\smallskip
\item $\diam(Q) \leq \dist(Q,\partial \Omega) \leq 4 \diam(Q)$ for all $Q\in \mathcal W(\Omega)$, and
\smallskip
\item if $Q_1\cap Q_2\neq \emptyset$, then $1/4 \leq \diam(Q_1)/\diam(Q_2)\leq 4$, for all $Q_1,Q_2\in \mathcal W(\Omega)$.\smallskip
\end{enumerate}
See \cite{Stein:Singular}*{Theorem 1 and Prop.\ 1, pp.~167--169} for the existence of the decomposition.  We denote by $\ell(Q)$ the side length of a cube $Q$; this is not to be confused with the length $\ell(\gamma)$ of a path $\gamma$. Two Whitney cubes $Q_1,Q_2\in\mathcal W(\Omega)$ with $\ell(Q_1)\geq \ell(Q_2)$ are \textit{adjacent} if a face of $Q_2$ is contained in a face of $Q_1$.

\begin{lemma}\label{quasi:adjacent}
Let $\Omega\subset \R^n$ be a domain and $Q_1,Q_2\in \mathcal W(\Omega)$ be adjacent cubes. Let $F_1\subset Q_1$ and $F_2\subset Q_2$ be continua with $\diam(F_i)\geq a\ell(Q_i)$ for some $a>0$, $i=1,2$, and let  $\rho\colon \R^n\to [0,\infty]$ be a Borel function. Then there exists a rectifiable path $\gamma\in \Gamma(F_1,F_2;\inter(Q_1\cup Q_2))$ such that
\begin{align*}
\int_{\gamma} \rho \, ds \leq c(n,a) \left( \|\rho \|_{L^n(Q_1)} +\|\rho\|_{L^n(Q_2)}  \right) \,\,\, \textrm{and}\,\,\, \ell(\gamma) \leq c(n,a)(\ell(Q_1)+\ell(Q_2)).
\end{align*}
\end{lemma}
\begin{proof}
By Lemma \ref{lemma:path_bound_modulus}, it suffices to show that $\md_n \Gamma(F_1,F_2;\inter(Q_1\cup Q_2))$ is uniformly bounded from below, depending only on $n$ and $a$.  This can be shown by mapping $Q_1\cup Q_2$ with a uniformly bi-Lipschitz map onto a ball. Euclidean balls are Loewner spaces; see \cite{Heinonen:metric}*{Chapter 8} for the definition and properties. Hence, the desired lower bound is satisfied.
\end{proof}

\begin{lemma}\label{quasi:chain}
Let $\Omega\subset \R^n$ be a domain, $\gamma\colon [0,1]\to \br\Omega$ be a path such that $\gamma((0,1))\subset \Omega$ and $\gamma(0),\gamma(1)\in \partial \Omega$, and $\rho\colon \R^n\to [0,\infty]$ be a Borel function. Then there exists a path $\gamma'\colon [0,1]\to \Omega$ with $ \gamma'(0)=\gamma(0)$, $\gamma'(1)=\gamma(1)$, $\gamma' ((0,1))\subset \bigcup_{\substack{Q\in \mathcal W(\Omega)\\ |\gamma|\cap Q\neq \emptyset}}Q$,
\begin{align*}
\int_{\gamma'}\rho\, ds \leq c(n) \sum_{\substack{Q\in \mathcal W(Q)\\ |\gamma'|\cap Q\neq \emptyset}} \|\rho\|_{L^n(Q)}\quad \textrm{and}\quad \ell( \gamma') \leq c(n) \sum_{\substack{Q\in \mathcal W(Q)\\ |\gamma'|\cap Q\neq \emptyset}} \ell(Q).
\end{align*} 
\end{lemma}
\begin{proof}
There exists a sequence $Q_i$, $i\in \Z$, of distinct Whitney cubes such that $Q_i$ is adjacent to $Q_{i+1}$ for each $i\in \Z$, $|\gamma|\cap Q_i\neq \emptyset$ for each $i\in \Z$, and $Q_i\to \gamma(0)$ as $i\to -\infty$ and $Q_i\to \gamma(1)$ as $i\to \infty$; see \cite{NtalampekosYounsi:rigidity}*{p.~143} for an argument. 

Consider the adjacent cubes $Q_0$ and $Q_1$. Let $F_1$ be the face of $Q_0$ that is opposite to the common face of $Q_0,Q_1$ and $F_2$ be the corresponding face of $Q_1$. We apply Lemma \ref{quasi:adjacent} to obtain a path $\gamma_{0}$ in $Q_0\cup Q_1$ connecting $F_1$ with $F_2$ and satisfying the conclusions of the lemma with $a=1/2$. We now consider $Q_1$ and $Q_2$. Let $F_1$ be a subcontinuum of $|\gamma_{0}|$ connecting opposite sides of $Q_1$ and let $F_2$ be the face of $Q_2$ opposite to the common face between $Q_1$ and $Q_2$. Applying Lemma \ref{quasi:adjacent} with $a=1/2$, we obtain a path $\gamma_{1}$ connecting $F_1$ with $F_2$ in $Q_1\cup Q_2$. Inductively, for each $i\in \Z$ we obtain a path $\gamma_{i}$ in $Q_i\cup Q_{i+1}$ as in the conclusions of Lemma \ref{quasi:adjacent} such that $|\gamma_i|\cap |\gamma_{i+1}|\neq \emptyset$. Since $\diam(Q_i)\to 0$ as $|i|\to \infty$, we can concatenate the paths $\gamma_i$ to obtain a path $\gamma'$ with the desired properties.
\end{proof}

Suppose that there exists a base point $x_0\in \Omega$ with $k_{\Omega}(\cdot,x_0)\in L^n(\Omega)$. It is shown in \cite{JonesSmirnov:removability}*{pp.~273--274} that there exists a tree-like family $\mathcal G$ of curves starting at $x_0$, connecting centers of adjacent Whitney cubes, and landing at $\partial \Omega$, that behave essentially like quasihyperbolic geodesics and so that each point of $\partial \Omega$ is the landing point of a curve of $\mathcal G$. For each cube $Q\in \mathcal W(\Omega)$ we define the \textit{shadow} $\SH(Q)$ of $Q$ to be the set of points $x\in \partial \Omega$ such that there exists a curve of $\mathcal G$ {starting at $x_0$}, passing through $Q$ and landing at $x$. We define
\begin{align*}
s(Q)= \diam( \SH(Q)).
\end{align*}
The set $\mathit{SH}(Q)$ is a compact subset of $\partial \Omega$ for each $Q\in \mathcal W(\Omega)$; see \cite{NtalampekosYounsi:rigidity}*{Lemma 2.7}. Moreover,  it is shown in \cite{JonesSmirnov:removability}*{p.~275} that
\begin{align}\label{quasi:shadow_quasihyperbolic}
\sum_{Q\in \mathcal W(\Omega)} s(Q)^n \lesssim_n \int_{\Omega} k(x,x_0)^n \,dx.
\end{align}

\begin{lemma}[\cite{NtalampekosYounsi:rigidity}*{Lemma 2.10}]\label{quasi:subpaths}
Let $\Omega\subset \R^n$ be a domain such that $k_{\Omega}(\cdot, x_0)\in L^n(\Omega)$ for some $x_0\in \Omega$. For each simple path $\gamma\colon [0,1] \to \R^n$ and $\varepsilon>0$ there exists a finite collection of paths $\{\gamma_i\colon [0,1] \to \br \Omega\}_{i\in I}$ such that
\begin{enumerate}[\upshape(i)]
	\item\label{quasi:i} $\partial \gamma_i\subset \partial \Omega$ and either $\gamma_i$ is constant or $\gamma_i((0,1))\subset \Omega$ for each $i\in I$,\smallskip
	\item\label{quasi:ii} there exists a path $\widetilde \gamma$ with $\partial \widetilde \gamma=\partial \gamma$ and $|\widetilde \gamma|\subset (|\gamma|\setminus \partial \Omega)\cup \bigcup_{i\in I} |\gamma_i|$,\smallskip
	\item\label{quasi:iii} if $Q\in \mathcal W(\Omega)$ is a Whitney cube with $|\gamma_i|\cap Q\neq \emptyset$ for some $i\in I$, then 
	$$|\gamma|\cap \SH(Q)\neq \emptyset \quad \textrm{and} \quad \ell(Q)\leq \varepsilon,$$
	\item\label{quasi:iv} $|\gamma_i|$ and $|\gamma_j|$ intersect disjoint sets of Whitney cubes $Q\in \mathcal W(\Omega)$ for $i\neq j$.
\end{enumerate}
Moreover,  if one replaces $\gamma_i$, for each $i\in I$, with a path $\gamma_i'\colon [0,1]\to \br \Omega$ such that $ \gamma_i'(0)=\gamma_i(0)$, $\gamma_i'(1)=\gamma_i(1)$, and either $\gamma_i'$ is constant or $\gamma_i' ((0,1))\subset \bigcup_{\substack{Q\in \mathcal W(\Omega)\\ |\gamma_i|\cap Q\neq \emptyset}}Q$, then conclusions \ref{quasi:i}--\ref{quasi:iv} also hold for the collection $\{\gamma_i'\}_{i\in I}$.
\end{lemma}

\begin{figure}
	\begin{overpic}[width=.9\linewidth]{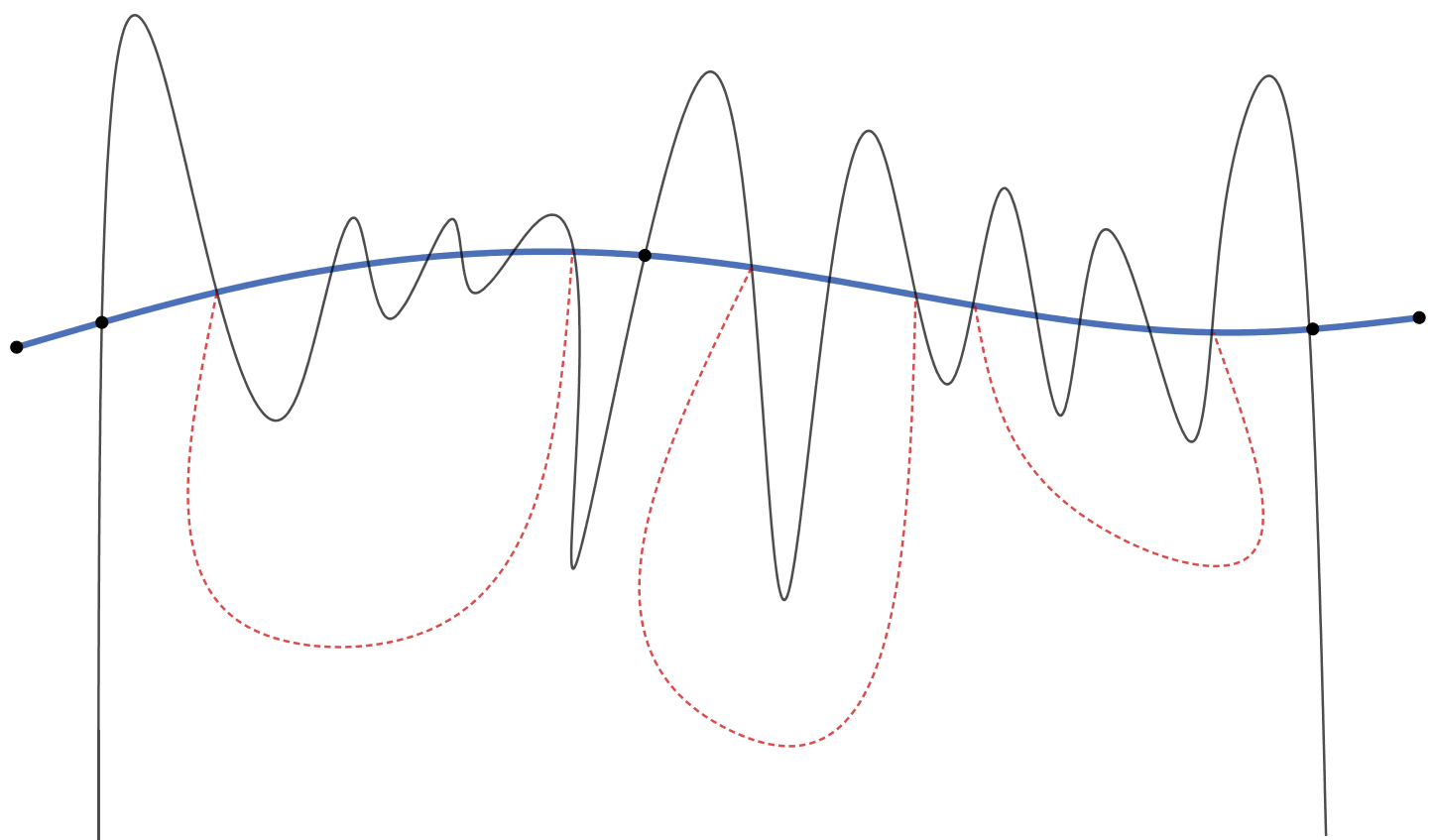}
		\put (-2,30) {$\gamma(0)$}
		\put (97, 32) {$\gamma(1)$}
		\put (4,37) {$\gamma_1$}
		\put (20,11) {$\gamma_2$}
		\put (42,42.5) {$\gamma_3$}
		\put (50, 5) {$\gamma_4$}
		\put (80, 17.5) {$\gamma_5$}
		\put (92,38) {$\gamma_6$}
		\put (80, 5) {$\Omega$}
		\put (93, 5) {$\partial \Omega$}
	\end{overpic}
	\caption{The path $\gamma$ (blue) and the paths $\gamma_i$ given by Lemma \ref{quasi:subpaths}. There are three constant paths $\gamma_i$. The path $\widetilde \gamma$ arises by replacing the arcs of $\gamma$ between the endpoints of $\gamma_i$ with $\gamma_i$.}\label{figure:detours}
\end{figure}
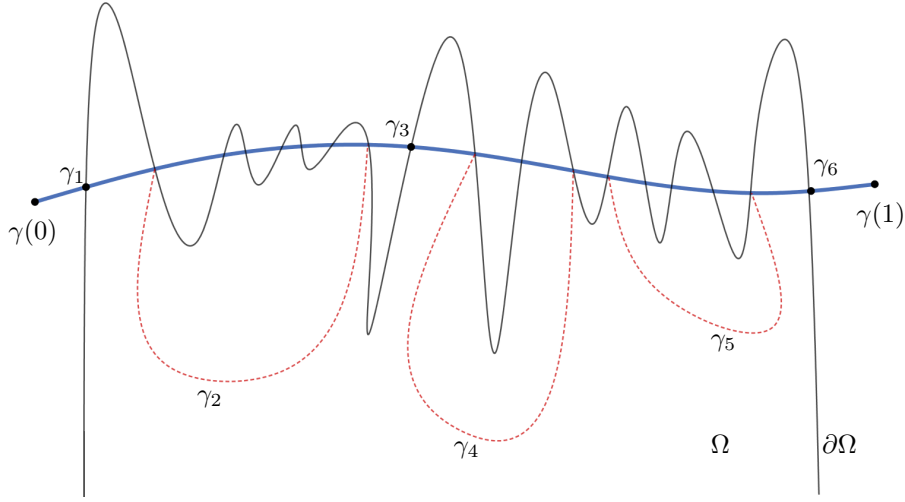

See Figure \ref{figure:detours} for an illustration. The formulation of this lemma in \cite{NtalampekosYounsi:rigidity}*{Lemma 2.10} is slightly different; however, the proof of \ref{quasi:i}, \ref{quasi:ii}, and \ref{quasi:iii} is identical. The paths $\gamma_i$ are concatenations of subpaths of paths of $\mathcal G$ and are replacing finitely many arcs of $\gamma$ that cover the set $|\gamma|\cap \partial \Omega$. Hence, using the path $\widetilde \gamma$ arising from these replacements one essentially avoids the set $|\gamma|\cap \partial \Omega$, except at the endpoints of $\gamma_i$.  We provide a sketch of the proof of \ref{quasi:iv}. One needs to concatenate appropriately the paths $\gamma_i$ that satisfy \ref{quasi:i}--\ref{quasi:iii}. If two paths $\gamma_i$, $\gamma_j$, $i\neq j$, meet a common Whitney cube $Q$, then one can concatenate these paths with a line segment inside $Q$. Then one considers a subpath ${\gamma_{ij}}$ of the concatenation so that \ref{quasi:i} and \ref{quasi:ii} are satisfied with $\gamma_{ij}$ in place of $\gamma_i$ and $\gamma_j$. After finitely many concatenations, one can obtain the family $\{\widetilde \gamma_i\}_{i\in \widetilde I}$ that satisfies all conditions \ref{quasi:i}--\ref{quasi:iv}.  The last part of the lemma provides some extra freedom in the choice of the paths $\gamma_i$; conclusion \ref{quasi:ii} for the collection $\{\gamma_i'\}_{i\in I}$ follows from the construction and the other conclusions are immediate since $\partial \gamma_i'=\partial \gamma_i$ and the trace of the path $\gamma_i'$ intersects no more Whitney cubes than $\gamma_i$ does.

\begin{proof}[Proof of Theorem \ref{theorem:quasihyperbolic}]
We will verify Theorem \ref{theorem:criterion_pfamily} \ref{cp:ii} to show that $\partial \Omega$ is $\CNED$. Let  $\rho\colon \R^n \to [0,\infty]$ be a Borel function with $\rho\in L^n(\R^n)$. We check condition \ref{cp:ii:i}. The proof of \ref{cp:ii:ii} is very similar and we omit it. Let $\gamma$ be a non-constant rectifiable path, and set $g= \rho+ \x_{\Omega}$. Since $\Omega$ is bounded (see Remark \ref{remark:quasihyperbolic}), we have $g\in L^n(\R^n)$. By Lemma \ref{lemma:p2} we have 
\begin{align*}
\int_{\R^n} \hspace{-1em}\sum_{\substack{\ell(Q)\leq \varepsilon \\ |\gamma+x|\cap \SH(Q)\neq \emptyset}} \hspace{-1em} \|g\|_{L^n(Q)}\, dx &= \sum_{\ell(Q)\leq \varepsilon} \|g\|_{L^n(Q)}\cdot  m_n( \{x:|\gamma+x|\cap \SH(Q)\neq \emptyset\}) \\
&\lesssim_n \sum_{\ell(Q)\leq \varepsilon} \|g\|_{L^n(Q)} \cdot \max \{\ell(\gamma),s(Q)\}s(Q)^{n-1}.
\end{align*}
Since $s(\cdot )\in \ell^n( \mathcal W(\Omega))$ by \eqref{quasi:shadow_quasihyperbolic}, we have $s(Q)\leq \ell(\gamma)$ if $\ell(Q)\leq \varepsilon$ and $\varepsilon$ is sufficiently small. Using this fact and H\"older's inequality, we bound the above sum by
\begin{align*}
\ell(\gamma) \bigg(\sum_{\ell(Q)\leq \varepsilon} \int_{Q} g^n \bigg)^{1/n} \cdot \bigg( \sum_{\ell(Q)\leq \varepsilon}  s(Q)^{n} \bigg)^{1-1/n}.
\end{align*} 
As $\varepsilon\to 0$, this converges to $0$. We conclude that as $\varepsilon \to 0$ along a sequence, for a.e.\ $x\in \R^n$ we have
\begin{align}\label{quasi:zeta}
\sum_{\substack{\ell(Q)\leq \varepsilon \\ |\gamma+x|\cap \SH(Q)\neq \emptyset}} \|\rho\|_{L^n(Q)} =o(1) \quad \textrm{and}\quad \sum_{\substack{\ell(Q)\leq \varepsilon \\ |\gamma+x|\cap \SH(Q)\neq \emptyset}} \ell(Q) =o(1)
\end{align}

Let $x\in \R^n$ such that \eqref{quasi:zeta} holds. Let $\eta$ be a strong subpath of $\gamma+x$ with distinct endpoints. We claim that Conclusion \ref{conclusion:b}$(\partial \Omega,\rho,\eta)$ is true. It suffices to prove this for a simple weak subpath of $\eta$ with the same endpoints that we still denote by $\eta$.  

For $\varepsilon>0$ we apply Lemma \ref{quasi:subpaths} to the path $\eta$ and obtain a finite collection of paths $\{\eta_i\}_{i\in I}$. To each non-constant path $\eta_i$ we apply Lemma \ref{quasi:chain} and we obtain a path $\eta_i'$ that has the same endpoints as $\eta_i$ and if $|\eta_i'|\cap Q\neq \emptyset$ for some $Q\in \mathcal W(\Omega)$, then $|\eta_i|\cap Q\neq \emptyset$. If $\eta_i$ is constant, we set $\eta_i'=\eta_i$. The last part of Lemma \ref{quasi:subpaths} allows us to replace each $\eta_i$ with $\eta_i'$ while retaining properties \ref{quasi:i}--\ref{quasi:iv}. By Lemma \ref{quasi:subpaths} \ref{quasi:i} and \ref{quasi:ii}, there exists a simple path $\widetilde \eta$ such that
\begin{enumerate}[\upshape(1)]
	\item $|\widetilde \eta|\cap \partial \Omega$ is a finite set (only the endpoints of $\eta_i'$ can lie in $\partial \Omega$), and\smallskip
	\item $\partial \eta=\partial \widetilde \eta$ and $|\widetilde \eta|\subset (|\eta|\setminus \partial \Omega)\cup \bigcup_{i\in I} |\eta_i'|$. \smallskip
\end{enumerate}
We  discard the paths $\eta_i'$ whose trace does not intersect $|\widetilde \eta|$. Furthermore, since the paths $\eta_i'$ satisfy the conclusions of Lemma \ref{quasi:chain}, in combination with Lemma \ref{quasi:subpaths} \ref{quasi:iii} and \ref{quasi:iv}, we have
\begin{enumerate}[\upshape(1)]\setcounter{enumi}{2}\smallskip
	\item $\displaystyle{\sum_{i\in I} \ell(\eta_i')\leq c(n) \sum_{i\in I} \sum_{\substack{\ell(Q)\leq \varepsilon \\ |\eta_i'|\cap Q\neq \emptyset}} \ell(Q) \leq  c(n)\sum_{\substack{\ell(Q)\leq \varepsilon \\ |\eta|\cap \SH(Q)\neq \emptyset}} \ell(Q) },\quad$ and 
	\item $\displaystyle{\sum_{i\in I} \int_{\eta_i'}\rho\, ds \leq c(n) \sum_{i\in I }\sum_{\substack{\ell(Q)\leq \varepsilon\\ |\eta_i'|\cap Q\neq \emptyset}} \|\rho\|_{L^n(Q)} \leq c(n)\sum_{\substack{\ell(Q)\leq \varepsilon \\ |\eta|\cap \SH(Q)\neq \emptyset}} \|\rho\|_{L^n(Q)}}$.
\end{enumerate}
Note that (1) implies \ref{criterion:bi}, and (2), (3), (4), together with \eqref{quasi:zeta}, imply \ref{criterion:biii} and \ref{criterion:biv}. Finally, given an open neighborhood $U$ of $|\eta|$, if $\varepsilon$ is sufficiently small, the sum of the lengths of $\eta_i'$ is small, and thus $\bigcup_{i\in I} |\eta_i'|\subset U$; this proves \ref{criterion:bii}. 
\end{proof}

\subsection{\texorpdfstring{Projections to axes and $\NED$ sets}{Projections to axes and NED sets}}

We present a result for sets whose projections to the coordinate axes have measure zero. This result will be crucial for the proof of Theorem \ref{example:ned}.

\begin{theorem}\label{theorem:projection}
Let $E\subset \R^2$ be a set whose projection to each coordinate direction has $1$-measure zero. Then $E\in \NED^w$. If, in addition, $m_2(\br E)=0$, then $E\in \NED$.  
\end{theorem}

This was proved for closed sets by Ahlfors--Beurling \cite{AhlforsBeurling:Nullsets}*{Theorem 10}. For sets that are not closed the proof is substantially more complicated. Note that $\br E$ is not $\NED$ in general even if it has measure zero. For example take $E= \Q \times \{0\}$, which is $\NED$ by Theorem \ref{theorem:zero}. However, its closure is $\br E=\R\times \{0\}$, which has measure zero, but it is not $\NED$ since it is not totally disconnected. 

We prove Theorem \ref{theorem:projection} only in dimension $2$, because of the following lemma, whose statement is very similar to Lemma \ref{lemma:loewner}. Except for that lemma, none of the arguments in the proof of Theorem \ref{theorem:projection} depend on the dimension.

\begin{lemma}\label{lemma:loewner_proj}
Let $E\subset \R^2$ be a set whose projection to each coordinate direction has $1$-measure zero. For $t>0$, denote by $Q_t$ the open square centered at the origin with side length $t$ and sides parallel to the coordinate axes. Let $0<r<R$ and suppose that $F_1,F_2\subset \R^2$ are disjoint continua such that $\partial Q_t$ intersects both $F_1$ and $F_2$ for every $r<t<R$. Then 
\begin{align*}
\md_n (\Gamma(F_1,F_2; Q_R\setminus \br{Q_r}) \cap \mathcal F_0(E))\geq \frac{1}{4} \log\left(\frac{R}{r}\right).
\end{align*}
\end{lemma}
\begin{proof}
We have $\partial Q_t\cap E=\emptyset$ for a.e.\ $t\in (r,R)$. Let $\rho$ be an admissible function for $\Gamma(F_1,F_2; Q_R\setminus \br{Q_r}) \cap \mathcal F_0(E)$. Then 
\begin{align*}
	 1\leq \int_{\partial Q_t} \rho\, ds\leq \left(\int_{\partial Q_t} \rho^2\, ds \right)^{1/2} \sqrt{4t}
\end{align*} 
for a.e.\ $t\in (r,R)$. By integration and Fubini's theorem, we have
\begin{align*}
\frac{1}{4}\log \left(\frac{R}{r}\right)=\int_{r}^R\frac{1}{4t}\, dt\leq  \int_{Q_R\setminus \br{Q_r}} \rho^2
\end{align*}
Infimizing over $\rho$ gives the conclusion.
\end{proof}

The proof of the following lemma is exactly the same as the proof of Lemma \ref{lemma:fattening}, where one uses Lemma \ref{lemma:loewner_proj} in place of Lemma \ref{lemma:loewner}; see Remark \ref{remark:fattening}. 

\begin{lemma}\label{lemma:fattening_proj}
Let $E\subset \R^2$ be set whose projection to each coordinate direction has $1$-measure zero. Then for every open set $U\subset \R^2$ and for every pair of non-empty, disjoint continua $F_1,F_2 \subset  U$ we have
\begin{align*}
\md_2(\Gamma(F_1,F_2;U) \cap \mathcal F_0(E))=\lim_{r\to 0} \md_2 (\Gamma(F_1^r,F_2^r;U) \cap \mathcal F_0(E)).
\end{align*}
\end{lemma}

\begin{proof}[Proof of Theorem \ref{theorem:projection}]
Let $F_1,F_2\subset \R^2$ be non-empty, disjoint continua.  We fix a small $r>0$ so that $F_1^r\cap F_2^r=\emptyset$ and let $\rho\colon \R^2\to [0,\infty]$ be a Borel function with $\rho\in L^2(\R^2)$ that is admissible for $\Gamma(F_1^r,F_2^r;\R^2)\cap \mathcal F_0(E)$. 

Consider a sequence of open sets $\{V_m\}_{m\in \N}$, such that $E\subset V_{m+1}\subset V_m\subset N_{1/m}(E)$ and such that the projection of $V_m$ to each coordinate axis has measure less than $1/m$ for each $m\in \N$. Observe that $\bigcap_{m=1}^\infty V_m$ has $2$-measure zero. For each $m\in \N$ define the closed set $X_m= (\R^2\setminus V_m)\cup F_1^r\cup F_2^r$. Note that a.e.\ line parallel to a coordinate direction does not intersect the set $\bigcap_{m=1}^\infty V_m$. Hence,   a.e.\ line parallel to a coordinate direction   lies in $\bigcup_{m=1}^\infty X_m$. Moreover, if $\gamma$ is a rectifiable path in $X_m$ joining $F_1^r$ to $F_2^r$, then $\gamma$ has a subpath in $\R^2\setminus V_{m}\subset \R^2\setminus E$ joining $F_1^r$ to $F_2^r$. By the admissibility of $\rho$, we obtain
\begin{align*}
\int_{\gamma}\rho\, ds \geq1.
\end{align*}

For each $m\in \N$, we define on $X_m$ the function 
\begin{align*}
g_m(x)= \min \left\{ \inf_{\gamma_x}\int_{\gamma_x} \rho\, ds, 1\right\}
\end{align*}
where the infimum is taken over all rectifiable paths $\gamma_x$ in $X_m$ that connect $F_1^r$ to $x$. By \cite{JarvenpaaEtAl:measurability}*{Corollary 1.10}, the function $g_m$ is measurable; the fact that $X_m$ is closed, thus complete, is important here. One can alternatively argue using \cite{HeinonenKoskelaShanmugalingamTyson:Sobolev}*{Lemma 7.2.13, p.~187}. Moreover, we have $0\leq g_m\leq 1$, $g_m=0$ on $F_1^r$, and $g_m=1$ on $F_2^r$.   Next, we show that $\rho$ is an \textit{upper gradient} of $g_m$. That is,
\begin{align*}
|g_m(y)-g_m(x)|\leq \int_{\gamma} \rho\, ds
\end{align*} 
for every rectifiable path $\gamma \colon [0,1]\to  X_m$ with $\gamma(0)=x$ and $\gamma(1)=y$. Since the roles of $x$ and $y$ are symmetric, we will only show that 
$$g_m(y)-g_m(x) \leq \int_{\gamma}\rho\, ds.$$
If $g_m(y)=1$, then this inequality is immediate, since then $g_m(y)-g_m(x)\leq 0$. Suppose $g_m(y)= \inf_{\gamma_y}\int_{\gamma_y} \rho\, ds$. We fix a curve $\gamma_x$ joining $F_1^r$ to $x$. Define a curve $\gamma_y$ by concatenating $\gamma_x$ with $\gamma$. Then 
$$g_m(y) \leq \int_{\gamma_y}\rho\, ds= \int_{\gamma}\rho\, ds + \int_{\gamma_x}\rho\, ds.$$
Infimizing over $\gamma_x$ gives the desired inequality. 

The sequence of sets $\{X_k\}_{k\in \N}$ is increasing. Thus, if $k\geq m$, then $g_k$ is defined by infimizing over a larger collection of paths compared to the definition of $g_m$. It follows that $0\leq g_k\leq g_m$ in $X_m$. Therefore, for each $m\in \N$, $g_k$ converges pointwise as $k\to \infty$ to a measurable function $g$ in $X_m$. Moreover, by the pointwise convergence, for each $m\in \N$ the function $\rho$ is an upper gradient of $g$ in $X_m$, $0\leq g\leq 1$ in $\bigcup_{m=1}^\infty X_m$, $g=0$ in a neighborhood of $F_1$, and $g=1$ in a neighborhood of $F_2$. On $\R^2\setminus \bigcup_{m=1}^\infty X_m \subset \bigcap_{m=1}^\infty V_m$ we define $g=0$. Thus, we have extended $g$ to a measurable function in $\R^2$. 
 
We claim that $g$ is absolutely continuous in a.e.\ line segment parallel to a coordinate direction. Let $L$ be a line segment parallel to $e_1=(1,0)$.  By construction, for $\h^1$-a.e.\ $z\in \{e_1\}^{\perp}$  the line segment $L+z$ lies in $ X_m$ for some $m\in \N$. Moreover, since $\rho\in L^2(\R^2)$, we have $\int_{L+z} \rho\, ds<\infty$ for a.e.\ $z\in \{e_1\}^{\perp}$.  By the upper gradient inequality we conclude that $g$ is absolutely continuous in $L+z$ for a.e.\ $z\in \{e_1\}^{\perp}$ and $|g_x|\leq \rho$ almost everywhere on $L+z$. This implies that $|g_x|\leq \rho$ a.e. Similarly, $|g_y|\leq \rho$ a.e.  Thus, $g$ lies in the classical Sobolev space $W^{1,2}_{\loc}(\R^2)$ and $|\nabla g|\leq \sqrt{2}\rho$ a.e. in $\R^2$; see \cite{Ziemer:Sobolev}*{Theorem 2.1.4, p.~44}.

If $\R^2 \setminus \br E\neq \emptyset$, then each $x\in \R^2\setminus \br E$ has a bounded open neighborhood $Y$ that is disjoint from $V_m$ for sufficiently large $m$, since $V_m\subset N_{1/m}(\br E)$. Thus, $Y\subset X_m$ and $g\in W^{1,2}(Y)$. Since $\rho$ is an upper gradient of $g$ in $Y$, by \cite{Hajlasz:Sobolev}*{Corollary 7.15} we conclude that $|\nabla g|\leq \rho$ a.e.\ in $Y$. Therefore, $|\nabla g|\leq \rho$ a.e.\ in $\R^2\setminus \br E$. Summarizing, at a.e.\ point of $\R^2$ we have
$$|\nabla g| \leq \rho \x_{\R^2\setminus \br E} + \sqrt{2}\rho\x_{\br E}.$$

Since $g=0$ in a neighborhood of $F_1$ and $g=1$ in a neighborhood of $F_2$, for each $\varepsilon>0$ there exists (by mollification) a smooth function $g_{\varepsilon}$ on $\R^2$ with the same properties and with $\|\nabla g_{\varepsilon}\|_{L^2(\R^2)}^2<\|\nabla g\|_{L^2(\R^2)}^2+\varepsilon$; see \cite{Ziemer:Sobolev}*{Lemma 2.1.3, p.~43}. It is immediate that $|\nabla g_{\varepsilon}|$ is admissible for $\Gamma(F_1,F_2;\R^2)$. Thus, 
\begin{align*}
\md_2\Gamma(F_1,F_2;\R^2)\leq \|\nabla g_{\varepsilon}\|_{L^2(\R^2)}^2 \leq \|\rho\|_{L^2(\R^2\setminus \br E)}^2+2\|\rho\|_{L^2(\br E)}^2+\varepsilon.
\end{align*}
First we let $\varepsilon\to 0$, and then we infimize over $\rho$ to obtain
\begin{align*}
\md_2\Gamma(F_1,F_2;\R^2) \leq c_0\md_2(\Gamma(F_1^r,F_2^r;\R^2)\cap \mathcal F_0(E)),
\end{align*}
where $c_0=1$ if $m_2(\br E)=0$ and $c_0=2$ otherwise. Now, we let $r\to 0$ and by Lemma \ref{lemma:fattening_proj} we obtain
\[\md_2 \Gamma(F_1,F_2;\R^2) \leq c_0\md_2(\Gamma(F_1,F_2;\R^2)\cap \mathcal F_0(E)).\qedhere\]
\end{proof}

\subsection{\texorpdfstring{A non-measurable $\CNED$ set}{A non-measurable CNED set}}\label{section:nonmeasurable}

Sierpi\'nski \cite{Sierpinski:nonmeasurable}, using the axiom of choice but not the continuum hypothesis, constructed a striking example of a non-mea\-sur\-able set $E\subset \R^2$ such that every line intersects $E$ in at most two points. This example served at that time as a counterexample to the converse of Fubini's theorem: if the slices of a planar set are measurable, then is the whole set measurable? We show here that Sierpi\'nski's set is $\CNED$.  Thus, the assumption of measurability in Lemma \ref{lemma:measure_zero} is necessary in order to derive that $\CNED$ sets have measure zero. 

\begin{prop}\label{prop:nonmeasurable}
There exists a non-measurable set $E\subset \R^2$ that is $ \CNED$. 
\end{prop}
\begin{proof}
Let $E$ be the non-measurable set of Sierpi\'nski. Let  $F_1,F_2\subset \R^2$ be non-empty, disjoint continua. We will show that
\begin{align*}
\md_2\Gamma(F_1,F_2;\R^2)\leq  \md_2(\Gamma(F_1,F_2;\R^2)\cap \mathcal F_{\sigma}(E)).
\end{align*}
According to a remarkable result of Aseev \cite{Aseev:nedhyperplane}*{Theorem 2.1}, we have 
$$\md_2\Gamma(F_1,F_2;\R^2)= \md_2(\Gamma(F_1,F_2;\R^2)\cap \mathcal F),$$
where $\mathcal F$ is the family of \textit{piecewise linear curves with respect to} $\R^2\setminus (F_1\cup F_2)$; that is, $\gamma \in \mathcal F$ if each point of $|\gamma|\setminus (F_1\cup F_2)$ has a neighborhood $V$ such that $|\gamma|\cap V$ consists of finitely many straight line segments. Hence, it suffices to show that 
$$\Gamma(F_1,F_2;\R^2)\cap \mathcal F\subset \Gamma(F_1,F_2;\R^2)\cap \mathcal F_{\sigma}(E).$$
Let $\gamma\in \Gamma(F_1,F_2;\R^2)\cap \mathcal F$. By the properties of the set $E$, and since $\gamma$ is piecewise linear, the set $|\gamma|\cap E$ is countable. Hence, $\gamma\in \mathcal F_{\sigma}(E)$, as desired.
\end{proof}

\section{Examples of non-negligible sets}\label{section:nonexamples}

\begin{proof}[Proof of Proposition \ref{proposition:packings}]
Let $E$ be the residual set of a packing as in the statement. Let $\Gamma_0$ denote the family of non-constant curves that intersect the set
$$S=\bigcup_{\substack{i,j\in \N\\  i\neq j}}(\partial D_i\cap \partial D_j).$$ 
Since $S$ is countable, we have $\md_n\Gamma_0=0$.  Consider continua $F_1,F_2$, contained in $D_1,D_2$, respectively. The claim that $E\notin \CNED$ follows once we establish that
\begin{align}\label{prop:packing:claim}
\Gamma(F_1,F_2;D_0) \cap \mathcal F_{\sigma}(E)\subset \Gamma_0.
\end{align} 

Let $\gamma\in \Gamma(F_1,F_2;D_0) \setminus \Gamma_0$. By considering a weak subpath, we assume that $\gamma$ is a simple path with the same properties. We consider an injective parametrization $\gamma\colon [0,1]\to D_0$. Let $A$ be the set of $t\in [0,1]$ such that there exists $\delta>0$ and $i\in \N$ with the property that $\gamma( (t-\delta,t)\cup (t,t+\delta))\subset D_i$ and $\gamma(t)\in \partial D_i$; that is, $\gamma$ ``bounces" on $\partial D_i$ at time $t$. Then $A$ is countable and relatively open in $\gamma^{-1}(E)$. 

We claim that the compact set $B=\gamma^{-1}(E)\setminus A$ is non-empty and perfect, in which case, it is uncountable. Therefore, $|\gamma|\cap E$ is uncountable and $\gamma \notin \mathcal F_{\sigma}(E)$, which completes the proof of \eqref{prop:packing:claim}. To see that $B$ is non-empty, let $t=\sup\{a\in [0,1]: \gamma(a)\in D_1\}$. Then $\gamma(t)\in \partial D_1$, so $t\in \gamma^{-1}(E)$, $\gamma((t-
\delta,t))\cap {D_1}\neq \emptyset$ for every $\delta>0$, and $\gamma((t,1])\cap {D_1}=\emptyset$. Thus, $t\notin A$ and $t\in B$, showing that $B\neq \emptyset$. We now show perfectness. Let  $t\in B$ and let $I\subset [0,1]$ be an open interval containing $t$. 

\smallskip
\noindent
\textit{Case 1:} Suppose that $\gamma(I)\subset E$. Since $A$ is countable, there exists $s\in I\setminus A$, $s\neq t$, with $\gamma(s)\in E$. Hence, $(I\setminus \{t\})\cap B\neq \emptyset$. 

\smallskip
\noindent
\textit{Case 2:} Suppose that $\gamma(I)\cap D_{i_0}\neq\emptyset$ for some $i_0\in \N$ and  $\gamma(t)\notin \partial D_{i_0}$. Without loss of generality $\gamma(s)\in D_{i_0}$ for some $s\in I$ with $s<t$. Let $s_1=\sup\{a\in (s,t): \gamma(a)\in D_{i_0}\}$. Then $\gamma(s_1)\in \partial D_{i_0}$, $s_1\neq t$, $\gamma((s_1-\delta,s_1))\cap D_{i_0}\neq \emptyset$ for every $\delta>0$, and $\gamma((s_1,t))\cap D_{i_0}=\emptyset$. Thus, $s_1\notin A$ and we have $s_1\in I\cap B$, so $(I\setminus \{t\})\cap B\neq \emptyset$. 

\smallskip
\noindent
\textit{Case 3:} Suppose $\gamma(t)\in \partial D_{i_0}$ for some $i_0\in \N$. Since $t\notin A$, there exists an open subinterval of $I$, say $J=(s,t)$, such that $\gamma(J)$ intersects the complement of $D_{i_0}$. If $\gamma(J)\subset E$, then by Case 1 we have $(J\setminus  \{ t \} )\cap B\neq \emptyset$. Suppose that $\gamma(J)\cap D_{i_1}\neq\emptyset$ for some $i_1\neq i_0$. Then $\gamma(t)\notin \partial D_{i_1}$, since $\gamma\notin \Gamma_0$ and $\gamma$ avoids the set $\partial D_{i_1}\cap \partial D_{i_0}$. We are now reduced to Case 2 with $i_1$ in place of $i_0$ and $J$ in place of $I$.
\end{proof}

\begin{remark}\label{remark:prop:packings}
One can relax the assumption of the proposition to requiring that $\partial D_i\cap \partial D_j$, $i\neq j$, has Sobolev $n$-capacity zero (see \cite{HeinonenKoskelaShanmugalingamTyson:Sobolev}*{Section 7.2}). Then the family of curves passing through $\partial D_i\cap \partial D_j$ has $n$-modulus zero.   
\end{remark}

Next, we establish a preliminary elementary result before proving Theorem \ref{example:ned}.

\begin{lemma}\label{lemma:open_compact}
Let $U\subset \R$ be an open set and $E\subset \R$ be a compact set with $m_1(E)>0$. Then there exists a sequence of similarities $\tau_i\colon \R \to \R$, $i\in \N$, and a set $N\subset \R$ with $m_1(N)=0$ such that 
\begin{align*}
U= N\cup \bigcup_{i\in \N} \tau_i(E).
\end{align*}
\end{lemma}
\begin{proof}
Suppose that $E\subset (a,b)$ and set $\lambda=m_1(E)/(b-a)$. Moreover, suppose that $U$ is bounded. Let $U_0=U$ and  let $I_{0,i}$, $i\in \N$, be the connected components of $U_0$, which are bounded open intervals. For each $i\in \N$, define $\tau_{0,i}$ to be the similarity that maps $(a,b)$ onto $I_{0,i}$. Then $m_1( \tau_{0,i}(E))/m_1(I_{0,i})=\lambda$ and 
\begin{align*}
m_1\left( \bigcup_{i\in \N} \tau_{0,i}(E)\right)= \lambda m_1(U_0).
\end{align*} 
We now define $U_1= U_0\setminus \bigcup_{i\in \N} \tau_{0,i}(E)$, which is open, and note that $m_1(U_1)=(1-\lambda)m_1(U_0)$. We proceed in the same way to obtain similarities $\tau_{1,i}$, $i\in \N$, that map $(a,b)$ to the connected components of $U_1$. In the $k$-th step, we obtain the set 
$$U_k=U_0\setminus \bigcup_{j=0}^{k-1}\bigcup_{i\in \N} \tau_{j,i}(E)$$
with $m_n(U_k)=(1-\lambda)^k m_1(U_0)$. Thus, $N=U_0\setminus \bigcup_{j=0}^{\infty}\bigcup_{i\in \N} \tau_{j,i}(E)$ is a null set, as desired. If $U$ is unbounded, we can simply write it as a countable union of bounded open sets and apply the previous result to each of them.
\end{proof}


\begin{proof}[Proof of Theorem \ref{example:ned}]
According to a construction of Wu \cite{Wu:cantor}*{Example 2}, there exist Cantor sets $G,F\subset \R$ such that $m_1(G)=0$, $m_1(F)>0$, and $G\times F$ is removable for the Sobolev space $W^{1,2}$. Thus, $G\times F$ is of class $\NED\subset \CNED$; this follows from \cite{VodopjanovGoldstein:removable}.

By Lemma \ref{lemma:open_compact}, there exist countably many scaled and translated copies $F_i\subset (0,1)$, $i\in \N$, of $F$, such that the set $E_1=[0,1]\setminus \bigcup_{i\in \N}F_i$ has $1$-measure zero.  We let $E_2=[0,1]\setminus E_1=\bigcup_{i\in \N}F_i$. We have
\begin{align*}
G\times [0,1] =G\times (E_1\cup E_2) = (G\times E_1) \cup (G\times E_2).
\end{align*}
The set  $G\times [0,1]$ is  not $\QCH$-removable (recall the discussion in the Introduction), so it is not $\CNED$ by Theorem \ref{theorem:removable}; this can also be proved directly. 

On the other hand, for each $i\in \N$ the set $G\times F_i$ is the quasiconformal image of $G\times F$, which is $\NED$. Compact $\NED$ sets are invariant under quasiconformal maps by Corollary \ref{corollary:qc_invariance}  (this also follows from \cite{AhlforsBeurling:Nullsets}*{Theorem 4}). Thus, $G\times F_i$ is  $\NED$ for each $i\in \N$. By Theorem \ref{theorem:unions} we conclude that
\begin{align*}
G\times E_2= \bigcup_{i\in \N} G\times F_i \in \NED.
\end{align*}
Finally, note that the projections of the set $G\times E_1$ to the coordinate axes have measure zero. Moreover, $\br {G\times E_1} \subset G\times [0,1]$, and the latter has $2$-measure zero. By Theorem \ref{theorem:projection} we conclude that $G\times E_1\in \NED$.
\end{proof}

\bibliography{biblio}
\end{document}